\pdfoutput=1
\nonstopmode

\documentclass{amsart}
\usepackage[hyperindex=true,
bookmarks=true,bookmarksnumbered=true]{hyperref}

\usepackage{amsmath, amssymb, amsthm, amscd}
\usepackage{mathrsfs}
\usepackage{tensor}
\usepackage{empheq}                                                                                                                      
\usepackage{enumitem}
        \setlist[enumerate,1]{before=\setcounter{enumi}{\value{equation}},after=\setcounter{equation}{\value{enumi}}}

\newcommand*\widefbox[1]{\fbox{\hspace{1em}#1\hspace{1em}}}

\usepackage{verbatim}
\usepackage[textsize=small]{todonotes}
\newcounter{todocounter}

\usepackage[all]{xy}
\newdir{ >}{{}*!/-10pt/@{>}}
\newdir^{ (}{{}*!/-5pt/@^{(}}
\def\undersetbrace#1\to#2{\underbrace{#2}_{#1}}                                                          
\def\oversetbrace#1\to#2{\overbrace{#2}^{#1}}
\def\AMSunderset#1\to#2{\underset{#1}{#2}}
\def\AMSoverset#1\to#2{\overset{#1}{#2}}

\def\East#1#2{\overset{#1}{\longrightarrow}}

\swapnumbers
\newtheorem{proposition}[subsection]{Proposition}
\newtheorem*{proposition*}{Proposition}
\newtheorem{theorem}[subsection]{Theorem}
\newtheorem*{theorem*}{Theorem}
\newtheorem{lemma}[subsection]{Lemma}
\newtheorem*{lemma*}{Lemma}
\newtheorem{claim}[subsubsection]{Claim}

\newtheorem*{corollary*}{Corollary}
\theoremstyle{definition}
\newtheorem*{remark*}{Remark}

\newtheorem*{example*}{Example}
\newtheorem{example}[subsection]{Example}
\newtheorem*{convention*}{Convention}
\newtheorem{convention}[subsection]{Convention}
\newtheorem*{hypothesis*}{Hypothesis}

\numberwithin{equation}{subsection}
\numberwithin{enumi}{subsection}

\def\o{\,\circ\,}
\def\X{\mathfrak X}
\def\al{\alpha}
\def\be{\beta}
\def\ga{\gamma}
\def\de{\delta}
\def\ep{\varepsilon}

\def\th{\theta}

\def\la{\lambda}
\def\rh{\rho}
\def\si{\sigma}
\def\ta{\tau}
\def\ph{\varphi}

\def\ch{\chi}
\def\ps{\psi}
\def\om{\omega}

\def\Si{\Sigma}

\def\i{^{-1}}
\def\x{\times}
\def\p{\partial}
\let\on=\operatorname

\def\AMSonly#1{}

\def\Id{\on{Id}}
\def\R{\mathbb R}

\def\Diff{{\on{Diff}}}

\newcommand{\sr}[1]%
{\ifmmode{}^\dagger\else${}^\dagger$\fi\ifvmode
\vbox to 0pt{\vss
 \hbox to 0pt{\hskip\hsize\hskip1em
 \vbox{\hsize3cm\raggedright\pretolerance10000
 \noindent #1\hfill}\hss}\vss}\else
 \vadjust{\vbox to0pt{\vss%
 \hbox to 0pt{\hskip\hsize\hskip1em%
 \vbox{\hsize3cm\raggedright\pretolerance10000%
 \noindent #1\hfill}\hss}\vss}}\fi%
}

\def\C{\mathbb{C}}

\def\I{\mathbb{I}}

\def\N{\mathbb{N}}

\def\R{\mathbb{R}}

\def\cA{\mathcal{A}}
\def\cB{\mathcal{B}}

\def\cD{\mathcal{D}}

\def\cG{\mathcal{G}}
\def\cH{\mathcal{H}}

\def\cK{\mathcal{K}}
\def\cL{\mathcal{L}}

\def\cS{\mathcal{S}}

\def\sB{\mathscr{B}}

\def\sK{\mathscr{K}}

\def\RR{\mathbb R}

\def\subs{\subseteq}

\def\oo{\infty}
\def\Cb{C_b}

\def\ev{\on{ev}}

\def\rM{\{M\}}
\def\bM{(M)}
\def\cBl{\cB_{\on{loc}}}

\def\BM{\cB^{[M]}}
\def\BrM{\cB^{\rM}}
\def\BbM{\cB^{\bM}}

\def\DM{\cD^{[M]}}
\def\DrM{\cD^{\rM}}
\def\DbM{\cD^{\bM}}

\def\CM{C^{[M]}}
\def\CrM{C^{\rM}}
\def\CbM{C^{\bM}}
\def\CMb{C^{[M]}_b}

\def\SLM{\tensor{\cS}{}_{[L]}^{[M]}}
\def\SrLM{\tensor{\cS}{}_{\{L\}}^{\rM}}
\def\SbLM{\tensor{\cS}{}_{(L)}^{\bM}}

\def\Hoo{H^{\infty}}

\def\Sp{W^{\infty,p}}

\def\SpM{W^{[M],p}}
\def\SpbM{W^{\bM,p}}
\def\SprM{W^{\rM,p}}

\def\set{\Si}
\def\Lin{\cL}

\def\ind{\varinjlim}
\def\proj{\varprojlim}

\def\DiffA{\Diff\cA(\R^n)}
\def\DiffB{\Diff\cB(\R^n)}
\def\DiffD{\Diff\cD(\R^n)}
\def\DiffS{\Diff\cS(\R^n)}
\def\DiffBM{\Diff\BM(\R^n)}
\def\DiffSp{\Diff\Sp(\R^n)}
\def\DiffSpM{\Diff\SpM(\R^n)}
\def\DiffDM{\Diff\DM(\R^n)}
\def\DiffSLM{\Diff\SLM(\R^n)}

\allowdisplaybreaks

\title[An exotic zoo of diffeomorphism groups on $\R^n$] 
{An exotic zoo of diffeomorphism groups on $\R^n$} 

  \author{Andreas Kriegl, Peter W. Michor, and Armin Rainer}
  
  \address{Andreas Kriegl: Fakult\"at f\"ur Mathematik, Universit\"at Wien, 
  Oskar-Morgenstern-Platz~1, A-1090 Wien, Austria}
  \email{andreas.kriegl@univie.ac.at}
  
  \address{Peter W. Michor: Fakult\"at f\"ur Mathematik, Universit\"at Wien, 
  Oskar-Morgenstern-Platz~1, A-1090 Wien, Austria}
  \email{peter.michor@univie.ac.at}
  
  \address{Armin Rainer: Fakult\"at f\"ur Mathematik, Universit\"at Wien, 
  Oskar-Morgenstern-Platz~1, A-1090 Wien, Austria}
  \email{armin.rainer@univie.ac.at}
  
  \thanks{AK was supported by FWF-Project P~23028-N13; 
  AR by FWF-Projects P~22218-N13 and P~26735-N25}

\date{\today}

\begin{document}

\begin{abstract}
  Let $\CM$ be a (local) Denjoy--Carleman class of Beurling or Roumieu type, where the weight sequence $M=(M_k)$ is log-convex 
  and has moderate growth.  
  We prove that the groups $\DiffBM$, $\DiffSpM$, $\DiffSLM$, and $\DiffDM$ 
  of $\CM$-diffeomorphisms on $\R^n$ which differ from the identity by a mapping in 
  $\BM$ (global Denjoy--Carleman),
  $\SpM$ (Sobolev--Denjoy--Carleman), 
  $\SLM$ (Gelfand--Shilov),
  or $\DM$ (Denjoy--Carleman with compact support)
  are $\CM$-regular Lie groups. As an application we use the $R$-transform to show that the 
  Hunter--Saxton PDE on the real line is well-posed 
        in any of the classes $W^{[M],1}$, $\SLM$, and $\DM$. Here we find some surprising groups with 
  continuous left translations and $\CM$ right translations (called half-Lie groups), which, however, also admit 
  $R$-transforms. 
\end{abstract}

\maketitle


\section{Introduction}

In this article we introduce a multitude of groups of $\CM$-diffeomorphisms on $\R^n$ and 
prove that all of them are $\CM$-regular Lie groups. 

Recall that a $C^\infty$-mapping $f$ is $\CrM$ if for each compact set $K$ there exists $\rh>0$ 
such that the set
\[
  \Big\{\frac{f^{(k)}(x)}{\rh^k\,k!\, M_k} : x \in K, k \in \N\Big\}
\] 
is bounded, where $M=(M_k)$ is a positive sequence. In this way we get the so-called Denjoy--Carleman 
classes of Roumieu type $\CrM$. If we replace the existential by a universal quantifier we obtain 
the Denjoy--Carleman classes of Beurling type $\CbM$. We will denote by $\CM$ either of them, 
and write $\Box$ for $\exists$ or $\forall$. In 
this paper we shall refer to $\CM$ as \emph{local} Denjoy--Carleman classes, for reasons which will 
become apparent instantly.

In \cite{KMRc}, \cite{KMRq}, and \cite{KMRu} we extended the class $\CM$ to mappings between 
admissible (that is convenient) locally convex spaces and proved that $\CM$ then forms a cartesian closed 
category, i.e., $\CM(E \times F,G) \cong \CM(E,\CM(F,G))$, provided that $M=(M_k)$ is log-convex and 
has moderate growth. Furthermore, we showed that the $\CM$-diffeomorphisms of a compact manifold 
form a $\CM$-regular Lie group. This theory goes by the name \emph{convenient setting}; see Section \ref{sec:DC} 
for a very short presentation of the results used in this paper.

In the present paper we apply the theory developed in \cite{KMRc}, \cite{KMRq}, and \cite{KMRu} in order 
to prove that various sets of $\CM$-diffeomorphisms of $\R^n$ form $\CM$-regular Lie groups.
We denote by $\DiffA$ the set of all mappings $\Id + f : \R^n \to \R^n$, where $\inf_{x \in \R^n} \det(\I_n + df(x))>0$ and
$f \in \cA$, for any of the following classes $\cA$ of test functions:
\begin{itemize}
  \item Global Denjoy--Carleman classes 
  \[
    \BM(\R^n) = \Big\{f\in C^\infty(\R^n) : \Box \rh>0 \sup_{\al \in \N^n}\frac{\|\p^\al f\|_{L^\infty(\R^n)}}{\rh^{|\al|}\,|\al|!\, M_{|\al|}}< \infty\Big\}.
  \]
  \item Sobolev--Denjoy--Carleman classes 
  \[  
    \SpM(\R^n) = \Big\{f\in C^\infty(\R^n) : \Box \rh>0 \sup_{\al \in \N^n}\frac{\|\p^\al f\|_{L^p(\R^n)}}{\rh^{|\al|}\,|\al|!\, M_{|\al|}}< \infty\Big\},
    \quad 1 \le p <\infty.
  \]
  \item Gelfand--Shilov classes 
  \[  
    \SLM(\R^n) = \Big\{f\in C^\infty(\R^n) : \Box \rh>0 
    \sup_{\substack{p\in \N\\\al \in \N^n}}\frac{\|(1+|x|)^p \p^\al f\|_{L^\infty(\R^n)}}{\rh^{p+|\al|}\,p!|\al|!\, L_p M_{|\al|}} < \infty\Big\}.
  \]
  \item Denjoy--Carleman functions with compact support
  \[
    \DM(\R^n) = \CM(\R^n) \cap \cD(\R^n) = \BM(\R^n) \cap \cD(\R^n).
  \]  
\end{itemize}

We require that $M=(M_k)$ is log-convex and has moderate growth, and that $\CbM \supseteq C^\om$ in the Beurling case. 
These assumptions guarantee (and are partly necessary for) the validity of basic results like stability under composition, 
inverse mapping theorem, solvability of ODEs, and cartesian closedness in the class $\CM$ which are essential for our analysis. 
Note that $\DM(\R^n)$ is trivial unless $M=(M_k)$ is non-quasianalytic.

For the sequence $L=(L_k)$ we just assume $L_k\ge 1$ for all $k$. 
Note that $\DM \subseteq \SLM$, see Proposition \ref{prop:incl}, and hence $\SLM$ is certainly non-trivial if $M=(M_k)$ is 
non-quasianalytic.     

The following is our main theorem.

\begin{theorem} \label{thm:main}
  Let $M=(M_k)$ be log-convex and have moderate growth; in the Beurling case we also assume $\CbM \supseteq C^\om$.
  Assume that $L=(L_k)$ satisfies $L_k \ge 1$ for all $k$. 
  Let $1 \le p < q\le \infty$.
  Then $\DiffBM$, $\DiffSpM$, $\DiffSLM$, and $\DiffDM$ are $\CM$-regular Lie groups.
  We have the following $\CM$ injective group homomorphisms
  \[
    \DiffDM \!\!\rightarrowtail\!\! \DiffSLM \!\!\rightarrowtail\!\! \DiffSpM \!\!\rightarrowtail\!\! \Diff{W^{[M],q}(\R^n)} 
    \!\!\rightarrowtail\!\! \DiffBM.
  \]
  Each group in this diagram is normal in the groups on its right.
\end{theorem}

For the precise meaning of \emph{$\CM$-regular} we refer to Section \ref{sec:reg}.

The \emph{classical} case, i.e., without predescribed $[M]$-growth was recently proved in \cite{MichorMumford13}: 
The groups of diffeomorphisms
\[
      \xymatrix{
        \DiffD \ar@{{ >}->}[r] & \DiffS \ar@{{ >}->}[r] & \Diff W^{\infty,2}(\R^n) \ar@{{ >}->}[r] & \DiffB 
      }
  \]
are $C^\infty$-regular Lie groups. The arrows in the diagram describe $C^\infty$ injective group homomorphisms, and 
each group is a normal subgroup of the groups on its right.

In \cite{Glockner05} the parameterization $\Id + f$ for $f\in \cD(\mathbb R^n)^n$ 
was used as global chart for $\Diff\cD(\mathbb R^n)$ for the first time.
In the paper \cite{Walter12} it was shown that $\Diff\cS(E)$ and $\Diff\cB(E)$ are $C^\infty$-regular 
Lie groups, where $E$ is a Banach space; also the case of other systems of weight functions on $E$ 
was treated. The method of proof of \cite{Walter12} is iterative in the degree of 
differentiability. 

We want to point out here that the conclusions of Theorem \ref{thm:main} carry over to the groups 
$\Diff\BM(E)$ and $\Diff\SLM(E)$ for a Banach space $E$
instead of $\mathbb R^n$ with some obvious changes in notation in the proofs given below and 
replacing Lemma~\ref{lem:LA} by a Neumann series argument. 
The definition of the spaces $\BM(E,F)$ and $\SLM(E,F)$ along with their basic topological properties 
are given in full generality for Banach spaces $E$ and $F$. 
We wanted to give a uniform
proof for all cases; since $\DiffSpM$ and $\DiffDM$ do not make sense on an infinite dimensional Banach 
space, the main arguments in this paper are done only for $E=\mathbb R^n$.
Moreover, there are exponential laws available for the spaces $\cA$ treated in this paper; we do 
not need them here, so relegated them to another paper \cite{KMR14b}.

The paper is organized as follows. We collect preliminaries on weight sequences and on Fa\`a di Bruno's formula in 
Section \ref{sec:prelim}. In Section \ref{sec:DC} we review the convenient setting of local Denjoy--Carleman classes.
We introduce the classes of test functions alluded to above and discuss some aspects of their 
topology in Section \ref{sec:test}, and we explain their relative inclusions in Section \ref{sec:incl}. 
In Section \ref{sec:comptest} we collect some results on composition of test functions for use in later sections.
We characterize $\CM$-plots, i.e., $\CM$-mappings defined in open subsets of Banach spaces, 
in spaces of test functions in Section \ref{sec:plots}; this is crucial for proving Theorem \ref{thm:main}, see Subsection \ref{ssec:localDC}. 
We show that $\DiffA$ is a group with respect to composition in Section \ref{sec:group}, and
that composition and inversion are $\CM$ in Section \ref{sec:comp} and Section \ref{sec:inv}. 
Regularity is shown in Section \ref{sec:reg}. The proof of Theorem \ref{thm:main} is completed in Section \ref{sec:end} 
by proving the assertions on normality. In Section \ref{sec:local} we show that Theorem \ref{thm:main} fails 
if we only require local $[M]$-estimates:
left translation is not 
$\CbM$ on $\Diff \cB^{\bM}_{\on{loc}}(\R)$, where $\cB^{\bM}_{\on{loc}}(\R) := \CbM(\R) \cap \cB(\R)$.  

Section \ref{sec:HS} is devoted to the Hunter--Saxton PDE on the real line, which 
corresponds to the geodesic equation on an extension of certain diffeomorphism groups on the real 
line for the right invariant weak Riemannian metric induced by the weak inner product 
$\langle  X,Y\rangle_{\dot H^1}:= \int_{\mathbb R}X'(x)\,Y'(x)\,dx$ on the Lie algebra. Using the 
results of the preceding sections, the $R$-transform from \cite{BBM14b} carries over 
to 1-dimensional extensions of many of the groups $\Diff\cA(\mathbb R)$ and we obtain, that the 
Hunter--Saxton equation is well-posed 
in the corresponding space $\cA(\R)$.
In Section~\ref{sec:strangeSpM} we find the surprising result that the corresponding extensions of 
groups modeled on $\SpM(\R)$ or on $\SpM(\R)\cap L^1(\R)$ for $1<p< \infty$ 
which are needed for the $R$-transform to work are only topological groups with $\CM$ right 
translations; we call them {\it half-Lie groups}. 
The applications in Section \ref{sec:HS} to the Hunter--Saxton equation were the motivation for us 
to check whether groups like $\Diff\cA(\mathbb R^n)$ were $\CM$-regular Lie groups.
An analogous $R$-transform for the homogeneous $H^1$-metric on the space of plane immersed curves 
modulo translations was developed in \cite{BBMM14}. An $R$-transform for several homogeneous 
$H^2$-metrics is in \cite{BBM14a}. Also in these cases Denjoy-Carleman regularity for the 
corresponding geodesic equations can be proved (although this is not stated).

We expect that all our results carry over to the framework of ultradifferentiable classes in the sense of 
\cite{BMT90}, where the growth is controlled by a weight function $\om$ instead of a sequence $M=(M_k)$. 
In fact, both, the weight sequence and the weight function approach, are subsumed under a more general notion 
of ultradifferentiable class defined via weight matrices for which the necessary tools of convenient calculus 
were developed in \cite{Schindl14}, see also \cite{RainerSchindl12} and \cite{RainerSchindl14}. 

\bigskip
\paragraph{\bf Notation}
We use $\N = \N_{>0} \cup \{0\}$.
For each multiindex $\al=(\al_1,\ldots,\al_n) \in \N^n$, we write
$\al!=\al_1! \cdots \al_n!$, $|\al|= \al_1 +\cdots+ \al_n$, and 
$\p^\al=\p^{|\al|}/\p x_1^{\al_1} \cdots \p x_n^{\al_n}$. 
We write $f^{(k)}(x) = d^k f(x)$ for the $k$-th order Fr\'echet derivative of $f$ at $x$ and 
$f^{(\al)}(x) = \p^\al f(x)$ for partial derivatives.

For locally convex spaces $E$ let $\sB(E)$ denote the set of all
closed absolutely convex bounded subsets $B \subseteq E$.
For $B \in \sB(E)$ we denote by $E_B$ the 
linear span of $B$ equipped with the Minkowski functional $\|x\|_B = \inf \{\la>0 : x \in \la B\}$. 
If $E$ is a convenient vector space, then $E_B$ is a 
Banach space.  The collection of compact subsets $K \subseteq U$ is denoted by $\sK(U)$.
In a Banach space $E$ we denote by $B_r(x) := \{y\in E : \|x-y\|<r\}$ the open ball with center $x$ and radius $r$.

We denote by $E^*$ (resp.\ $E'$) the dual space of continuous (resp.\ bounded) linear functionals. $\Lin(E_1,\ldots,E_k;F)$
is the space of $k$-linear bounded mappings $E_1 \x \cdots \x E_k \to F$; if $E_i =E$ for all $i$, we also write $\Lin^k(E;F)$.  
If $E$ and $F$ are Banach spaces, then $\|~\|_{\Lin^k(E;F)}$ denotes the operator norm on $\Lin^k(E;F)$. 

The symbol $\Box$ stands for a quantifier $\forall$ or $\exists$. It is always tied to some space of $[M]$-ultradifferentiable 
functions and should be interpreted as $\Box:=\forall$ if $[M] = \bM$ and $\Box:=\exists$ if $[M] = \rM$.  
Statements that involve more than one $[M]$ symbol must not be interpreted by mixing $\bM$ and $\rM$.

\section{Preliminaries} \label{sec:prelim}

\subsection{Weight sequences} \label{ssec:ws}

Let $M=(M_k)=(M_k)_{k=0,1,\ldots}$ denote a sequence of positive numbers.
We shall always assume that $M_0 = 1 \le M_1$. 

We say that $M=(M_k)$ is \textbf{log-convex} if $k\mapsto\log M_k$ is convex, or equivalently, if 
\begin{equation}
  M_k^2 \le M_{k-1} M_{k+1}, \quad k \in \N.
\end{equation}
If $M=(M_k)$ is log-convex, then $M=(M_k)$ has the following properties:
\begin{align}\label{eq:logconvex}
  &M=(M_k) \text{ is weakly log-convex, i.e., }  k!\, M_k \text{ is log-convex}, \\
  &(M_k)^{1/k} \text{ is non-decreasing}, \label{eq:nondecreasing}\\
  &M_j  M_k\le M_{j+k}, \quad\text{for } j,k\in \N, \label{eq:alg}\\
  &M_1^j \, M_k\ge M_j\, M_{\al_1} \cdots M_{\al_j}, \quad \text{for }\al_i\in \N_{>0}, ~\al_1+\dots+\al_j = k,  \label{eq:FdB} \\
  &M_1^k \, M_n\ge M_k\, M_{1}^{k_1} \cdots M_{n}^{k_n}, \quad \text{for }k_i\in \N, ~\sum_{i=1}^n i k_i = n, ~\sum_{i=1}^n k_i = k; \label{eq:Childress}
\end{align}
cf.\ \cite{KMRu} or \cite{RainerSchindl12}, and \cite[Prop~4.4]{BM04} for \eqref{eq:Childress}. 

We say that $M=(M_k)$ is
\textbf{derivation closed} if 
\begin{equation} \label{eq:dc}
\sup_{k \in \N_{>0}} \Big(\frac{M_{k+1}}{M_k}\Big)^{\frac{1}{k}} < \infty,
\end{equation}
and that $M=(M_k)$ has \textbf{moderate growth} if 
\begin{equation} \label{eq:mg0}
\sup_{j,k \in \N_{>0}} \Big(\frac{M_{j+k}}{M_j \, M_k}\Big)^{\frac{1}{j+k}} <
\infty.
\end{equation}
Obviously, \eqref{eq:mg0} implies \eqref{eq:dc}. 
If $M=(M_k)$ is
derivation closed, then also $k!\, M_k$ is derivation closed and we have
\begin{equation} \label{eq:dc1}
  (k+j)!\, M_{k+j} \le C^{j(k+j)}\, k!\, M_k, \quad\text{for } k,j \in \N 
\end{equation}
for some constant $C\ge 1$.

A weakly log-convex sequence $M=(M_k)$ is called \textbf{quasianalytic} if 
\begin{equation}
  \sum_{k=1}^\infty (k!\, M_k)^{-1/k} = \infty,
\end{equation}
and \textbf{non-quasianalytic} otherwise. It is  called 
\textbf{strongly non-quasianalytic} if
\begin{equation} \label{eq:snqa}
  \sup_k \frac{M_{k+1}}{M_k} \sum_{\ell\ge k} \frac{M_{\ell}}{(\ell+1) M_{\ell+1}} < \infty.
\end{equation}
We refer to \cite{KMRc}, \cite{KMRq}, \cite{KMRu}, or \cite{RainerSchindl12} 
for a detailed exposition of the connection between these conditions on $M=(M_k)$
and the properties of $\CM$. 

\subsection{Fa\`a di Bruno's formula} \label{ssec:Faa}

Let $f \in C^\infty(\R^m,\R)$ and $g \in C^\infty(\R^n,\R^m)$.
We have, cf.\ \cite[Prop~4.3]{BM04}, for all $\ga \in \N^n\setminus \{0\}$, 
\begin{equation} \label{eq:FaaBM}
  \frac{\p^\ga(f\o g)(x)}{\ga!} = \sum \frac{\al!}{k_1!\cdots k_\ell!}\, \frac{(\p^\al f)(g(x))}{\al!} 
    \Big(\frac{\p^{\de_1}g(x)}{\de_1!}\Big)^{k_1} \cdots \Big(\frac{\p^{\de_\ell}g(x)}{\de_\ell!}\Big)^{k_\ell},
\end{equation}
  where $\al=k_1+\cdots+ k_\ell$ and the sum is taken over all sets $\{\de_1,\ldots,\de_\ell\}$ of $\ell$ distinct 
  elements of $\N^n \setminus \{0\}$ and all ordered $\ell$-tuples $(k_1,\ldots,k_\ell) \in (\N^m \setminus \{0\})^\ell$, 
  $\ell = 1, 2,\ldots$, such that $\ga = \sum_{i=1}^\ell |k_i| \de_i$.

The conclusion of the following lemma will be used several times.

\begin{lemma*}
  Let $M=(M_k)$ satisfy \thetag{\ref{ssec:ws}.\ref{eq:Childress}} and let $A>0$. Then there are positive constants $B,C$ depending 
  only on $AM_1$, $m$, and $n$, and $C \to 0$ as $A \to 0$ such that
  \[
    \sum \frac{\al!}{k_1!\cdots k_\ell!} A^{|\al|} M_{|\al|} M_{|\de_1|}^{|k_1|} \cdots M_{|\de_\ell|}^{|k_\ell|} \le B C^{|\ga|} M_{|\ga|}
  \]
  where the sum is as above.   
\end{lemma*} 

\begin{proof}
  Inequality (\ref{ssec:ws}.\ref{eq:Childress}) implies $M_{|\al|} M_{|\de_1|}^{|k_1|} \cdots M_{|\de_\ell|}^{|k_\ell|} \le M_1^{|\al|}M_{|\ga|}$, see 
  \cite[Cor~4.5]{BM04}, and hence 
  \[
    \sum \frac{\al!}{k_1!\cdots k_\ell!} A^{|\al|} M_{|\al|} M_{|\de_1|}^{|k_1|} \cdots M_{|\de_\ell|}^{|k_\ell|} \le 
    \sum \frac{\al!}{k_1!\cdots k_\ell!} (A M_1)^{|\al|} M_{|\ga|}.
  \]
  Now the assertion follows from the fact that $h_{\ga}:=\sum \frac{\al!}{k_1!\cdots k_\ell!} (A M_1)^{|\al|}$ are the coefficients of a 
  convergent power series $\sum_{\ga \in \N^n} h_\ga x^\ga$, see \cite[Lem~4.8]{BM04}. Its domain of convergence increases as $A \to 0$.
\end{proof}

We will also use Fa\`a di Bruno's formula for Fr\'echet derivatives of mappings between Banach spaces: for $k\ge 1$, 
\begin{align} \label{eq:FaaF}
\frac{(f\o g)^{(k)}(x)}{k!} 
= \on{sym}\Big( \sum_{j\ge 1} \sum_{\substack{\al\in \N_{>0}^j\\ \al_1+\dots+\al_j =k}}
\frac{1}{j!} f^{(j)}(g(x))\o \Big(\frac{g^{(\al_1)}(x)}{\al_1!}\x\ldots\x\frac{g^{(\al_j)}(x)}{\al_j!}\Big)\Big),
\end{align}
where $\on{sym}$ denotes symmetrization of multilinear mappings.

Occasionally, we shall use formula \eqref{eq:FaaBM} for mappings $g : U \times \R^n \to U \times \R^n$ and $f : U \times \R^n \to \R^n$ 
defined on a product $U \times \R^n$, where $U$ is open in a Banach space. If we write $\p^\ga = \p_1^{\ga_0}\p_2^{\ga'}$ 
for multiindices $\ga=(\ga_0,\ga') = (\ga_0,\ga_1,\ldots,\ga_n) \in \N^{1+n}$, where
$\p_1^{\ga_0}$ are iterated total derivatives with respect to $u \in U$ and $\p_2^{\ga'}$ are partial derivatives 
with respect to $x \in \R^n$, see Convention \ref{convention}, then formula \eqref{eq:FaaBM} remains true up to symmetrization 
on the right-hand side.

\section{Review of local Denjoy--Carleman classes} \label{sec:DC}

\subsection{Local Denjoy--Carleman classes on Banach spaces}

Let $E,F$ be Banach spaces, $U \subseteq E$ open.
We define the \textbf{local Denjoy--Carleman classes} 
\[
  C^{[M]}(U,F) := \Big\{f \in C^\infty(U,F) : \forall K \in \sK(U) ~\Box \rh>0: 
  \|f\|_{K,\rh}^M <\infty\Big\}, 
\]
where (for any subset $K \subseteq U$)
\[
  \|f\|_{K,\rh}^M := \sup_{\substack{k \in \N\\ x \in K}} \frac{\|f^{(k)}(x)\|_{\Lin^k(E;F)}}{\rh^k k! M_k}.
\]
See \cite[4.2]{KMRu} for the locally convex structure of these spaces.
The elements of $C^{\bM}(U,F)$ are said to be of Beurling type; those of $C^{\rM}(U,F)$ of Roumieu type. 
If $M_k=1$, for all $k$, then $C^{(M)}(\R) = C^{(M)}(\R,\R)$ consists of the restrictions 
of the real and imaginary parts of all entire functions, 
while $C^{\{M\}}(\mathbb R)$ coincides with the ring $C^\om(\mathbb R)$ of real analytic functions.

\subsection{Convenient setting of local Denjoy--Carleman classes} \label{ssec:localDC}

The classes $C^{[M]}$ can be extended to 
convenient vector spaces, and they then form cartesian closed categories if the weight sequence $M=(M_k)$ has some regularity properties:
This has been developed in \cite{KMRc}, \cite{KMRq}, and \cite{KMRu}.
 
\begin{hypothesis*}
  From now on we assume that the weight sequence $M=(M_k)$ has the following properties, cf.\ Subsection \ref{ssec:ws}:
\begin{enumerate}
  \item\label{DC.1} $1 =M_0 \le M_k \le M_{k+1}$ for all $k$. 
  \item\label{DC.2} $M$ is log-convex. 
  \item\label{DC.3} $M$ has moderate growth.   
\end{enumerate}
In the Beurling case $C^{[M]} = C^{(M)}$ we also require that 
\begin{enumerate}
  \item\label{DC.4} $C^\om\subseteq C^{(M)}$, 
        or equivalently $M_k^{1/k} \to \infty$ or $M_{k+1}/M_k\to \infty$.
\end{enumerate}  
\end{hypothesis*}

A locally convex space $E$ is called \textbf{convenient} if it is \textbf{$c^\infty$-complete}, i.e., the following equivalent conditions
hold:
\begin{itemize}
   \item A curve $c : \R \to E$ is  $C^\infty$ if and only if $\ell \o c : \R\to \R$ is $C^\infty$ for all $\ell \in E^*$.
   \item Any Mackey--Cauchy sequence converges in $E$.
   \item $E_B$ is a Banach space for all $B \in \sB(E)$.
\end{itemize} 
We equip $E$ with the \textbf{$c^\infty$-topology}, i.e., the final topology with respect to any of the following sets of mappings:
\begin{itemize}
   \item $C^\infty(\R,E)$.
   \item The Mackey-convergent sequences in $E$.
   \item The injections $i_B: E_B \to E$, where $B \in \sB(E)$.
\end{itemize}

For convenient vector spaces $E$ and $F$ and $c^\infty$-open $U \subseteq E$ we define:
\[
  \CM(U,F) := 
  \Bigl\{f\in C^\oo(U,F):
  \forall \ell\; \forall B: \ell \o f \o i_B \in \CM(U_B,\R)\Bigr\}, 
\]
where $\ell \in F^*$, $B\in \sB(E)$, and
$U_B := i_B^{-1}(U)$.
We equip $\CM(U,F)$ with the
initial locally convex structure induced by all linear mappings 
\begin{align*}
\CM(U,F) &\East{\CM(i_B,\ell)}{}  \CM(U_B, \mathbb R), \quad f \mapsto \ell\o f\o i_B.
\end{align*}
Then $\CM(U,F)$ is a
convenient vector space.

The class of all $\CM$-mapping between convenient vector spaces forms a cartesian closed category:

\begin{theorem*}[$\CM$-exponential law, {\cite[5.2]{KMRu}}]
    For convenient vector spaces $E,F,G$ and $c^\infty$-open subsets $U \subseteq E$, $V \subseteq F$,
    \[
      \CM(U \times V,G) \cong \CM(U,\CM(V,G))
    \]
    is a linear $\CM$-diffeomorphism.
\end{theorem*}

The proof of this theorem (and convenient calculus in general) work uniformly over all classes $\CM$,
including the class of real analytic mappings as $C^\om=C^{\{1,1,1,\dots\}}$. The real analytic convenient 
setting was first developed in a different approach in \cite{KrieglMichor90}.

A mapping $f$ between $c^\infty$-open subsets of convenient vector spaces is $\CM$ if and only if it 
respects \textbf{$\CM$-(Banach) plots}, i.e., $\CM$-mappings defined on open balls in Banach spaces. 
Actually, it is enough that $f$ takes $\CM_b$-plots to $\CM$-plots. 

By $\CM_b$ we denote the respective classes defined by boundedness conditions: 
\[
  \CM_b(U,F) := \Big\{f \in C^\infty(U,F) : \forall B ~\forall K ~\Box \rh>0 : \set^M_{K,\rh}(f) \text{ is bounded in } F\Big\}
\]
where $B \in \sB(E)$, $K \in \sK(U_B)$, and (for each subset $K \subseteq U_B$)
\begin{equation}\label{eq:set}
  \set^M_{K,\rh}(f) := \Big\{\frac{f^{(k)}(x)(v_1,\dots,v_k)}{\rh^k\, k!\, M_k}:k\in \mathbb N,x\in K,\|v_i\|_E\leq 1\Big\}.
\end{equation}
In general, $\CM_b(U,F) \subsetneq \CM(U,F)$ and there is no $\CM_b$-exponential law, but:

\begin{lemma*}[{{\cite[4.4,~4.5]{KMRu}}}] 
  We have $\CM_b(U,F) = \CM(U,F)$ in any of the following situations: 
  \begin{itemize}
   \item The Beurling case $[M] = \bM$.
   \item $E$ and $F$ are Banach spaces.
   \item There exists a Baire vector space topology on the dual $F^*$ for which $\on{ev}_x$ is continuous for all $x \in F$. 
\end{itemize}  
\end{lemma*}

In special cases $\CrM$-regularity can be tested along curves:
\begin{itemize}
  \item For non-quasianalytic $M=(M_k)$, the mapping $f$ is $C^{\{M\}}$ if it maps $C^{\{M\}}$-curves to 
        $C^{\{M\}}$-curves, by \cite{KMRc}.
  \item For certain quasianalytic $M=(M_k)$, the mapping $f$ is $C^{\{M\}}$ if it maps $C^{\{N\}}$-curves 
        to $C^{\{N\}}$-curves, for all non-quasianalytic log-convex $N=(N_k)$ which are larger than 
        $M=(M_k)$, by \cite{KMRq}. Real analytic  functions are not among these classes.
\end{itemize}

On open sets in Banach spaces, $C^{[M]}$-vector fields have $C^{[M]}$-flows, 
and between Banach spaces, the $C^{[M]}$ implicit function theorem holds, see \cite[9.2]{KMRu}, 
\cite{Yamanaka89}, and \cite{Yamanaka91}.

For $M=(M_k)$ and $N=(N_k)$ we have
\begin{align*}
 C^{[M]} \subseteq C^{[N]} \quad &\Leftrightarrow \quad 
 M \preceq N \quad :\Leftrightarrow \quad  \exists C,\rh > 0 ~\forall k : M_k \le C \rh^k N_k   \\ 
 C^{[M]} = C^{[N]} \quad &\Leftrightarrow \quad  
 M \approx N \quad :\Leftrightarrow \quad M \preceq N \text{ and } N \preceq M \\
 C^{\{M\}} \subseteq C^{(N)} \quad &\Leftrightarrow \quad 
 M \lhd N \quad :\Leftrightarrow \quad \forall \rh>0 ~\exists C>0 ~\forall k : M_k \le C \rh^k N_k    
\end{align*}
In particular, $C^\om \subseteq C^{\{M\}} \Leftrightarrow \cH(\C^n) \subseteq C^{(M)}(U) ~\forall U\subseteq \R^n 
  \Leftrightarrow \varliminf M_k^{\frac{1}{k}}>0$ and $C^\om \subseteq C^{(M)} \Leftrightarrow \lim M_k^{\frac{1}{k}} = \infty$.

\subsection{Groups of $C^\om$ or $\CM$ diffeomorphisms on a compact manifold}

Every compact $C^1$ manifold $A$ carries a unique $C^1$-diffeo\-mor\-phic real analytic structure.
Let $\Diff^\om(A)$ be the real analytic regular Lie group of all 
real analytic diffeomorphisms of $A$, with the real analytic structure described in 
\cite[8.11]{KrieglMichor90}, see also 
\cite[theorem 43.4]{KM97}.

For every $M=(M_k)$ the group 
$\Diff^{[M]}(A)$ of 
$C^{[M]}$-diffeo\-mor\-phisms of $A$ is a regular $C^{[M]}$-group (but not better), by 
\cite[6.5]{KMRc}, \cite[5.6]{KMRq}, and \cite[9.8]{KMRu}.

\section{Classes of test functions} \label{sec:test}

In the following let $E, F$ be Banach spaces, and let $U \subseteq E$ be open.
Later on in this paper we shall only be concerned with the case that $E$ and $F$ are finite dimensional,
but treating the general case at this point does not complicate the presentation and will 
be useful for reference in other papers. 

In this section $M=(M_k)$ and $L=(L_k)$ are any positive sequences.

\subsection{Smooth functions with globally bounded derivatives}

Consider
\[
  \cB(U,F) := \big\{f \in C^\infty(U,F) : \|f\|^{(k)}_{U} < \infty \text{ for all } k\in \N\big\},
\]
where
\[
   \|f\|^{(k)}_{U} := \sup_{x \in U} \|f^{(k)}(x)\|_{\Lin^k(E;F)}, 
\]
with its natural Fr\'echet topology.

\subsection{Rapidly decreasing Schwartz functions}

Consider  
\[
  \cS(E,F) := \big\{f \in C^\infty(E,F) : \|f\|^{(p,q)}_{E} < \infty \text{ for all } p,q \in \N\big\},
\]
where 
\[
  \|f\|^{(p,q)}_{E} := \sup_{x \in E}  (1+\|x\|)^p \|f^{(q)}(x)\|_{\Lin^q(E;F)},
\]
with its natural Fr\'echet topology.

\subsection{Global Denjoy--Carleman classes}

Let $\rh>0$.
Consider the Banach space 
\[
  \cB^M_\rh(U,F) := \big\{f \in C^\infty(U,F) : \|f\|_{U,\rh}^M < \infty\big\},
\]
where
\[
  \|f\|_{U,\rh}^M := \sup_{\substack{k \in \N\\ x \in U}} \frac{\|f^{(k)}(x)\|_{\Lin^k(E;F)}}{\rh^k k! M_k}.
\]
We define the Fr\'echet space 
\[
  \cB^{\bM}(U,F) := \varprojlim_{n \in \N} \cB^M_{\frac{1}{n}}(U,F) 
\]
and  
\[
  \cB^{\rM}(U,F) := \varinjlim_{n \in \N} \cB^M_{n}(U,F)
\]
which is a compactly regular (LB)-space and thus {($c^\oo$-)}complete, webbed, and (ultra-)bornological; 
see Lemma \ref{lem:cpreg} below.

\subsection{Gelfand--Shilov classes}

Let $\si>0$.
Consider the Banach space 
\[
  \cS_{L,\si}^M(E,F) := \big\{f \in C^\infty(E,F) : \|f\|_{E,\si}^{L,M} < \infty\big\}.
\]
with the norm
\[
  \|f\|_{E,\si}^{L,M} := \sup_{\substack{p,q \in \N\\ x \in E}}  \frac{(1+\|x\|)^p \|f^{(q)}(x)\|_{\Lin^q(E;F)}}{\si^{p+q}\, p! q!\, L_p M_q}.
\]
We define the Fr\'echet space 
\[
  \SbLM(E,F) := \varprojlim_{n \in \N} \cS^M_{L,\frac{1}{n}}(E,F) 
\]
and  
\[
  \SrLM(E,F) := \varinjlim_{n \in \N} \cS^M_{L,n}(E,F)
\]
which is a compactly regular (LB)-space and thus {($c^\oo$-)}complete, webbed, and (ultra-)bornological; 
see Lemma \ref{lem:cpreg} below.

\subsection{Smooth functions with globally $p$-integrable derivatives}

In this context we assume that $E$ and $F$ are finite dimensional.
For $p \in [1,\infty]$, consider the space 
\begin{align*}
  \Sp(\R^m,\R) &= \Sp(\R^m) = \bigcap_{k \in \N} W^{k,p}(\R^m) \\ 
    &= \big\{f \in C^\infty(\R^m) : \|f^{(\al)}\|_{L^p(\R^m)} < \infty \text{ for all } \al \in \N^m\big\}  
\end{align*}
with its natural Fr\'echet topology (cf.\ \cite[p.~199]{Schwartz66}), and set
\[
  \Sp(\R^m,\R^n) := (\Sp(\R^m,\R))^n.  
\]
The most important case is $p=2$:
\[
  W^{\infty,2}(\R^m,\R^n) = \Hoo(\R^m,\R^n).
\]
Note that $W^{\infty,\infty}(\R^m,\R^n) =\cB(\R^m,\R^n)$, so henceforth we restrict ourselves to the case $p \in [1,\infty)$.

\subsection{Sobolev--Denjoy--Carleman classes}

Let $p \in[1,\infty)$ and $\rh>0$.
Consider the Banach space 
\[
  W^{M,p}_{\si}(\R^m,\R) := \big\{f \in C^\infty(\R^m) : \|f\|_{\R^m,\si}^{M,p} < \infty\big\},
\]
where
\[
  \|f\|_{\R^m,\si}^{M,p} := \sup_{\al \in \N^m} \frac{\|f^{(\al)}\|_{L^p(\R^m)}}{\si^{|\al|}\, |\al|!\, M_{|\al|}}.
\]
We define the Fr\'echet space 
\[
  \SpbM(\R^m,\R) := \proj_{n \in \N} W^{M,p}_{\frac{1}{n}}(\R^m,\R) 
\]
and  
\[
  \SprM(\R^m,\R) := \ind_{n \in \N} W^{M,p}_{n}(\R^m,\R)
\]
which is a compactly regular (LB)-space and thus {($c^\oo$-)}complete, webbed, and (ultra-)bornological; 
see Lemma \ref{lem:cpreg} below.
We set
\[
   \SpM(\R^m,\R^n) := (\SpM(\R^m,\R))^n.
\]

\subsection{Smooth functions with compact support}
We denote by $\cD(\R^m)=\cD(\R^m,\R)$ the nuclear (LF)-space of smooth functions on $\R^m$ with compact support, and set  
\[
   \cD(\R^m,\R^n) := (\cD(\R^m,\R))^n.
\]

\subsection{Denjoy--Carleman functions with compact support}

We define 
\[
  \DM(\R^m) := \CM(\R^m) \cap \cD(\R^m) = \BM(\R^m) \cap \cD(\R^m) 
\]
which is non-trivial only if $M=(M_k)$ is non-quasianalytic.
We equip $\DM(\R^m)$ with the following topology,
\begin{align*}
  \cD^{[M]}(\R^m) &= \varinjlim_{K \in \sK(\R^m)} \cD^{[M]}_K(\R^m)  
\end{align*}
where
\begin{align*}
   \cD^{\bM}_K(\R^m) :=  \varprojlim_{\ell \in \N} \cD^M_{K,\frac 1 \ell}(\R^m) \\
   \cD^{\rM}_K(\R^m) :=  \varinjlim_{\ell \in \N} \cD^M_{K,\ell}(\R^m)
\end{align*}
and
\[
  \cD^M_{K,\rh}(\R^m) := \{f \in C^\infty(\R^m) : \on{supp} f \subseteq K, \|f\|_{K,\rh}^M = \|f\|_{\R^m,\rh}^M < \infty\}
\]
is a Banach space.
Then $\cD^{\bM}(\R^m)$ is a (LFS)-space and $\cD^{\rM}(\R^m)$ is a Silva space, see \cite{Komatsu73}, and both are hence complete and convenient. 
Alternatively one may consider the 
topology induced by the inclusion in the diagonal of $C^{[M]}(\R^m) \times \cD(\R^m)$ or $\cB^{[M]}(\R^m) \times \cD(\R^m)$ 
which are bornologically equivalent 
to the former topologies.
We set
\[
   \DM(\R^m,\R^n) := (\DM(\R^m,\R))^n.
\]


\begin{lemma} \label{lem:cpreg}
  The inductive limits defining
  $\cB^{\rM}(U,F)$, $\SrLM(E,F)$, and $\SprM(\R^\ell,\R)$ are compactly regular.
\end{lemma}

\begin{proof}
  It suffices by \cite[Satz 1]{Neus78} to verify condition (M) of \cite{Retakh70}:
  There exists a sequence of increasing 0-neighborhoods $U_n\subs \cB^M_n(U,F)$ (resp.\ $\cS^M_{L,n}(E,F)$, 
  or $W^{M,p}_{n}(\R^\ell,\R)$), such that for 
  each $n$ there exists an $m\geq n$ for which the topologies of $\cB^M_k(U,F)$ (resp.\ $\cS^M_{L,k}(E,F)$, or $W^{M,p}_{k}(\R^\ell,\R)$) and
  of $\cB^M_m(U,F)$ (resp.\ $\cS^M_{L,m}(E,F)$, or $W^{M,p}_{m}(\R^\ell,\R)$) coincide on $U_n$ for all $k\geq m$.

  Let us write $\|~\|_\si$ for either $\|~\|_{U,\si}^M$, $\|~\|_{E,\si}^{L,M}$, or $\|~\|_{\R^\ell,\si}^{M,p}$.
  For $\si' \ge \si$ we have $\|f\|_{\si'} \le \|f\|_{\si}$. 
  So consider the $\ep$-balls $U^{\si}_\ep(f):=\{g:\|g-f\|_{\si}\leq \ep\}$ in $\cB^M_\si(U,F)$ (resp.\ $\cS^M_{L,\si}(E,F)$, 
  or $W^{M,p}_{\si}(\R^\ell,\R)$).
  It suffices to show that for $\si >0$, $\si_1:=2\si$, $\si_2>\si_1$, $\ep>0$, and 
  $f\in U^{\si}_1:=U^{\si}_1(0)$ there exists $\de>0$ such that 
  $U^{\si_2}_\de(f)\cap U^{\si}_1\subs U^{\si_1}_\ep(f)$.
  
  \paragraph{\bf Case $\cA \in \{\BM,\SrLM\}$}
  We prove this for $\SrLM$, however the following arguments also give a proof for $\BrM$ if we 
  set $p = 0$, take the supremum over $x \in U$, and agree that $L_0=1$.
  Since $f\in U^{\si}_1$ we have
        \[
                \|f\|^{(p,q)}_{E} = \sup_{x\in E} (1+\|x\|)^p \|f^{(q)}(x)\|_{\Lin^q(E;F)}\leq p!q!\si^{p+q} L_p M_q \text{ for all }p,q.
        \]
  Let $\frac1{2^N}<\frac{\ep}2$ and $\de:=\ep\,\big(\frac{\si_1}{\si_2}\big)^{N-1}$.
  Let $g\in U^{\si_2}_\de(f)\cap U^{\si}_1$, i.e.,
        \begin{align*}
                \|g\|^{(p,q)}_{E} &\leq p!q!\si^{p+q} L_p M_q \quad \text{ and }\\
                \|g- f\|^{(p,q)}_{E} &\leq \de\,p!q!\si_2^{p+q} L_p M_q   \quad \text{ for all }p,q.
        \end{align*}
  Then
        \begin{align*}
                \|g-f\|^{(p,q)}_{E} &\leq \|g\|^{(p,q)}_E+\|f\|^{(p,q)}_E \leq 2\, p!q!\si^{p+q} L_p M_q
                =2\,p!q! \si_1^{p+q} L_p M_q\,\frac1{2^{p+q}} \\
                &< \ep\, p!q! \si_1^{p+q} L_p M_q\quad\text{ for }p+q\geq N\\
        \intertext{and}
                \|g-f\|^{(p,q)}_E&\leq  \de\,p!q!\si_2^{p+q} L_p M_q
                \leq \ep\, p!q!\si_1^{p+q} L_p M_q
                \quad\text{ for }p+q<N,
        \end{align*}
  which proves the statement.   

  \paragraph{\bf Case $\cA = \SprM$} The same arguments work for $\SprM$:
  For $N$ and $\de$ as above, $f\in U^{\si}_1$ and $g\in U^{\si_2}_\de(f)\cap U^{\si}_1$ give         
  \begin{align*}
    \|f^{(\al)}\|_{L^p(\R^\ell)} &\leq |\al|!\, \si^{|\al|} M_{|\al|}  \quad \text{ and } \\
    \|g^{(\al)}\|_{L^p(\R^\ell)} &\leq |\al|!\, \si^{|\al|} M_{|\al|} \quad \text{ and }\\
    \|g^{(\al)}- f^{(\al)}\|_{L^p(\R^\ell)} &\leq \de\, |\al|!\, \si_2^{|\al|} M_{|\al|}   \quad \text{ for all } \al,
  \end{align*}
  and hence
        \begin{align*}
                \|g^{(\al)}- f^{(\al)}\|_{L^p(\R^\ell)} &\leq \|g^{(\al)}\|_{L^p(\R^\ell)} + \|f^{(\al)}\|_{L^p(\R^\ell)} 
                \leq 2\, |\al|!\, \si^{|\al|} M_{|\al|}\\
                &=\frac{\si_1^{|\al|}}{2^{|\al|-1}}\, |\al|!\, M_{|\al|} 
                < \ep\, \si_1^{|\al|}\, |\al|!\,  M_{|\al|}, \quad\text{ for }|\al|\geq N,\\
                \|g^{(\al)}- f^{(\al)}\|_{L^p(\R^\ell)} &\leq \de\,  \si_2^{|\al|}\, |\al|!\, M_{|\al|}
                \leq \ep\, \si_1^{|\al|}\, |\al|!\, M_{|\al|},
                \quad\text{ for }|\al|<N,
        \end{align*}
  as required.    
\end{proof}

\section{Inclusions} \label{sec:incl}

Let $M=(M_k)$ and $L=(L_k)$ be any positive sequences, where $L_k\ge 1$ for all $k$. 

\begin{proposition} \label{prop:incl}
  Let $E,F$ be Banach spaces.
  We have the following inclusions. 
    \[
      \xymatrix{
        \cS(E,F) \ar@{{ >}->}[r]  & \cB(E,F) \ar@{{ >}->}[r] & C^\infty(E,F) \\
        \SrLM(E,F) \ar@{{ >}->}[r] \ar@{{ >}->}[u]  & \BrM(E,F) \ar@{{ >}->}[u] \ar@{{ >}->}[r] & \CrM(E,F) \ar@{{ >}->}[u]\\
        \SbLM(E,F) \ar@{{ >}->}[r] \ar@{{ >}->}[u]  & \BbM(E,F) \ar@{{ >}->}[u] \ar@{{ >}->}[r] & \CbM(E,F) \ar@{{ >}->}[u]
      }
    \]
  In the next diagram we omit the source $\R^m$ and the target $\R^n$, i.e., we write $\cA$ instead of $\cA(\R^m,\R^n)$.  
  Let $1 \le p<q <\infty$.
  For the inclusions marked by $*$ we assume that $M=(M_k)$ is derivation closed. 
    \[
      \xymatrix{
        \cD~ \ar@{{ >}->}[r] & \cS~ \ar@{{ >}->}[r] & \Sp~ \ar@{{ >}->}[r] & W^{\infty,q}~ \ar@{{ >}->}[r] 
        & \cB~ \ar@{{ >}->}[r] & C^\infty \\
        \DrM~ \ar@{{ >}->}[r] \ar@{{ >}->}[u] & \SrLM~ \ar@{{ >}->}[r] \ar@{{ >}->}[u] & \SprM~ \ar@{{ >}->}[r]^{*} \ar@{{ >}->}[u] 
        & W^{\rM,q}~ \ar@{{ >}->}[r]^{*} \ar@{{ >}->}[u] 
        & \BrM~ \ar@{{ >}->}[u] \ar@{{ >}->}[r] & \CrM~ \ar@{{ >}->}[u]\\
        \DbM~ \ar@{{ >}->}[r] \ar@{{ >}->}[u] &\SbLM~ \ar@{{ >}->}[r] \ar@{{ >}->}[u] & \SpbM~ \ar@{{ >}->}[r]^{*} \ar@{{ >}->}[u] 
        & W^{\bM,q}~ \ar@{{ >}->}[r]^{*} \ar@{{ >}->}[u] 
        & \BbM~ \ar@{{ >}->}[u] \ar@{{ >}->}[r] & \CbM~ \ar@{{ >}->}[u]
      }
    \]
   All inclusions are continuous. 
   If the target is $\R$ (or $\C$) then all spaces are algebras, provided that $M=(M_k)$ is weakly log-convex,
   and each space in 
   \[
      \xymatrix{
        \cD(\R^m)~ \ar@{{ >}->}[r] & \cS(\R^m)~ \ar@{{ >}->}[r] & \Sp(\R^m)~ \ar@{{ >}->}[r] & W^{\infty,q}(\R^m)~ \ar@{{ >}->}[r] 
        & \cB(\R^m) 
      }
    \]
    is a $\cB(\R^m)$-module, and thus an ideal in each space on its right, likewise 
    each space in 
   \[
      \xymatrix{
        \DM(\R^m) \ar@{{ >}->}[r] & \SLM(\R^m) \ar@{{ >}->}[r] & \SpM(\R^m) \ar@{{ >}->}[r] & W^{[M],q}(\R^m) \ar@{{ >}->}[r] 
        & \BM(\R^m) 
      }
    \]
    is a $\BM(\R^m)$-module, and thus an ideal in each space on its right.  
\end{proposition}

\begin{proof}
  The first diagram is evident. 
  So are the vertical arrows in the second diagram. 
  The inclusion $\cS(\R^m,\R^n) \rightarrowtail \Sp(\R^m,\R^n)$ follows from
  \begin{equation} \label{eq:incl1}
    \int_{\R^m} |f^{(\al)}|^p\, dx = \int_{\R^m} (1+|x|)^{p(m+1)} |f^{(\al)}|^p\, \frac{dx}{(1+|x|)^{p(m+1)}} 
    \le \big(C \|f\|^{(m+1,|\al|)}_{\R^m}\big)^p,
  \end{equation}
  where $C>0$ is a constant depending only on $m$ and $p$.
  The inclusion $\Sp(\R^m,\R^n) \rightarrowtail \cB(\R^m,\R^n)$ follows from the Sobolev inequality
  \begin{equation} \label{eq:incl2}
    \sup_{x \in \R^m} |f(x)| \le C \|f\|_{W^{k,p}}, \quad k:= \lfloor \tfrac{m}{p} \rfloor +1,
  \end{equation}
  where $C$ is a constant depending only on $p$ and $m$. For $p<q$ we have, by \eqref{eq:incl2},
  \begin{align} \label{eq:incl3}
    \|f\|_{L^q} \le \|f\|_{L^p}^{p/q} \|f\|_{L^\infty}^{1-p/q} \le C \|f\|_{W^{k,p}}, 
  \end{align}
  for a constant $C$ depending only on $p$, $q$, and $m$, and hence $\Sp(\R^m,\R^n) \rightarrowtail W^{\infty,q}(\R^m,\R^n)$.

  By \eqref{eq:incl1}, we have 
  \begin{align*}
    \sup_{\al \in \N^m} \frac{\|f^{(\al)}\|_{L^p(\R^m)}}{\si^{|\al|}\, |\al|!\, M_{|\al|}} 
    \le C \sup_{\al \in \N^m} \frac{\|f\|^{(m+1,|\al|)}_{\R^m}}{\si^{|\al|}\, |\al|!\, M_{|\al|}} 
    \le C\, \si^{m+1}\, (m+1)!\, L_{m+1}\, \|f\|^{L,M}_{\R^m,\si},
  \end{align*}
  which implies the inclusion $\SLM(\R^m,\R^n) \rightarrowtail \SpM(\R^m,\R^n)$. 
  If $f \in \DM(\R^m,\R^n)$ and $\on{supp} f \subseteq B_r(0)$, then
  \begin{align*}
    \sup_{\substack{p,q \in \N\\ x \in B_r(0)}}  \frac{(1+|x|)^p \|f^{(q)}(x)\|_{\Lin^q(\R^m;\R^n)}}{\si^{p+q}\, p! q!\, L_p M_q} 
    &\le \sup_{\substack{p,q \in \N\\ x \in B_r(0)}} \frac{(1+r)^p}{\si^p\, p!\, L_p} \frac{ \|f^{(q)}(x)\|_{\Lin^q(\R^m;\R^n)}}{\si^{q}\, q!\, M_q}
    \\
    &\le C(\si,r)\, \|f\|^M_{\R^m,\si}, 
  \end{align*}
  and thus the inclusion $\DM(\R^m,\R^n) \rightarrowtail \SLM(\R^m,\R^n)$.
  
  Finally, let us consider the inclusions $\SpM(\R^m,\R^n)\rightarrowtail W^{[M],q}(\R^m,\R^n)$ and $\SpM(\R^m,\R^n)\rightarrowtail \BM(\R^m,\R^n)$.
  By allowing $q=\infty$, we can handle them simultaneously.
  By \eqref{eq:incl2} and \eqref{eq:incl3}, we have 
  \begin{align*}
    \|f^{(\al)}\|_{L^q} 
    \le  C \|f^{(\al)}\|_{W^{k,p}} 
    =  C \sum_{|\be|\le k} \|f^{(\al+\be)}\|_{L^p}, 
  \end{align*}
  and, thus, using $(|\al|+|\be|)! \, M_{|\al|+|\be|} \le A^{|\be|(\al|+|\be|)}\, |\al|!\, M_{|\al|}$ for some $A\ge 1$ from 
  \thetag{\ref{ssec:ws}.\ref{eq:dc1}}, we obtain
  \begin{align*}
    \sup_{\al \in \N^m} \frac{\|f^{(\al)}\|_{L^q}}{\si^{|\al|}\, |\al|!\, M_{|\al|}}
    &\le  C \sup_{\al \in \N^m} \sum_{|\be|\le k} \frac{\|f^{(\al+\be)}\|_{L^p}}{\si^{|\al|}\, |\al|!\, M_{|\al|}} \\
    &\le  C \sup_{\al \in \N^m} \sum_{|\be|\le k} \frac{\|f^{(\al+\be)}\|_{L^p}}{\si^{|\al|+|\be|}\, (|\al|+|\be|)! \, M_{|\al|+|\be|}}\, 
    A^{|\be|(|\al|+|\be|)} \si^{|\be|} \\
    &\le  C \sup_{\al \in \N^m} \sum_{|\be|\le k} \frac{\|f^{(\al+\be)}\|_{L^p}}{(\si/A^k)^{|\al|+|\be|}\, (|\al|+|\be|)! \, M_{|\al|+|\be|}}\, 
    \si^{|\be|}\\
    &\le  C \sup_{\al \in \N^m} \max_{|\be|\le k} \frac{\|f^{(\al+\be)}\|_{L^p}}{(\si/A^k)^{|\al|+|\be|}\, (|\al|+|\be|)! \, M_{|\al|+|\be|}}\, 
    \sum_{|\be|\le k} 
    \si^{|\be|} \\
    &\le \tilde C  \sup_{\ga \in \N^m} \frac{\|f^{(\ga)}\|_{L^p}}{(\si/A^k)^{|\ga|}\, |\ga|! \, M_{|\ga|}} 
    = \tilde C \|f\|^{M,p}_{\R^m,\si/A^k},
  \end{align*}
  for a constant $\tilde C$ depending only on $m$, $p$, $q$, and $\si$.
\end{proof}

\begin{remark*}
  The fact that 
  $\cD$ is dense in $\Sp$ (but not in $\cB$) and \eqref{eq:incl2} imply that each element of $\Sp$ must tend to $0$ at infinity 
  together with all its iterated partial derivatives. 
\end{remark*}

\section{Composition of test functions}  \label{sec:comptest}

In this section we collect results on composition of test functions that will be needed later on.
Let $M=(M_k)$ and $L=(L_k)$ be any positive sequences, where $L_k\ge 1$ for all $k$.

In general, $\cB$ and $\BM$ are stable under composition, but $\Sp$, $\SpM$, $\cS$, $\SLM$, $\cD$, and $\DM$ are not. 
The following example shows that the ``$0$th derivative'' of the composite $f \o g$ may not have the required decay 
properties at infinity, since $g$ is globally bounded.

\begin{example*}
  Let $f,g \in \cD(\R,\R)$ be such that $f|_{[-1,1]} = 1$ and $|g|\le 1$. 
  The composite $f \o g= 1$ is not in $\bigcup_{1 \le p <\infty} \Sp(\R,\R)$, and hence neither in $\cD(\R,\R)$ nor in $\cS(\R,\R)$.     
\end{example*}

We want to ask that a mapping is of class $\cB$ or $\BM$, but only from the first derivative onwards. 
For $E,F$ Banach spaces, $U \subseteq E$ open, we set 
\begin{align*}
  \cB_2(U,F) &:=  \Big\{f \in C^\infty(U,F) : \|f\|^{(k)}_{U} < \infty \text{ for all } k\in \N_{\ge 1}\Big\}, \\
  \BM_2(U,F) &:= \Big\{f \in C^\infty(U,F) : 
    \Box \rh>0 \sup_{k \in \N_{\ge1}, x\in U} \frac{\|f^{(k)}(x)\|_{\Lin^k(E;F)}}{\rh^k\, k!\,M_k} <\infty \Big\}.
\end{align*}
We choose the subscript $2$ in order to be consistent with Section \ref{sec:HS}, where the subscript $1$ is reserved for a 
different meaning.

\begin{theorem} \label{thm:Bcomp}
  Let $M=(M_k)$ be log-convex.
  Let $E,F,G$ be Banach spaces, and let $U \subseteq E$ and $V \subseteq F$ be open. Then:
  \begin{align*}
    f \in \cB(V,G), ~g \in \cB_2(U,V)    &\Longrightarrow f \o g \in \cB(U,G), \\
    f \in \BM(V,G), ~g \in \BM_2(U,V)    &\Longrightarrow f \o g \in \BM(U,G).
  \end{align*}
\end{theorem}

\begin{remark*}
  In particular, $f \o g$ is $\cB$ or $\BM$ if $f$ and $g$ are $\cB$ or $\BM$, respectively.
\end{remark*}

\begin{proof}
Let $\cA \in \{\cB,\BM\}$. 

\paragraph{\bf Case $\cA= \cB$}
By Fa\`a di Bruno's formula for Banach spaces (\ref{ssec:Faa}.\ref{eq:FaaF}), 
we find    
\begin{align} \label{eq:Faa}
\begin{split}
  &\frac{\|(f\o g)^{(k)}(x)\|_{\Lin^k(E;G)}}{k!} \le \\
&\le \sum_{j\ge 1} \sum_{\substack{\al\in \N_{>0}^j\\ \al_1+\dots+\al_j =k}}
\frac{\|f^{(j)}(g(x))\|_{\Lin^j(F;G)}}{j!}\;\prod_{i=1}^j\;
\frac{\|g^{(\al_i)}(x)\|_{\Lin^{\al_i}(E;F)}}{\al_i!}
\end{split}
\end{align}
Taking the supremum over $x \in U$, we deduce 
\begin{align*}
  \|f\o g\|^{(k)}_{U} \le k! \sum_j \sum_\al \frac{\|f\|^{(j)}_{V}}{j!} \prod_i \frac{\|g\|^{(\al_i)}_{U}}{\al_i!} < \infty 
\end{align*}
for each $k \ge 1$. For $k=0$ we have 
\[
  \|f\o g\|^{(0)}_{U} \le \|f\|^{(0)}_{V} < \infty.
\]

\paragraph{\bf Case $\cA = \BM$} 
By \eqref{eq:Faa} and by \thetag{\ref{ssec:ws}.\ref{eq:FdB}}, 
we find    
\begin{align*} 
\begin{split}
   &\frac{\|(f\o g)^{(k)}(x)\|_{\Lin^k(E;G)}}{k!M_k} \le \\
&\le \sum_{j\ge 1} M_1^j \!\!\!\!\sum_{\substack{\al\in \N_{>0}^j\\ \al_1+\dots+\al_j =k}}
\frac{\|f^{(j)}(g(x))\|_{\Lin^j(F;G)}}{j!M_j}\;\prod_{i=1}^j\;
\frac{\|g^{(\al_i)}(x)\|_{\Lin^{\al_i}(E;F)}}{\al_i!M_{\al_i}}
\\ 
&\le  M_1 C_f C_g \rh_f \rh_g^k  \sum_{j\ge 1} \binom{k-1}{j-1} (M_1  \rh_f C_g)^{j-1}  
= M_1 C_f C_g \rh_f \rh_g^k (1 + M_1  \rh_f C_g)^{k-1}.
\end{split} 
\end{align*}
This clearly implies the Roumieu case $\cA=\BrM$.
For the Beurling case $\cA=\BbM$,
given $\rh>0$, let $\si>0$ be such that $\rh = \sqrt \si + \si$ and set $\rh_g := \sqrt \si$ and 
$\rh_f := (C_g M_1)^{-1} \sqrt \si$. Then $\|f \o g\|^M_{U,\rh}< \infty$. 
\end{proof}

\begin{theorem} \label{thm:Wcomp2}
  Let $M=(M_k)$ be log-convex. 
  If $g : \R^n \to \R^n$ is a $C^\infty$-diffeomorphism satisfying $\inf_{x\in \R^n} 
  |\det dg(x)|>0$, then we have the following implications: 
  \begin{align*}
    f \in  \Sp(\R^n,\R^n),~ g \in \cB_2(\R^n,\R^n) &\Longrightarrow  f \o g \in \Sp(\R^n,\R^n), \\
    f \in  \SpM(\R^n,\R^n),~ g \in \BM_2(\R^n,\R^n) &\Longrightarrow  f \o g \in \SpM(\R^n,\R^n). 
  \end{align*}
\end{theorem}

\begin{proof}
  By assumption,
  \begin{align} \label{eq:Wcomp2}
    \int_{\R^n} |(\p^\al f)(g(x))|^p\, dx = \int_{\R^n} |(\p^\al f)(y)|^p\, \frac{dy}{|\det dg(g^{-1}(y))|} 
    \le C \int_{\R^n} |(\p^\al f)(y)|^p\, dy, 
  \end{align}
  for a constant $C$ depending only on $g$. 

  Assume that $f \in \Sp$ and $g \in \cB_2$.
  By Fa\`a di Bruno's formula (\ref{ssec:Faa}.\ref{eq:FaaBM}), 
  each partial derivative $\p^\ga(f\o g)$ of $f \o g$ with $|\ga| \ge 1$ is $p$-integrable over $\R^n$,
  using \eqref{eq:Wcomp2} and the fact that $L^p(\R^n)$ is a $L^\infty(\R^n)$-module.

  Now assume that $f \in \SpM$ and $g \in \BM_2$.   
  We use Fa\`a di Bruno's formula (\ref{ssec:Faa}.\ref{eq:FaaBM}) in order to see that 
  \begin{equation} \label{eq:Wcomp3}
    \Box \rh>0 ~\exists D>0 ~\forall |\ga| \ge 1 : \|\p^\ga (f \o g)\|_{L^p(\R^n)} \le D \rh^{|\ga|}\, |\ga|!\, M_{|\ga|}. 
  \end{equation}
  Indeed, using \eqref{eq:Wcomp2} and $|\al|!/\al! \le n^{|\al|}$, we may infer from (\ref{ssec:Faa}.\ref{eq:FaaBM})
  that
  \begin{align*}
    \frac{\|\p^\ga(f\o g)\|_{L^p(\R^n)}}{\ga!} &\le C \sum \frac{\al!}{k_1!\cdots k_\ell!}\, \frac{\|\p^\al f\|_{L^p(\R^n)}}{\al!} \\
    &\hspace{3cm} \times
    \Big(\frac{\|\p^{\de_1}g\|_{L^\infty(\R^n)}}{\de_1!}\Big)^{k_1} 
    \cdots \Big(\frac{\|\p^{\de_\ell}g\|_{L^\infty(\R^n)}}{\de_\ell!}\Big)^{k_\ell}\\
    &\le    
    C \sum \frac{\al!}{k_1!\cdots k_\ell!}\, C_f (n C_g \rh_f)^{|\al|} (n\rh_g)^{|\ga|} 
    \, M_{|\al|} M_{|\de_1|}^{|k_1|} \cdots M_{|\de_\ell|}^{|k_\ell|}
  \end{align*}
  and we may conclude \eqref{eq:Wcomp3}, by Lemma \ref{ssec:Faa}.  
  This implies that $f \o g \in \SpM(\R^n,\R^n)$, since we already know that $f \o g \in L^p(\R^n)$.
\end{proof}

\begin{theorem} \label{thm:Scomp2}
  Let $M=(M_k)$ be log-convex.  
  If $g : \R^n \to \R^n$ is a $C^\infty$-diffeomorphism satisfying $\lim_{|x|\to\oo}g(x)-x=0$, we have the following implications: 
  \begin{align*}
    f \in  \cS(\R^n,\R^n),~ g \in \cB_2(\R^n,\R^n) &\Longrightarrow  f \o g \in \cS(\R^n,\R^n), \\
    f \in  \SLM(\R^n,\R^n),~ g \in \BM_2(\R^n,\R^n) &\Longrightarrow  f \o g \in \SLM(\R^n,\R^n). 
  \end{align*}
\end{theorem}

\begin{proof} 
  By assumption, there is a constant $C>0$ so that 
  \[
    \frac{1+|x|}{1+|g(x)|} \le C
  \]
  for all $x$. Thus, if $f \in \cS$, then 
  \begin{equation} \label{eq:Scomp4}
    (1+ |x|)^p |(\p^\al f)(g(x))| \le C^p (1+ |g(x)|)^p |(\p^\al f)(g(x))|
  \end{equation}
  is bounded in $x$ for all $p$ and $\al$.
  Hence, if $g \in B_2$, Fa\`a di Bruno's formula (\ref{ssec:Faa}.\ref{eq:FaaF}) implies that $f \o g \in S$.

  Now assume that $f \in \SLM$ and $g \in \BM_2$.   
  We use Fa\`a di Bruno's formula (\ref{ssec:Faa}.\ref{eq:FaaBM}) in order to see that 
  \begin{align} \label{eq:Scomp3}
    \Box \rh>0 ~\exists D>0 &~\forall p \in \N ~\forall |\ga| \ge 1 ~\forall x \in \R^n :\\ \nonumber
    &(1+|x|)^p |\p^\ga (f \o g)(x)| \le D \rh^{p+|\ga|}\, p!|\ga|!\,L_p M_{|\ga|}. 
  \end{align}
  Indeed, using \eqref{eq:Scomp4} we may infer from (\ref{ssec:Faa}.\ref{eq:FaaBM}) that 
  \begin{align*}
    &\frac{(1+|x|)^p|\p^\ga(f\o g)(x)|}{p! \ga!} \\
    &\le C^p \sum \frac{\al!}{k_1!\cdots k_\ell!}\, \frac{(1+|g(x)|)^p|(\p^\al f)(g(x))|}{p!\al!} 
    \Big(\frac{|\p^{\de_1}g(x)|}{\de_1!}\Big)^{k_1} 
    \cdots \Big(\frac{|\p^{\de_\ell}g(x)|}{\de_\ell!}\Big)^{k_\ell}\\
    &\le    
    (C\rh_f)^p L_p  \sum \frac{\al!}{k_1!\cdots k_\ell!}\, C_f (n C_g \rh_f)^{|\al|} (n\rh_g)^{|\ga|} 
    \,M_{|\al|} M_{|\de_1|}^{|k_1|} \cdots M_{|\de_\ell|}^{|k_\ell|}
  \end{align*}
  and we may conclude \eqref{eq:Scomp3}, by Lemma \ref{ssec:Faa}. 
  Thus $f \o g \in \SLM(\R^n,\R^n)$, since we already know that $f \o g \in \cS(\R^n,\R^n)$.
\end{proof}

\section{\texorpdfstring{$\CM$}{CM}-plots in spaces of test functions} \label{sec:plots}

In this section we characterize $\CM$-plots in $\BM(\R^n)$, $\SpM(\R^n)$, $\SLM(\R^n)$, and $\DM(\R^n)$.
We assume that $M=(M_k)$ satisfies Hypothesis \ref{ssec:localDC} and that $L=(L_k)$ satisfies $L_k\ge 1$ for all $k$.

Since $\BrM(\R^n)$, $\SprM(\R^n)$, $\SrLM(\R^n)$ are compactly regular (LB)-spaces, by Lemma~\ref{lem:cpreg}, 
and $\DrM(\R^n)$ is even a Silva space, the respective dual spaces 
can be equipped with the Baire topology of a countable limit of Banach spaces. 
Hence the sets of $\CM_b$-plots and of $\CM$-plots in each of these spaces coincides, by Lemma (\ref{ssec:localDC}.\ref{eq:set}).


\begin{lemma} \label{lem:key}
  Let $E$ be a Banach space, and let $U \subseteq E$ be open. 
  For a function $f \in \CM(U \times \R^n)$ consider the following conditions:
  \begin{align}\label{cond:B}
    \tag{$C\cB$} &\forall K \in \sK(U) ~ \Box \rh>0 : 
    \sup_{\substack{k\in \N, \al \in \N^n\\ (u,x) \in K \times \R^n\\ \|v_j\|_E \le 1}} 
    \frac{|\p_u^k\p_x^{\al} f(u,x)(v_1,\dots,v_k)|}{\rh^{k+|\al|}\,  (k+|\al|)!\, M_{k+|\al|}} < \infty.
    \\
    \label{cond:W}\tag{$CW$} &\forall K \in \sK(U) ~ \Box \rh>0 : 
    \sup_{\substack{k \in \N,\al \in \N^n\\ u \in K\\  \|v_j\|_E\le 1}}  
    \frac{\big(\int_{\R^n} |\p_u^k\p_x^{\al} f(u,x)(v_1,\dots,v_k)|^p\, dx\big)^{1/p}}{\rh^{k+|\al|}\, (k+|\al|)!\, M_{k+|\al|}}< \infty.
    \\
    \label{cond:S}\tag{$C\cS$} &\forall K \in \sK(U) ~ \Box \rh>0 : 
    \sup_{\substack{k,m \in \N, \al \in \N^n\\ (u,x) \in K \times \R^n\\ \|v_j\|_E \le 1}} 
    \frac{(1+|x|)^m|\p_u^k\p_x^{\al} f(u,x)(v_1,\dots,v_k)|}
    {\rh^{m+k+|\al|}\,  m! (k+|\al|)! \, L_m M_{k+|\al|}}
    < \infty. 
    \\
    \label{cond:D}\tag{$C\cD$}  &\forall K \in \sK(U) ~\exists L \in \sK(\R^n) ~\forall u \in K :
       \on{supp} f(u,~) \subseteq L.
  \end{align}
  Then
  \begin{align*}
    \Cb^{[M]}(U,\BM(\R^n)) &= C^{[M]}(U, \BM(\R^n)) \\
    &=
    \big\{f^\vee : f\in C^{[M]}(U \times \R^n) \text{ satisfies } \eqref{cond:B}\big\},\\ 
    \Cb^{[M]}(U, \SpM(\R^n)) &=  C^{[M]}(U, \SpM(\R^n)) \\
    &= 
    \big\{f^\vee : f\in C^{[M]}(U \times \R^n) \text{ satisfies } \eqref{cond:W}\big\},\\
    \Cb^{[M]}(U, \SLM(\R^n)) &=  \CM(U, \SLM(\R^n)) \\
    &= \big\{f^\vee : f\in \CM(U \times \R^n) \text{ satisfies } \eqref{cond:S}\big\},\\
    \Cb^{[M]}(U, \DM(\R^n)) &= \CM(U, \DM(\R^n)) \\
    &= \big\{f^\vee : f\in C^{[M]}(U \times \R^n) \text{ satisfies } \eqref{cond:D}\big\}.
  \end{align*}
\end{lemma}

\begin{proof}
Let $\cA \in \{\BM,\SpM,\SLM,\DM\}$.
If $f^\vee \in \CM(U, \cA(\R^n))$, then $f \in \CM(U \times \R^n)$, 
by the $C^{[M]}$-exponential law (Theorem \ref{ssec:localDC}) and by Proposition \ref{prop:incl}, and 
conversely, if 
$f \in C^{[M]}(U \times \R^n)$, then $f^\vee \in C^{[M]}(U,C^{[M]}(\R^n))$.

\paragraph{\bf Case $\cA \in \{\BM,\SpM,\SLM\}$}
If $f^\vee \in \CM_b(U, \cA(\R^n))$ then 
for each $K \in \sK(U)$ $\Box \rh_1,\rh_2>0$ the set $\set^M_{K,\rh_1}(f^\vee)$ is bounded in 
$\cB^M_{\rh_2}(\R^n)$, $W^{M,p}_{\rh_2}(\R^n)$, or $\cS_{L,\rh_2}^M(\R^n)$, respectively, 
since $\BrM(\R^n)$, $\SprM(\R^n)$, $\SrLM(\R^n)$ are compactly regular (LB)-space, by Lemma~\ref{lem:cpreg}.
That is   
\begin{align} \label{eq:keyB}
  &\sup_{\substack{k \in \N, \al \in \N^n\\ (u,x) \in K \times \R^n\\ \|v_j\|_E \le 1}} 
  \frac{|\p_u^k\p_x^{\al} f(u,x)(v_1,\dots,v_k)|}{\rh_1^{k}\,\rh_2^{|\al|}\,  k!\,|\al|!\, M_{k} M_{|\al|}}
  < \infty,
  \\
  \label{eq:keyW}&\sup_{\substack{k \in \N,\al \in \N^n\\ u \in K\\  \|v_j\|_E\le 1}} 
  \frac{\big(\int_{\R^n} |\p_u^k\p_x^{\al} f(u,x)(v_1,\dots,v_k)|^p\, dx\big)^{1/p}}{\rh_1^{k}\,\rh_2^{|\al|}\,  k!\,|\al|!\, M_{k} M_{|\al|}}
  < \infty, \quad \text{ or }
  \\
  \label{eq:keyS}&\sup_{\substack{k,m \in \N, \al \in \N^n\\ (u,x) \in K \times \R^n\\ \|v_j\|_E \le 1}} 
  \frac{(1+|x|)^m|\p_u^k\p_x^{\al} f(u,x)(v_1,\dots,v_k)|}
  {\rh_1^{k}\,\rh_2^{m+|\al|}\,  k!\,m!\,|\al|!\, L_m M_{k} M_{|\al|}}
  < \infty,
\end{align}
which implies \eqref{cond:B}, \eqref{cond:W}, or \eqref{cond:S}, respectively. 

Conversely, suppose that $f$ satisfies \eqref{cond:B}, \eqref{cond:W}, or \eqref{cond:S}. 
It follows that for each $K \in \sK(U)$  $\Box \rh_1,\rh_2>0$ so that the supremum 
in \eqref{eq:keyB}, \eqref{eq:keyW}, or \eqref{eq:keyS} is finite, 
since $M=(M_k)$ has moderate growth (\ref{ssec:ws}.\ref{eq:mg0}). 

\paragraph{\bf Case $\cA =\DM$}
First we show that \eqref{cond:D} is equivalent to the following condition.
\begin{gather} \label{cond:D'}\tag{$C\cD'$}
  \begin{split}
    \forall K \in \sK(U) ~\exists L \in \sK(\R^n) ~\forall u \in K :
     \on{supp} f(u,~) \subseteq L, \quad \text{ and } \\
     \Box \rh>0 ~\forall \al \in \N^n : \sup_{\substack{k \in \N\\(u,x)\in K \times L}} 
      \frac{\|\p_u^k \p_x^\al f(u,x)\|_{\Lin^{k}(E;\R)}}{\rh^k\,k!\, M_k} < \infty.
  \end{split}
\end{gather}
Let us prove the non-trivial direction \eqref{cond:D} $\Rightarrow$ \eqref{cond:D'}. 
Fix $K \in \sK(U)$.
Then there exists $L \in \sK(\R^n)$ so that $\on{supp} f(u,~) \subseteq L$ for all $u \in K$.  
Since $f \in C^{[M]}(U \times \R^n)$, $\Box \rh >0$ $\exists C>0$ so that, for all $k \in \N$, $\al \in \N^n$, $(u,x) \in K \times L$,
\begin{align*}
  \|\p_u^k \p_x^\al f(u,x)\|_{\Lin^{k}(E;\R)} &\le C \rh^{k+|\al|} (k+|\al|)!\, M_{k+|\al|}
   \le C \rh_1^{k+|\al|} k!\, |\al|!\, M_{k}  M_{|\al|}, 
\end{align*}
for $\Box \rh_1>0$, as $M$ has moderate growth (\ref{ssec:ws}.\ref{eq:mg0}), thus \eqref{cond:D'}.

Now we prove the assertions of the lemma.  
If $f^\vee \in \Cb^{[M]}(U, \DM(\R^n))$, then 
$\forall K \in \sK(U)$ $\Box \rh>0$ the set  
$\set^M_{K,\rh}(f^\vee)$ 
is bounded in $\cD^{[M]}(\R^n)$. There exists $L \in \sK(\R^n)$ 
so that $\set^M_{K,\rh}(f^\vee)$ is bounded in $\cD(L)$. Hence $f$ satisfies \eqref{cond:D}.

Suppose that $f$ satisfies \eqref{cond:D} and thus \eqref{cond:D'}.
We will show that $f^\vee \in C^{[M]}_b(U, \cD(\R^n))$, i.e.,
$\forall K \in \cK(U)$ $\Box \rh>0$ the set $\set^M_{K,\rh}(f^\vee)$ is bounded in $\cD(\R^n)$.
Condition \eqref{cond:D'} guarantees, for given $K \in \cK(U)$,
the existence of $L \in \sK(\R^n)$ and $\Box \rh>0$ so that for all $\al \in \N^n$ the set  
\[
  \Big\{\p_x^\al \Big(\frac{(f^\vee)^{(k)}(u)}{\rh^k\,k!\, M_k}\Big)(x) = \frac{\p_u^k \p_x^\al f(u,x)}{\rh^k\,k!\, M_k} : 
k\in \mathbb N,(u,x)\in K \times L\Big\}
\]
is bounded. That is, $\set^M_{K,\rh}(f^\vee)$ is bounded in $\cD(L)$ and hence also in $\cD(\R^n)$.  
\end{proof}

\section{Groups of diffeomorphisms on \texorpdfstring{$\R^n$}{Rn}} \label{sec:group}

In this section we assume that $M=(M_k)$ satisfies Hypothesis \ref{ssec:localDC} and that $L=(L_k)$ satisfies $L_k\ge 1$ for all $k$.

We define
\begin{align*}
  \DiffB &:= \big\{F=\Id+f: f\in \cB(\R^n,\R^n), 
    \inf_{x\in \R^n} \det(\mathbb I_n + df(x)) >0\big\}\\
    \DiffSp &:= \big\{F=\Id+f: f\in \Sp(\R^n,\R^n), 
    \det(\mathbb I_n + df) >0\big\}\\
    \DiffS&:= \big\{F=\Id+f: f\in \cS(\R^n,\R^n), 
    \det(\mathbb I_n + df) >0\big\}\\
    \DiffD &:= \big\{F=\Id+f: f\in \cD(\R^n,\R^n), 
    \det(\mathbb I_n + df) >0\big\}
\end{align*}
and the ultradifferentiable versions
\begin{empheq}[box=\widefbox]{align*}
  \;\DiffBM &:= \big\{F=\Id+f: f\in \BM(\R^n,\R^n), 
    \inf_{x\in \R^n} \det(\mathbb I_n + df(x)) >0\big\}\\
    \DiffSpM &:= \big\{F=\Id+f: f\in \SpM(\R^n,\R^n), 
    \det(\mathbb I_n + df) >0\big\}\\
    \DiffSLM &:= \big\{F=\Id+f: f\in \SLM(\R^n,\R^n), 
    \det(\mathbb I_n + df) >0\big\}\\
    \DiffDM &:= \big\{F=\Id+f: f\in \DM(\R^n,\R^n), 
    \det(\mathbb I_n + df) >0\big\} 
\end{empheq}
Let $\cA \in \{\cB,\Sp,\cS,\cD,\BM,\SpM,\SLM,\DM\}$.
Then $\DiffA$ is a manifold  modelled on the convenient vector space $\cA(\R^n,\R^n)$ with 
global chart $\{f\in \cA(\R^n,\R^n):\inf_{x\in \R^n} \det(\mathbb I_n + df(x)) >0\big\}\ni f 
\mapsto \Id +f$.

Note that for $\cA\in \{\Sp,\cS,\cD,\SpM,\SLM,\DM\}$ the condition $\det(\mathbb I_n + df) >0$ 
implies $\inf_{x\in \R^n} \det(\mathbb I_n + df(x)) >0$, 
since a function $f \in \Sp(\R^n,\R^n)$ tends to $0$ at $\infty$ together with all its partial derivatives.

\begin{theorem} \label{thm:MichorMumford}
  $\DiffB$, $\DiffSp$, $\DiffS$, and $\DiffD$ are $C^\infty$-regular Lie groups.
\end{theorem} 

\begin{proof}
  All this was proved in \cite{MichorMumford13} with one exception: for $\DiffSp$ only the 
  case $p=2$ was considered. For $p \ne 2$ the proof is just the same.
\end{proof}

Our goal is to show a corresponding result for the above groups of ultradifferentiable diffeomorphisms, i.e., 
Theorem \ref{thm:main}.
From now on we treat only $\cA\in \{\BM,\SpM,\SLM,\DM\}$.

\begin{lemma} \label{lem:diff} 
  Each element of 
  $\DiffBM$, $\DiffSpM$, $\DiffSLM$, or $\DiffDM$ is a $C^{[M]}$-diffeomorphism of $\R^n$.
\end{lemma}

\begin{proof}
  An element $F$ in any of the sets in question is also an element of $\DiffB$, and thus 
  a $C^\infty$-diffeomorphism of $\R^n$, by Theorem \ref{thm:MichorMumford} (see \cite[Proof of 3.8]{MichorMumford13}). 
  The $C^{[M]}$ inverse function theorem implies that $F$ is a 
  $C^{[M]}$-diffeomorphism. 
\end{proof}

%
%

Before we show that $\DiffA$ are groups with respect to composition, let us 
state two lemmas.

\begin{lemma}[{\cite[p.~201]{Yamanaka89}}] \label{lem:Yam}
  Let $M=(M_k)$ be log-convex and let $A,C,\rh>0$. For $N \in \N_{\ge 2}$ define
  \[
  \ps_N(t) := C A \sum_{j=2}^N \frac{\rh^{j-1} M_{j-1}}{j} t^j, \quad t \in \R.
  \]
  Then there exist $c_i$ so that the function $g_N(s) = \sum_{i=1}^\infty c_i s^i$ satisfies 
  \begin{equation} \label{eq:g_N}
    g_N(s)=  A s+ \ps_N(g_N(s)), \quad \text{ for small } s \in \R,
  \end{equation}
  and 
  \begin{align*}
    0 < i c_i &< 
    A (4A(CA+1) \rh)^{i-1} M_{i-1}, \quad \text{ for } i =2,\ldots,N. 
  \end{align*}
\end{lemma}

\medskip

\begin{lemma} \label{lem:LA}
  Let $A : \R^n \to \R^n$ be a linear invertible mapping. 
  Then $\|A^{-1}\| \le |\det A|^{-1} \|A\|^{n-1}$.
\end{lemma}

\begin{proof}
  Using the polar decomposition of the matrix $A$ we have $A=UP$ for an orthogonal matrix $U$ and 
  a positive semidefinite Hermitian matrix $P = V \on{diag}(s_1,\ldots,s_n) V^*$. 
  Then $\|A^{-1}\| = \|P^{-1} U^{-1}\| = \|P^{-1}\| = (\min_i s_i)^{-1}$ 
  and $|\det A| = \det P = s_1 s_2 \cdots s_n \le (\min_i s_i) (\max_i s_i)^{n-1} = \|A^{-1}\|^{-1} \|A\|^{n-1}$.
\end{proof}

\begin{proposition} \label{prop:Bgroup}
  $\DiffBM$, $\DiffSpM$, $\DiffSLM$ and $\DiffDM$ are groups with respect to composition.
\end{proposition}

\begin{proof}
  Let $\cA \in \{\BM,\SpM,\SLM,\DM\}$.

  \begin{claim} \label{claim:a} 
    If $F = \Id+f$ and $G= \Id +g$ are elements of $\DiffA$, then so is $F \o G$.
  \end{claim}

  We have  
  \begin{align} \label{eq:group1}
  \begin{split}
    &((\Id+f)\o(\Id+g))(x)=x+g(x)+f(x+g(x)), \quad \text{ and } \\
    &\hspace{0.5cm} \inf_{x \in \R^n} \det ((\I_n + df(x+g(x)))(\I_n +dg(x)))>0.   
  \end{split}
  \end{align}
  We must check that $h(x):= f(x+g(x))$ is
  in $\cA(\R^n,\R^n)$ if $f,g\in\cA(\R^n,\R^n)$.

  This follows from Theorem~\ref{thm:Bcomp} for $\cA = \BM$,  
  from Theorem \ref{thm:Wcomp2} for $\cA = \SpM$, and from Theorem \ref{thm:Scomp2} for $\cA=\SLM$, 
  since $x \mapsto x + g(x)$ is in $\BM_2(\R^n,\R^n)$,
  by Proposition \ref{prop:incl}.
  In the Case $\cA = \DM$,
  $h \in \CM(\R^n,\R^n)$ clearly has compact support.
  The proof of Claim \ref{claim:a} is complete.

  \begin{claim} \label{claim:b}
    If $F = \Id+f\in \DiffA$, then $G = \Id + g := F^{-1} \in \DiffA$.    
  \end{claim}

  By Lemma \ref{lem:diff}, $F^{-1}$ exists as an element of $\CM(\R^n,\R^n)$. 
  The identity
  \begin{equation} \label{eq:group2}
    (\Id+g)\o(\Id+f)=\Id  \quad\Longleftrightarrow\quad f(x)+g(x+f(x))=0
  \end{equation}
  implies, together with Lemma \ref{lem:LA} (and using $\|A\|^{-1} \le \|A^{-1}\|$),
  \[
  \det (\I_n + dg(x+f(x))) = \det (\I_n+df(x))^{-1} \ge \|\I_n + df(x)\|^{-n}
  \]
  which is bounded away from $0$. 
  It remains to show that $g \in \cA(\R^n,\R^n)$.

  To this end we fix $a \in \R^n$ and set $b=F(a)$ and $T=F'(a)^{-1} = G'(b)$. Defining 
  \begin{equation} \label{eq:ph}
    \ph := \Id - T \o F,  
  \end{equation}
  we have 
  \begin{equation} \label{eq:G}
    G = T + \ph \o G.
  \end{equation}

  \paragraph{\bf Case $\cA = \BrM$}
  This proof is inspired by \cite{Yamanaka89}. 
  Since $f \in \BrM(\R^n,\R^n)$ and since $M=(M_k)$ is derivation closed, there exist constants $C,\rh>0$ 
  so that 
  \[
    \|f^{(k)}(x)\|_{\Lin^k(\R^n;\R^n)} \le C \rh^{k-1} (k-1)!\, M_{k-1} \quad \text{ for } x \in \R^n, k \in \N_{\ge 1}
  \]
  and thus 
  \[
    \|\ph^{(k)}(x)\|_{\Lin^k(\R^n;\R^n)} \le C \|T\|_{\Lin(\R^n;\R^n)} \rh^{k-1} (k-1)!\, M_{k-1} \quad \text{ for } x \in \R^n, 
    k \in \N_{\ge 2}.
  \]
  Since $\inf_{x \in \R^n} \det (\I_n + df(x)) >0$ and by Lemma \ref{lem:LA}, 
  \[
    A:=\sup_{x \in \R^n} \|F'(x)^{-1}\|_{\Lin(\R^n;\R^n)} = \sup_{x \in \R^n} \|(\I_n + df(x))^{-1}\|_{\Lin(\R^n;\R^n)} < \infty.
  \]
  Define $\ps_N$ and $g_N$ as in Lemma~\ref{lem:Yam}. Then
  \begin{gather*}
  \|G'(b)\|_{\Lin(\R^n;\R^n)} = \|F'(a)^{-1}\|_{\Lin(\R^n;\R^n)} = \|T\|_{\Lin(\R^n;\R^n)} \le A = g_N'(0) \quad \text{ and } \\
  \|\ph^{(k)}(a)\|_{\Lin^k(\R^n;\R^n)} \le \ps_N^{(k)}(0),  \quad 2 \le k \le N.
  \end{gather*}
  Applying Fa\`a di Bruno's formula (\ref{ssec:Faa}.\ref{eq:FaaF}) to \eqref{eq:G} and to \thetag{\ref{lem:Yam}.\ref{eq:g_N}} we can deduce inductively that $\|G^{(k)}(b)\|_{\Lin^k(\R^n;\R^n)} \le g_N^{(k)}(0)$, for $2 \le k \le N$, in fact, 
  \begin{align} \label{eq:yamanaka}
  \begin{split}
    \|G^{(k)}(b)\|_{\Lin^k(\R^n;\R^n)} 
    &\le  
    \sum_{j\ge 2} \sum_{\substack{\al\in \N_{>0}^j\\ \sum \al_i=k}}
    \frac{k!}{j!\al!} \|\ph^{(j)}(a)\|_{\Lin^j(\R^n;\R^n)} 
    \prod_{i=1}^j \|G^{(\al_i)}(b)\|_{\Lin^{\al_i}(\R^n;\R^n)}
    \\
    &\le  
    \sum_{j\ge 2} \sum_{\substack{\al\in \N_{>0}^j\\ \al_1+\dots+\al_j =k}}
    \frac{k!}{j!\al!} \ps_N^{(j)}(0) 
    \prod_{i=1}^j g_N^{(\al_i)}(0) 
    \\
    &= g_N^{(k)}(0);
  \end{split}
  \end{align}
  note that $\ph'(a)=0$ and $\ps_N'(0)=0$.
  Thus
  \begin{align*}
  \|g^{(k)}(b)\|_{\Lin^k(\R^n;\R^n)} &= 
  \|G^{(k)}(b)\|_{\Lin^k(\R^n;\R^n)} \le g_N^{(k)}(0) = k!\, c_k \\ &< A (4A(CA+1) \rh)^{k-1} (k-1)!\, M_{k-1}
  \\ &\le \frac{1}{4 M_1(CA+1) \rh} (4A(CA+1) \rh)^{k} k!\, M_{k},
  \end{align*}
  As $N$ was arbitrary,  
  we have $g \in \BrM(\R^n,\R^n)$; that $g$ and $g^{(1)}$ are globally bounded follows e.g.\ from Theorem \ref{thm:MichorMumford}. 

  \paragraph{\bf Case $\cA=\BbM$}
  Consider
  \[
  L_k := \frac{1}{k!} \sup_{x \in \R^n} \|f^{(k)}(x)\|_{\Lin^k(\R^n;\R^n)}.
  \]
  Then $L \lhd M$ and since $M_{k+1}/M_k \to \infty$
  there exists a log-convex sequence $N=(N_k)$ satisfying $N_{k+1}/N_k \to \infty$
  and such that $L \le N \lhd M$, 
  by \cite[Lemma~6]{Komatsu79b}; the proof of \cite[Lemma~6]{Komatsu79b} shows that, if $M=(M_k)$ is derivation closed, 
  then we may find a derivation closed $N=(N_k)$ with the above properties.
  Thus, $f \in \cB^{\{N\}}(\R^n,\R^n)$ and, by the Roumieu case, $g \in \cB^{\{N\}}(\R^n,\R^n) \subseteq \BbM(\R^n,\R^n)$.  

  Note that in this step of the proof $N=(N_k)$ need not have moderate growth. 

  \paragraph{\bf Case $\cA=\SpM$}
  Since $f =(f_1,\ldots,f_n) \in \SpM(\R^n,\R^n)$, we have $\Box\rh>0$ $\exists C>0$ 
  so that 
  \[
    \|f_i^{(\al)}\|_{L^p(\R^n)} \le C \rh^{|\al|} |\al|!\, M_{|\al|} \quad \text{ for }  \al \in \N^n, 
    ~i=1,\ldots,n,
  \]
  and thus, as $\ph^{(\al)}(x) = -T(f^{(\al)}(x))$ if $|\al|\ge2$, by \eqref{eq:ph},
  \[
    \|\ph_i^{(\al)}\|_{L^p(\R^n)} \le D \rh^{|\al|} |\al|!\, M_{|\al|} \quad \text{ for } \al \in \N^n, ~|\al|\ge2, 
    ~i=1,\ldots,n,
  \]
  for $D:=C \|T\|_{\Lin(\R^n;\R^n)}$.
  We know from Proposition \ref{prop:incl} and from Case $\cA = \BM$ above that $g \in \BM(\R^n,\R^n)$ and 
  hence $G \in \BM_2(\R^n,\R^n)$. Moreover, $G$ is a diffeomorphism on $\R^n$ satisfying $\inf_x \det d G(x)>0$.
  Thus, applying Theorem \ref{thm:Wcomp2} to \eqref{eq:G}, we may conclude that $g \in \SpM(\R^n,\R^n)$. 
  (Note that we know by Theorem \ref{thm:MichorMumford} that $g$ and all its iterated partial derivatives 
  are $p$-integrable, and so if the growth conditions defining $\SpM$ 
	and so if the growth conditions defining $W^{[M],p}$ hold from some order of derivation onwards,
	     then they also hold for all lower orders of derivation.)
	
  \paragraph{\bf Case $\cA=\SLM$}
  This is analogous to the Case $\cA=\SpM$.
  Here $f \in \SLM(\R^n,\R^n)$ implies $\Box \rh>0$ $\exists C>0$ 
  so that for all $p \in \N$, $\al \in \N^n$, $i=1,\ldots,n$, $x \in \R^n$,
  \[
    (1+|x|)^p|f_i^{(\al)}(x)| \le C \rh^{p+|\al|}\, p!|\al|!\, L_pM_{|\al|},
  \]
  and, as $\ph^{(\al)}(x) = -T(f^{(\al)}(x))$ if $|\al|\ge2$, 
  \[
    (1+|x|)^p|\ph_i^{(\al)}(x)| \le D \rh^{p+|\al|}\, p! |\al|!\, L_p M_{|\al|},
  \]
  for all $p \in \N$, $\al \in \N^n$, $|\al|\ge2$, $i=1,\ldots,n$, $x \in \R^n$, and 
  $D:=C \|T\|_{\Lin(\R^n;\R^n)}$.
  Again Proposition \ref{prop:incl} and Case $\cA = \BM$ imply $g \in \BM(\R^n,\R^n)$ and 
  hence $G \in \BM_2(\R^n,\R^n)$. Moreover, $G= \Id +g$ is a diffeomorphism on $\R^n$ satisfying 
  $G(x) - x \to 0$ as $|x| \to \infty$, since we already know that $g \in \cS$, by Theorem \ref{thm:MichorMumford}.
  Thus, applying Theorem \ref{thm:Scomp2} to \eqref{eq:G}, we may conclude that $g \in \SLM(\R^n,\R^n)$.

  \paragraph{\bf Case $\cA=\DM$}
  By \eqref{eq:group2}, we have 
  $\on{supp} g = \on{supp} f$, and thus $g \in \DM(\R^n,\R^n)$.

  This finishes the proof of Claim \ref{claim:b} and thus of the proposition.
\end{proof}

\section{Composition} \label{sec:comp}

We assume that $M=(M_k)$ satisfies Hypothesis \ref{ssec:localDC} and that $L=(L_k)$ satisfies $L_k\ge 1$ for all $k$.

For simplicity we shall employ the following notational convention in this section.

\begin{convention} \label{convention}
We write $\p^\ga = \p_1^{\ga_0}\p_2^{\ga'}$ for multiindices $\ga=(\ga_0,\ga') = (\ga_0,\ga_1,\ldots,\ga_n) \in \N^{1+n}$, where
$\p_1^{\ga_0}$ are iterated total derivatives with respect to $u$ in a Banach space $E$ and $\p_2^{\ga'}$ are partial derivatives 
with respect to $x \in \R^n$.
\end{convention}

\begin{theorem} \label{thm:compDiff}
  Composition is $\CM$ on $\DiffBM$, $\DiffSpM$, $\DiffSLM$ and $\DiffDM$, i.e.,
  \begin{align*}
    \on{comp} &: \DiffBM \times \DiffBM \to \DiffBM, \\
    \on{comp} &: \DiffSpM \times \DiffSpM \to \DiffSpM, \\
    \on{comp} &: \DiffSLM \times \DiffSLM \to \DiffSLM, \\
    \on{comp} &: \DiffDM \times \DiffDM \to \DiffDM, \\
    &\quad \on{comp}(F,G):= F \o G,
  \end{align*}
  are $\CM$-mappings.
\end{theorem}

\begin{proof}
It suffices to prove that $\on{comp}$ maps 
$\CMb$-plots to $\CM$-plots; see Subsection \ref{ssec:localDC}. 
By Lemma~\ref{lem:key}, it is enough 
to consider $U \ni u \mapsto \Id+f(u,~)$ and $U \ni u \mapsto \Id+g(u,~)$, where 
$U$ is open in some Banach space $E$ and $f,g \in \CM(U\times \R^n,\R^n)$ satisfy \eqref{cond:B}, \eqref{cond:W}, \eqref{cond:S}
or \eqref{cond:D}, 
and to check that $h(u,x) := f(u,x+g(u,x))$ (cf.\ (\ref{prop:Bgroup}.\ref{eq:group1}))
satisfies \eqref{cond:B}, \eqref{cond:W}, \eqref{cond:S}
or \eqref{cond:D}, respectively. 
That $h \in \CM(U\times \R^n,\R^n)$ follows from the fact that $\CM$ is stable under composition. 

Let us define $\ph=(\ph_1,\ph_2)$ by setting
\begin{align} \label{eq:compph}
  \begin{split}
    \ph_1(u,x) := u \quad &\text{ and } \quad \ph_2(u,x) := x + g(u,x) \quad \text{ such that } \\
    &h = f \o \ph.  
  \end{split}
\end{align}
Since $g=(g_1,\ldots,g_n)$ satisfies \eqref{cond:B}, \eqref{cond:W}, \eqref{cond:S}
or \eqref{cond:D}, we may conclude that $\forall K \in \sK(U)$
$\Box \si>0$ $\exists D>0$ such that, for all $u \in K$, and all $x \in \R^n$, 
\begin{align} \label{eq:Bestph}
\begin{split}
  \|\p^{\de} \ph_1(u,x)\|_{\Lin^{\de_0}(E;E)} &\le D \si^{|\de|}\, |\de|!\, M_{|\de|} \quad  
  \forall \de \in \N^{1+n}, 
  \\
  \|\p^{\de} \ph_2(u,x)\|_{\Lin^{\de_0}(E;\R^n)} &\le D \si^{|\de|}\, |\de|!\, M_{|\de|}   
  \\
  &\hspace{.5cm}
  \forall \de_0\in \N,\de' \in \N^n \setminus \{0\} 
   \text{ and } \forall \de_0\in \N_{\ge1},\de' \in \N^n,
\end{split}
\end{align}
where we apply Convention \ref{convention}.
Here we use (the proof of) Proposition \ref{prop:incl}; for instance, in the case $\cA= \SpM$, 
\[
  \sup_{\substack{\de \in \N^{1+n}\\x \in \R^n}} \frac{|\p^{\de} g_i(u,x)(v_1,\ldots,v_{\de_0})|}
  {\si^{|\de|}\, |\de|!\, M_{|\de|}}
  \le \tilde D \sup_{\substack{\de \in \N^{1+n}}} \frac{\|\p^{\de} g_i(u,~)(v_1,\ldots,v_{\de_0})\|_{L^p(\R^n)}}
  {\tilde \si^{|\de|}\, |\de|!\, M_{|\de|}}.
\]

\paragraph{\bf Case $\cA=\BM$}
We already know from Theorem \ref{thm:MichorMumford} that for all $K \in \sK(U)$, $\al \in \N^{1+n}$, 
\[
  \sup_{(u,x) \in K \times \R^n} \|\p^\al h(u,x)\|_{\Lin^{\al_0}(E;\R^n)} < \infty.
\]
Let $K \in \sK(U)$ be fixed. Then, since $f=(f_1,\ldots,f_n)$ satisfies \eqref{cond:B},  
we have 
$\Box \rh>0$ $\exists C>0$ such that, for all $i=1,\ldots,n$, all $u \in K$, and all $x \in \R^n$, 
\begin{align} \label{eq:fcompB}
  \|\p^{\al} f_i(u,x)\|_{\Lin^{\al_0}(E;\R)}  &\le C \rh^{|\al|}\, |\al|!\, M_{|\al|} \quad \forall \al \in \N^{1+n}.
\end{align}
Fa\`a di Bruno's formula (\ref{ssec:Faa}.\ref{eq:FaaBM}), \eqref{eq:Bestph}, and \eqref{eq:fcompB} then give, 
for all $i=1,\ldots,n$, $u \in K$, $x \in \R^n$, and 
$\ga \in \N^{1+n} \setminus \{0\}$, 
\begin{align*}
  &\frac{\|\p^{\ga} h_i(u,x)\|_{\Lin^{\ga_0}(E;\R)}}{\ga!}  \\
  &\le  \sum \frac{\al!}{k_1!\cdots k_\ell!}  
  \frac{\|\p^\al f_i(u,x+g(u,x))\|_{\Lin^{\al_0}(E;\R)}}{\al!} \\ 
  &\hspace{1cm} \times \Big( \frac{\|\p^{\de_{1}} \ph_1(u,x)\|_{\Lin^{\de_{10}}(E;E)}}{\de_{1}!}\Big)^{k_{11}} 
  \cdots \Big( \frac{\|\p^{\de_{\ell}} \ph_1(u,x)\|_{\Lin^{\de_{\ell0}}(E;E)}}{\de_{\ell}!}\Big)^{k_{\ell1}} \\
  &\hspace{1cm} \times \Big( \frac{\|\p^{\de_{1}} \ph_2(u,x)\|_{\Lin^{\de_{10}}(E;\R^n)}}{\de_{1}!}\Big)^{k_{12}} 
  \cdots \Big( \frac{\|\p^{\de_{\ell}} \ph_2(u,x)\|_{\Lin^{\de_{\ell0}}(E;\R^n)}}{\de_{\ell}!}\Big)^{k_{\ell2}}
  \\
  &\le    
    \sum \frac{\al!}{k_1!\cdots k_\ell!}\, C ((n+1) D \rh)^{|\al|} ((n+1)\si)^{|\ga|} 
    \, M_{|\al|} M_{|\de_1|}^{|k_1|} \cdots M_{|\de_\ell|}^{|k_\ell|},
\end{align*}
where $\al = k_1+\cdots +k_\ell$ and the sum is taken over all sets $\{\de_1,\ldots,\de_\ell\}$ of $\ell$ distinct 
elements in $\N^{1+n}\setminus \{0\}$ and all ordered $\ell$-tuples $(k_1,\ldots,k_\ell) \in (\N^{1+n}\setminus \{0\})^\ell$, 
$\ell = 1,2, \ldots$, such that $\ga = |k_1| \de_1 +\cdots + |k_\ell| \de_\ell$.
By Lemma \ref{ssec:Faa}, we conclude that $\Box \ta>0$ $\exists B>0$ such that
\[
  \frac{\|\p^{\ga}h(u,x)\|_{L^{\ga_0}(E;\R^n)}}{\ga!} \le B \ta^{|\ga|}\, M_{|\ga|}
  \quad \forall (u,x) \in K \times \R^n, \ga \in \N^{1+n} \setminus \{0\}.
\]
Thus $h$ satisfies \eqref{cond:B}.

\paragraph{\bf Case $\cA=\SpM$}
We already know from Theorem \ref{thm:MichorMumford} that for all $K \in \sK(U)$, $i=1,\ldots,n$, $\al \in \N^{1+n}$, 
\[
  \sup_{u \in K,\|v_j\|_E\le1} \int_{\R^n} |\p^\al h_i(u,x)(v_1,\ldots,v_{\al_0})|^p \, dx< \infty.
\]
Let $K \in \sK(U)$ be fixed. Then, since $f=(f_1,\ldots,f_n)$ satisfies \eqref{cond:W}, 
we have 
$\Box \rh>0$ $\exists C>0$ such that, for all $i=1,\ldots,n$, all $u \in K$, and all $v_j \in E$ with $\|v_j\|_E \le 1$,
\begin{align} \label{eq:fcompW}
  \|\p^{\al} f_i(u,~)(v_1,\ldots,v_{\al_0})\|_{L^p(\R^n)}  &\le C \rh^{|\al|}\, |\al|!\, M_{|\al|} \quad \forall  \al \in \N^{1+n}.
\end{align}
Since $x \mapsto \ph_2(u,x) = x+g(u,x)$ is a diffeomorphism on $\R^n$ satisfying $\inf_x \det d_x \ph_2(u,x)>0$, we have,  
for all $i$, $u$, and $v_j$, 
  \begin{align*} 
    &\int_{\R^n} |(\p^\al f_i)(u,\ph_2(u,x))(v_1,\ldots,v_{\al_0})|^p\, dx \\
    &= \int_{\R^n} |(\p^\al f_i)(u,y)(v_1,\ldots,v_{\al_0})|^p\, \frac{dy}{|\det d_x\ph_2(u,x)|} \\
    &\le C(g) \int_{\R^n} |(\p^\al f_i)(u,y)(v_1,\ldots,v_{\al_0})|^p\, dy; 
  \end{align*}
cf.\ \thetag{\ref{thm:Wcomp2}.\ref{eq:Wcomp2}}.
Fa\`a di Bruno's formula (\ref{ssec:Faa}.\ref{eq:FaaBM}), \eqref{eq:Bestph}, and \eqref{eq:fcompW} then give, 
for all $i=1,\ldots,n$, $u \in K$, $v_j \in E$ with $\|v_j\|_E \le 1$, and 
$\ga \in \N^{1+n} \setminus \{0\}$, 
\begin{align*}
  &\frac{\|\p^{\ga} h_i(u,~)(v_1,\ldots,v_{\ga_0})\|_{L^p(\R^n)}}{\ga!}  \\
  &\le C(g,p) \sum \frac{\al!}{k_1!\cdots k_\ell!} \sup_{\|v_j\|_E \le 1} 
  \frac{\|\p^\al f_i(u,~)(v_1,\ldots,v_{\al_0})\|_{L^p(\R^n)}}{\al!} \\ 
  &\hspace{1cm} \times \Big(\sup_{u,x} \frac{\|\p^{\de_{1}} \ph_1(u,x)\|_{\Lin^{\de_{10}}(E;E)}}{\de_{1}!}\Big)^{k_{11}} 
  \cdots \Big(\sup_{u,x} \frac{\|\p^{\de_{\ell}} \ph_1(u,x)\|_{\Lin^{\de_{\ell0}}(E;E)}}{\de_{\ell}!}\Big)^{k_{\ell1}} \\
  &\hspace{1cm} \times \Big(\sup_{u,x} \frac{\|\p^{\de_{1}} \ph_2(u,x)\|_{\Lin^{\de_{10}}(E;\R^n)}}{\de_{1}!}\Big)^{k_{12}} 
  \cdots \Big(\sup_{u,x} \frac{\|\p^{\de_{\ell}} \ph_2(u,x)\|_{\Lin^{\de_{\ell0}}(E;\R^n)}}{\de_{\ell}!}\Big)^{k_{\ell2}}
  \\
  &\le    
    C(g,p) \sum \frac{\al!}{k_1!\cdots k_\ell!}\, C ((n+1) D \rh)^{|\al|} ((n+1)\si)^{|\ga|} 
    \, M_{|\al|} M_{|\de_1|}^{|k_1|} \cdots M_{|\de_\ell|}^{|k_\ell|}
\end{align*}
where the sum is as in Case $\cA=\BM$.
By Lemma \ref{ssec:Faa}, 
we conclude that $\Box \ta>0$ $\exists B>0$ such that
\[
  \frac{\|\p^\ga h_i(u,~)(v_1,\ldots,v_{\ga_0})\|_{L^p(\R^n)}}{\ga!} \le B \ta^{|\ga|}\, M_{|\ga|},
\]
for all $i=1,\ldots,n$, $u \in K$, $v_j \in E$ with $\|v_j\|_E \le 1$, and 
$\ga \in \N^{1+n} \setminus \{0\}$.
Thus $h$ satisfies \eqref{cond:W}.

\paragraph{\bf Case $\cA=\SLM$}
We already know from Theorem \ref{thm:MichorMumford} that for all $K \in \sK(U)$, $p \in \N$, $\al \in \N^{1+n}$, 
\[
  \sup_{(u,x) \in K \times \R^n} (1+|x|)^p \|\p^\al h(u,x)\|_{\Lin^{\al_0}(E;\R^n)} < \infty.
\]
Let $K \in \sK(U)$ be fixed. Then, since $f=(f_1,\ldots,f_n)$ satisfies \eqref{cond:S}, 
we have $\Box \rh>0$ $\exists C>0$ such that, for all $i=1,\ldots,n$, $p \in \N$, $\al \in \N^{1+n}$, $u \in K$, and $x \in \R^n$,
\begin{align} \label{eq:fcompS}
  (1+|x|)^p\|\p^{\al} f_i(u,x)\|_{\Lin^{\al_0}(E;\R)} \le C \rh^{p+|\al|}\, p!|\al|!\, L_p M_{|\al|}.
\end{align}
That $g$ satisfies \eqref{cond:S} implies that there exists $C(g)>0$ so that 
\[
   \frac{1+|x|}{1+|\ph_2(u,x)|} = \frac{1+|x|}{1+|x+g(u,x)|} \le C(g) 
\] 
for all $(u,x) \in K \times \R^n$.
Fa\`a di Bruno's formula (\ref{ssec:Faa}.\ref{eq:FaaBM}), \eqref{eq:Bestph}, and \eqref{eq:fcompS} then give,
for $(u,x) \in K \times \R^n$, $p \in \N$, and $\ga \in \N^{1+n} \setminus \{0\}$, 
\begin{align*}
  &\frac{(1+|x|)^p\|\p^{\ga} h_i(u,x)\|_{\Lin^{\ga_0}(E;\R)}}{p!\ga!}  \\
  &\le C(g)^p\sum \frac{\al!}{k_1!\cdots k_\ell!} \frac{(1+|\ph_2(u,x)|)^p\|\p^\al f(u,\ph_2(u,x))\|_{\Lin^{\al_0}(E;\R)}}{p!\al!} \\ 
  &\hspace{1cm} \times \Big( \frac{\|\p^{\de_{1}} \ph_1(u,x)\|_{\Lin^{\de_{10}}(E;E)}}{\de_{1}!}\Big)^{k_{11}} 
  \cdots \Big( \frac{\|\p^{\de_{\ell}} \ph_1(u,x)\|_{\Lin^{\de_{\ell0}}(E;E)}}{\de_{\ell}!}\Big)^{k_{\ell1}} \\
  &\hspace{1cm} \times \Big( \frac{\|\p^{\de_{1}} \ph_2(u,x)\|_{\Lin^{\de_{10}}(E;\R^n)}}{\de_{1}!}\Big)^{k_{12}} 
  \cdots \Big( \frac{\|\p^{\de_{\ell}} \ph_2(u,x)\|_{\Lin^{\de_{\ell0}}(E;\R^n)}}{\de_{\ell}!}\Big)^{k_{\ell2}}
  \\
  &\le    (C(g) \rh)^p L_p
    \sum \frac{\al!}{k_1!\cdots k_\ell!}\, C ((n+1) D \rh)^{|\al|} ((n+1)\si)^{|\ga|} 
    \, M_{|\al|} M_{|\de_1|}^{|k_1|} \cdots M_{|\de_\ell|}^{|k_\ell|}
\end{align*}
where the sum is as above.
By Lemma \ref{ssec:Faa}, 
we conclude that $\Box \ta>0$ $\exists B$ such that
\[
  \frac{(1+|x|)^p\|\p^{\ga}h_i(u,x)\|_{\Lin^{\ga_0}(E;\R^n)}}{p!\ga!} \le B \ta^{p+|\ga|}\, L_p M_{|\ga|},
\]
for all $i=1,\ldots,n$, $(u,x) \in K \times \R^n$, $p \in \N$, and $\ga \in \N^{1+n} \setminus \{0\}$.
Thus $h$ satisfies \eqref{cond:S}.

\paragraph{\bf Case $\cA=\DM$} 
Since $f$ and $g$ satisfy \eqref{cond:D}, there are $K_f, K_g \in \sK(\R^n)$ so that $\on{supp} f(u,~) \subseteq K_f$ 
and $\on{supp} g(u,~) \subseteq K_g$ for all $u \in U$.
Then 
\[
  \on{supp} h(u,~) \subseteq \overline{\bigcup_{u \in U} (\Id + g(u,~))^{-1}(K_f)} =:  K_h,
\]
where $K_h \subseteq \R^n$ is compact, since $\bigcup_{u \in U} (\Id + g(u,~))^{-1}(K_f)$ is bounded.
Indeed, suppose that there exist sequences $u_k \in U$ and $x_k \in \R^n$ so that $x_k\to \infty$ and 
$x_k + g(u_k,x_k) \in K_f$. Then there is a subsequence $x_{k_n} \not\in K_g$ and thus 
$x_{k_n} \in K_f$, contradicting unboundedness. So $h$ satisfies condition \eqref{cond:D}.
\end{proof}

\section{Inversion} \label{sec:inv}

We assume that $M=(M_k)$ satisfies Hypothesis \ref{ssec:localDC} and that $L=(L_k)$ satisfies $L_k\ge 1$ for all $k$. 

\begin{theorem} \label{thm:invB}
  Inversion is $\CM$ on $\DiffBM$, $\DiffSpM$, $\DiffSLM$ and $\DiffDM$, i.e.,
  \begin{align*}
    \on{inv} &: \DiffBM \to \DiffBM, \\
    \on{inv} &: \DiffSpM \to \DiffSpM, \\
    \on{inv} &: \DiffSLM \to \DiffSLM, \\
    \on{inv} &: \DiffDM \to \DiffDM, \\
    &\hspace{.3cm} \on{inv}(F) := F^{-1},
  \end{align*}
  are $\CM$-mappings.
\end{theorem}

\begin{proof}
It suffices to prove that $\on{inv}$ maps 
$\CMb$-plots to $\CM$-plots; see Subsection \ref{ssec:localDC}. 
By Lemma~\ref{lem:key}, it is enough 
to consider $U \ni u \mapsto \Id+f(u,~)$, where 
$U$ is open in some Banach space $E$ and $f \in \CM(U\times \R^n,\R^n)$ satisfies \eqref{cond:B}, 
\eqref{cond:W}, \eqref{cond:S}, or \eqref{cond:D}, 
and to check that $g$ defined by $(\Id+f)\i=\Id+g$ satisfies \eqref{cond:B}, \eqref{cond:W}, \eqref{cond:S}, 
or \eqref{cond:D},
respectively. 
Since $g$ satisfies the implicit equation 
\begin{equation} \label{eq:implicit}
 (\Id +f) \o (\Id + g) = \Id \quad\Longleftrightarrow\quad  g(u,x) + f(u,x+g(u,x)) = 0
\end{equation}
we already know that $g$ is $\CM$ 
by the $\CM$ implicit function theorem; see Subsection \ref{ssec:localDC}.

Fix $K \in \sK(U)$ and let us introduce some notation.
We consider the mappings
\begin{align*}
    \textbf{f} &: U \x \R^n \x \R^n \to \R^n, \quad \textbf{f}(u,x,y) := y + f(u,x+y), \quad \text{and} \\
    \textbf{F} &: U \x \R^n \x \R^n \to U \x \R^n \x \R^n, \quad \textbf{F}(u,x,y) := (u,x,\textbf{f}(u,x,y)).
\end{align*}  
Then $\textbf{F}$ is a global $\CM$-diffeomorphism with inverse
\begin{align*}
    \textbf{G} &: U \x \R^n \x \R^n \to U \x \R^n \x \R^n, \quad \textbf{G}(u,x,y) := (u,x,y + g(u,x+y)).
\end{align*}
We shall abbreviate
\[
  H := E \x \R^n \x \R^n,\quad V := U \x \R^n \x \R^n,\quad  \text{ and } \quad L := K \x \R^n \x \R^n.  
\]
Fix $p_0 = (u_0,x_0,y_0) \in L$, and set $q_0=\textbf{F}(p_0)$ and $T =\textbf{F}'(p_0)^{-1} = \textbf{G}'(q_0)$. 
Defining 
  \begin{equation} \label{eq:bigph}
    \ph = \Id_V - T \o \textbf{F},  
  \end{equation}
  we have 
  \begin{equation} \label{eq:bigG}
    \textbf{G} = T + \ph \o \textbf{G}.
  \end{equation}

\paragraph{\bf Case $\cA = \BrM$}
  
  Since $f \in C^{\rM}(U \x \R^n,\R^n)$ satisfies \eqref{cond:B} and since $M=(M_k)$ is derivation closed, 
  there exist constants $C,\rh>0$ 
  so that  
  \begin{equation} \label{eq:fM}
    \|\p_1^{k}\p_2^\ell f(u,x)\|_{\Lin^{(k,\ell)}(E,\R^n;\R^n)} \le C \rh^{k+\ell-1} (k+\ell-1)!\, M_{k+\ell-1}
  \end{equation} 
  for $(u,x) \in K \x \R^n$ and $(k,\ell) \in \N^{2} \setminus\{0\}$; here $\p_1$ and $\p_2$ denote total derivatives. 
  Thus 
  \begin{align*}
    \|\p_1^{k} \p_2^{\ell} \p_3^{m} \textbf{f}(u,x,y)\|_{\Lin^{(k,\ell,m)}(E,\R^n,\R^n;\R^n)} 
    &\le
    \|(\p_1^{k} \p_2^{\ell+m} f)(u,x+y)\|_{\Lin^{(k,\ell+m)}(E,\R^n;\R^n)} 
    \\&\le
    C \rh^{k+\ell +m-1} (k+\ell +m-1)!\, M_{k+\ell +m-1}
  \end{align*}
  for all $p=(u,x,y) \in K \x \R^n \x \R^n$ and $k,\ell,m \in \N$ with $k+\ell+m\ge 2$.
  In view of \eqref{eq:bigph}, 
  we then find 
  \begin{equation} \label{eq:estph}
    \|\ph^{(h)}(p)\|_{\Lin^h(H;H)} \le C \|T\|_{\Lin(H;H)} \rh^{h-1}\, (h-1)!\, M_{h-1}, \quad \forall p \in L, 
    h \in \N_{\ge 2};
  \end{equation}
  possibly with different constants $C,\rh>0$.
  We have 
  \[
    \textbf{F}'(p) =\textbf{F}'(u,x,y) = 
    \begin{pmatrix}
     \Id_{E \times \R^n} & 0 \\
     \p_{(u,x)} \textbf{f}(u,x,y) & \I_n + \p_2 f(u,x+y)  
    \end{pmatrix}
  \]
  and hence
  \[
    \textbf{F}'(p)^{-1} = 
    \begin{pmatrix}
     \Id_{E \times \R^n} & 0 \\
     - (\I_n + \p_2 f(u,x+y))^{-1} \p_{(u,x)} \textbf{f}(u,x,y) & (\I_n + \p_2 f(u,x+y))^{-1}  
    \end{pmatrix}.
  \]
  It follows that
  \begin{align} \label{eq:A}
    A&:=\sup_{p \in L} \|\textbf{F}'(p)^{-1}\|_{\Lin(H;H)} < \infty, 
  \end{align}
  since $\inf_{(u,x) \in K \times \R^n} \det (\I_n + \p_2 f(u,x)) >0$ and by Lemma \ref{lem:LA},
  and since 
  $\p_{(u,x)} \textbf{f}(u,x,y) = \p_1 f(u,x + y) + \p_2 f(u,x + y)$ is bounded for $p \in L$.  
  
  Define $\ps_N$ and $g_N$ as in Lemma~\ref{lem:Yam}. Then
  \begin{gather*}
  \|\textbf{G}'(q_0)\|_{\Lin(H;H)} =\|\textbf{F}'(p_0){-1}\|_{\Lin(H;H)} = \|T\|_{\Lin(H;H)} \le A = g_N'(0) \quad \text{ and } \\
  \|\ph^{(h)}(p_0)\|_{\Lin^h(H;H)} \le \ps_N^{(h)}(0),  \quad 2 \le h \le N,
  \end{gather*}
  by \eqref{eq:estph}.
  Applying Fa\`a di Bruno's formula (\ref{ssec:Faa}.\ref{eq:FaaF}) to \eqref{eq:bigG} and \thetag{\ref{lem:Yam}.\ref{eq:g_N}} 
  we can deduce inductively (cf.\ (\ref{prop:Bgroup}.\ref{eq:yamanaka})) that, for $2 \le h \le N$,
  \begin{align*} 
  \|\textbf{G}^{(h)}(q_0)\|_{\Lin^h(H;H)} \le g_N^{(h)}(0) = h!\, c_h < \frac{1}{4 M_1(CA+1) \rh} (4A(CA+1) \rh)^{h}\, h!\, M_{h}. 
  \end{align*}
  Since $N$ and $(u_0,x_0) \in K \times \R^n$ were arbitrary, this implies that $g$ satisfies \eqref{cond:B}; 
  we know from Theorem \ref{thm:MichorMumford} that $g$ and $g^{(1)}$ are bounded on $K \times \R^n$.

\paragraph{\bf Case $\cA = \BbM$}
  This follow from the Roumieu case. 
  More precisely,
  the proof of Proposition \ref{prop:Bgroup}, Case $\cA = \BbM$, applied to   
  \[
    L_k := \frac{1}{k!} \sup_{(u,x) \in K \times \R^n} \|f^{(k)}(u,x)\|_{\Lin^k(E \times \R^n;\R^n)}
  \]
  provides a log-convex derivation closed sequence $N=(N_k)$ satisfying $N_{k+1}/N_k \to \infty$ and $L \le N \lhd M$. 
  So $f$ satisfies the Roumieu type estimate \eqref{eq:fM} for $M=(M_k)$ replaced by $N=(N_k)$, and the 
  arguments presented in ``Case $\cA = \BrM$'', which only require that $N=(N_k)$ be log-convex and derivation closed, 
  imply that $g$ satisfies the Roumieu version of \eqref{cond:B} for $N=(N_k)$. Thanks to $N \lhd M$, $g$ also 
  satisfies the Beurling version of \eqref{cond:B} for $M=(M_k)$.

\paragraph{\bf Case $\cA \in \{\SpM,\SLM\}$}
We may infer from Proposition \ref{prop:incl} and from Case $\cA=\BM$ that $g$
satisfies \eqref{cond:B}.
Let us define $\ph=(\ph_1,\ph_2)$ by setting
\[
  \ph_1(u,x) := u \quad \text{ and } \quad \ph_2(u,x) := x + g(u,x)
\]   
such that \eqref{eq:implicit} becomes
\begin{equation} \label{eq:implicit2}
  g = - f \o \ph.  
\end{equation}
We are now in the situation of \thetag{\ref{thm:compDiff}.\ref{eq:compph}} (for $h=-g$). 
The proof of Theorem \ref{thm:compDiff}, Case $\cA \in \{\SpM,\SLM\}$, 
(applied to \eqref{eq:implicit2} instead of \thetag{\ref{thm:compDiff}.\ref{eq:compph}}) 
shows that $g$ satisfies \eqref{cond:W} or \eqref{cond:S}, respectively.

\paragraph{\bf Case $\cA = \DM$}
The identity \eqref{eq:implicit} implies that $\on{supp} f(u,~) \subseteq \on{supp} g(u,~)$, and so 
$f$ satisfies \eqref{cond:D}.
\end{proof}

\section{Regularity} \label{sec:reg}

We assume that $M=(M_k)$ satisfies Hypothesis \ref{ssec:localDC} and that $L=(L_k)$ satisfies $L_k\ge 1$ for all $k$.

Following \cite{KMRu}, see also \cite{KM97r} and \cite[38.4]{KM97}, 
a $C^{[M]}$-Lie group $G$ with Lie algebra $\mathfrak g=T_eG$ 
is called $C^{[M]}$-regular if the following holds:
\begin{itemize}
\item 
For each $C^{[M]}$-curve 
$X\in C^{[M]}(\mathbb R,\mathfrak g)$ there exists a $C^{[M]}$-curve 
$g\in C^{[M]}(\mathbb R,G)$ whose right logarithmic derivative is $X$, i.e.,
\[
\begin{cases} g(0) &= e \\
\partial_t g(t) &= T_e(\mu^{g(t)})X(t) = X(t).g(t)
\end{cases} 
\]
The curve $g$ is uniquely determined by its initial value $g(0)$, if it
exists.
\item
Put $\on{evol}^r_G(X)=g(1)$, where $g$ is the unique solution required above. 
Then $\on{evol}^r_G: C^{[M]}(\mathbb R,\mathfrak g)\to G$ is required to be
$C^{[M]}$ also. 
\end{itemize}

\begin{theorem} \label{thm:reg}
  $\DiffBM$, $\DiffSpM$, $\DiffSLM$, and $\DiffDM$ are $\CM$-regular.
\end{theorem}

\begin{proof}
Let $\cA \in \{\BM,\SpM,\SLM,\DM\}$.
Let $\R \ni t\mapsto X(t,~)$ 
be a $C^{[M]}$-curve in the Lie algebra 
$\X_{\cA}(\R^n)=\cA(\R^n,\R^n)$. 
By Lemma \ref{lem:key}, 
we can assume that $X \in \CM(\R \times \R^n,\R^n)$ and satisfies 
\eqref{cond:B}, \eqref{cond:W}, \eqref{cond:S}, or \eqref{cond:D}.
The evolution of this time dependent vector field is 
the function given by the ODE
\begin{align} \label{eq:ODE}
\begin{split}
  &\on{Evol}(X)(t,x) = x+f(t,x), \\
  &\begin{cases} \p_t (x+f(t,x)) =f_t(t,x)= X(t,x+f(t,x)),\\
  f(0,x)=0. \end{cases}
\end{split}
\end{align}
Consider the autonomous vector field $(1,X)(t,x)=(1,X(t,x))$ on 
$\mathbb R\x \mathbb R^{n}$ and its flow $t\mapsto \on{Fl}^{(1,X)}_t$ given by
$\p_t\on{Fl}^{(1,X)}_t(s,x) = (1,X)(\on{Fl}^{(1,X)}_t(s,x))$ with initial condition 
$\on{Fl}^{(1,X)}_0(s,x)=(s,x)$.
This flow is complete since $(1,X)$ is bounded. 
Then $(t,\on{Evol}(X)(t,x)) = (t, x+f(t,x)) = \on{Fl}^{(1,X)}_t(0,x)$
and 
$(0,\on{Evol}(X)(t)^{-1}(x)) = \on{Fl}^{(1,X)}_{-t}(t,x)$.

We have to show the following:
\begin{itemize}
  \item $f(t,~)\in \cA(\R^n,\R^n)$ for each $t\in \R$
  \item $t \mapsto f(t,~) \in \cA(\R^n,\R^n)$ is $\CM$
  \item $X \mapsto f$ is $\CM$
\end{itemize}

\paragraph{\bf Case $\cA = \BM$}
In this case $X$ satisfies 
\eqref{cond:B}. Thus $(1,X)$ satisfies \eqref{cond:B} on $(a,b)\x \mathbb R^n$ for each bounded 
interval $(a,b)$.
By \cite{Yamanaka91}, see also \cite{Komatsu80}, its flow $\on{Fl}^{(1,X)}$ satisfies \eqref{cond:B} 
on $(a,b)\x \mathbb R^n$, and thus 
$t \mapsto f(t,~) \in \BM(\R^n,\R^n)$ is $\CM$, by Lemma \ref{lem:key}. 
In order to prove that 
\[
  \CM(\R,\X_{\BM}(\R^n))\ni X\mapsto \on{Evol}(X)(1,~)\in \DiffBM 
\]
is $\CM$, we consider a $\CM$-plot $X$ in $\CM(\R,\X_{\BM}(\R^n))$, i.e., 
$(u,t) \mapsto X(u,t,~) \in \CM(U \times \R,\X_{\BM}(\R^n))$.
Then we can argue as before.
Since $(0,\on{Evol}(X)(t)^{-1}(x)) = \on{Fl}^{(1,X)}_{-t}(t,x)$, we also have
that $\inf_{x \in \R^n} \det (\p_x \on{Evol}(X)(t,x))>0$.

\paragraph{\bf Case $\cA \in \{\SpM, \SLM\}$} 
By Case $\cA = \BM$ and Proposition \ref{prop:incl}, we already know that $f$ satisfies \eqref{cond:B}.
By \eqref{eq:ODE} we are in the situation of \thetag{\ref{thm:compDiff}.\ref{eq:compph}}. 
The proof of Theorem \ref{thm:compDiff} implies that $f_t$ satisfies \eqref{cond:W} or \eqref{cond:S}, respectively. 
Since we know from Theorem \ref{thm:MichorMumford} that $\|f(t,~)\|_{L^p(\R^n)}$ or 
$\sup_x(1+|x|)^p |f(t,x)|$ are bounded locally in $t$
(and since $M=(M_k)$ is increasing), 
we may conclude that $f$ satisfies \eqref{cond:W} or \eqref{cond:S}, respectively. 
Then we can finish the proof as in Case $\cA = \BM$.  

\paragraph{\bf Case $\cA = \DM$}
Since $X$ satisfies \eqref{cond:D}, for each $C>0$
there exists $r>0$ so that $\on{supp} X(t,~) \subseteq \overline {B_r(0)}$ for all $0 \le t \le C$, where 
$\overline{B_r(0)} \subseteq \R^n$ denotes the closed ball of radius $r$ centered at $0$. 
For $0 \le t \le C$ we consider  
\begin{equation} \label{eq:regD2}
  |f(t,x)|\le \int_0^t  |f_t(s,x)|ds=\int_0^t |X(s,x+f(s,x))|\,ds.
\end{equation}
It follows that, for $(t,x) \in [0,C] \times \R^n$, we have 
$|f(t,x)| \le t B$, where $B= \max\{|X(t,x)| : (t,x) \in [0,C] \times \R^n\}$, and hence
if $|x|>r+tB$ then $f(t,x)=0$, by \eqref{eq:regD2}. 
Similarly for negative $t$.
That means 
that $f$ satisfies \eqref{cond:D}, and thus 
$t \mapsto f(t,~) \in \DM(\R^n,\R^n)$ is $\CM$, by Lemma~\ref{lem:key}.
We may finish the proof as in Case $\cA = \BM$.
\end{proof}

\section{End of proof of Theorem \ref{thm:main}} \label{sec:end}

We assume that $M=(M_k)$ satisfies Hypothesis \ref{ssec:localDC} and that $L=(L_k)$ satisfies $L_k\ge 1$ for all $k$.

Theorems \ref{thm:compDiff}, \ref{thm:invB}, and \ref{thm:reg} imply that 
$\DiffDM$, $\DiffSLM$, $\DiffSpM$, and $\DiffBM$ are $\CM$-regular Lie groups.

That, for $1 \le p<q$, 
\[
    \DiffDM \!\!\rightarrowtail\!\! \DiffSLM \!\!\rightarrowtail\!\! \DiffSpM \!\!\rightarrowtail\!\! \Diff{W^{[M],q}(\R^n)} 
    \!\!\rightarrowtail\!\! \DiffBM
\]
are $\CM$ injective group homomorphisms follows from Proposition \ref{prop:incl} and the fact that 
a linear mapping between convenient vector spaces is $\CM$ if and only if it is bounded, see \cite[8.3]{KMRu}.

It remains to show that each group in the above diagram is a normal subgroup of the groups on its right.
It suffices to show that each group is a normal subgroup in $\DiffBM$.
So let $\Id+g \in \DiffBM$ and $(\Id+g)^{-1} = \Id +f \in \DiffBM$, which implies $g(x)+f(x+g(x))=0$.
Then, for $\Id + h \in \DiffBM$, 
\begin{align}\label{eq:normal}
\begin{split} 
  &((\Id+g)^{-1} \o (\Id+h) \o (\Id+g))(x) = ((\Id+f) \o (\Id+h) \o (\Id+g))(x) \\
  &= x + g(x) + h(x+g(x)) + f(x + g(x) + h(x+g(x))) \\
  &= x + h(x+g(x)) + f(x + g(x) + h(x+g(x)))  - f(x+g(x)) \\
  &= x + h(x+g(x)) + \int_0^1 df(x + g(x) + t h(x+g(x)))(h(x+g(x)))\, dt.
\end{split}
\end{align}

Assume that $h \in \SpM(\R^n,\R^n)$ or $h \in \SLM(\R^n,\R^n)$.
Then $x \mapsto h(x+g(x))$ is in $\SpM(\R^n,\R^n)$ or in $\SLM(\R^n,\R^n)$, by Theorem \ref{thm:Wcomp2} or Theorem \ref{thm:Scomp2}, 
respectively, which implies the assertion, 
since $[0,1] \times \R^n \ni (t,x) \mapsto df(x + g(x) + t h(x+g(x)))$ is $\BM$.

If $h$ has compact support, then so does \eqref{eq:normal}. The proof of Theorem \ref{thm:main} is complete.

\section{Composition is not \texorpdfstring{$C^{\bM}$}{CM} on \texorpdfstring{$\Diff{\cBl^{\bM}}(\R)$}{DiffBMloc}}\label{sec:local}

We define 
\[
  \cBl^{[M]}(\mathbb R^n) := C^{[M]}(\mathbb R^n) \cap \mathcal B(\mathbb R^n)
\]
and 
\[
  \Diff{\cBl^{[M]}}(\R^n) := \big\{F=\Id+f: f\in \cBl^{[M]}(\R^n,\R^n), 
    \inf_{x\in \R^n} \det(\mathbb I_n + df(x)) >0\big\}. 
\]
On $\cBl^{[M]}(\R^n)$ we consider the topology induced by its inclusion in the diagonal of $C^{[M]}(\mathbb R^n) \times \cB(\mathbb R^n)$.

We assume that $M=(M_k)$ satisfies Hypothesis \ref{ssec:localDC}.

\begin{lemma} \label{lem:B}
  Let $U$ be open in a Banach space $E$.
  For $f \in C^{[M]}(U \times \R^n)$ consider the following condition:
  \begin{equation}\label{cond:Bloc}
    \tag{$C\cBl$} \forall K \in \sK(U) ~ \Box \rh>0 ~\forall \al \in \N^n: 
    \sup_{\substack{k\in \N\\ (u,x) \in K \times \R^n\\ \|v_j\|_E \le 1}} 
    \frac{|\p_u^k\p_x^{\al} f(u,x)(v_1,\dots,v_k)|}{\rh^{k}\,  k!\, M_{k}} < \infty.  
  \end{equation}
  Then
  \begin{align*}
    C^{[M]}_b(U, \cBl^{[M]}(\mathbb R^n)) 
    &\subseteq \big\{f^\vee : f\in C^{[M]}(U \times \R^n) \text{ satisfies } \eqref{cond:Bloc}\big\}
    \\&\subseteq C^{[M]}(U, \cBl^{[M]}(\mathbb R^n)).
  \end{align*}
  In the Beurling case $C^{[M]}=C^{(M)}$ the inclusions are equalities.
\end{lemma}

\begin{proof}
If $f^\vee \in \CM_b(U, \cBl^{[M]}(\R^n))$, then $f \in \CM(U \times \R^n)$, 
by the $C^{[M]}$-exponential law (Theorem \ref{ssec:localDC}), and 
for each $K \in \sK(U)$ $\Box \rh>0$ the set $\set^M_{K,\rh}(f^\vee)$ is bounded $\cB(\R^n)$, i.e., 
$f$ satisfies \eqref{cond:Bloc}. This shows the first inclusion.
 
If $f \in \CM(U \times \R^n)$, then $f^\vee \in \CM(U, \CM(\R^n))$, by the $C^{[M]}$-exponential law. 
That $f$ satisfies \eqref{cond:Bloc} means that $f \in \CM_b(U, \cB(\R^n))$.  The second inclusion follows. 

Equality in the Beurling case follows from Lemma \ref{ssec:localDC}. 
\end{proof}

\begin{example} \label{example}
  Let $M= (M_k)$ be strongly non-quasianalytic (\ref{ssec:ws}.\ref{eq:snqa})
and assume $M_{k+1}/M_k  \nearrow \infty$. Assume that $C^{(M)} \supseteq \cG^{3/2}$, 
where $\cG^s := C^{\{(k!)^{s-1}\}}$, for $s\ge1$, denotes the Gevrey class of order $s$.
For instance,  
we may take $M_k := (k!)^{s-1}$ for any $s>3/2$.

By \cite{Petzsche88} there exists a sequence of functions $\ch_k \in \cD^{(M)}(\R)$ so that 
\begin{align*}
  & \exists A : \quad \on{supp} \ch_k \subseteq \Big[-\frac{AM_k}{M_{k+1}},\frac{AM_k}{M_{k+1}}\Big] \\
  & \ch_k^{(j)}(0) = \de_{jk} \\
  & \forall \rh > 0 ~\exists C(\rh), H(\rh) :\quad   \|\ch_k\|^M_{\R,\rh} = \sup_{\substack{j \in \N\\ x \in \R}} \frac{|\ch_k^{(j)} (x)|}{\rh^j j! M_j} \le C(\rh) \frac{H(\rh)^k}{k! M_k}.
\end{align*}
We define 
\[
\mu_k :=  2^k \sqrt{k!}\, M_k.
\]
Then, as $M_k \le M_{k+1}$, we have
\[
  \mu_{k+1} - \mu_k \ge 2^k \sqrt{k!}\,  M_k\, (2 \sqrt{k+1} - 1) \ge 1,   
\]  
in particular, $\mu_k \nearrow \infty$. By Stirling's formula, $k! \le k^k \le e^k k!$, and hence   
\begin{align} \label{eq:infty}
  \Big(\frac{\mu_k}{k! M_k}\Big)^{1/k} = \frac{2}{(\sqrt{k!})^{1/k}} &\to 0 \nonumber \\
  k \Big(\frac{\mu_k}{k! M_k}\Big)^{1/k} = \frac{2k}{(\sqrt{k!})^{1/k}} &\to \infty
\end{align}
as $k \to \infty$.
Let us set 
\[
  f(x) := \sum_{k\ge0} \mu_k \ch_k(Ak(x-\mu_k)), \quad x \in \R.
\]
We have 
\begin{align*}
  |f^{(j)}(x)| &\le \sum_k (Ak)^j \mu_k |\ch_k^{(j)}(Ak(x-\mu_k))| \\ 
  &\le  \sum_k (Ak)^j \mu_k \frac{H(\rh)^k}{k! M_k} C(\rh) \rh^j j! M_j \\
  &= C(\rh) (A \rh)^j j! M_j  \sum_k k^j \frac{(2 H(\rh))^k}{\sqrt{k!}}  \\
  &\le  \tilde C(\rh,j) (A \rh)^j j! M_j < \infty,
\end{align*}
for some constant $\tilde C(\rh,j)$ depending on $\rh$ and $j$.
On each compact subset of $\R$ the sum in the definition of $f$ is finite, since the support of the $k$th summand 
is contained in $[\mu_k-\frac{M_k}{kM_{k+1}},\mu_k + \frac{M_k}{kM_{k+1}}]$. It follows that $f$ is an element 
of $\cBl^{(M)}(\R)$. 
On the other hand
\begin{align} \label{eq:jder}
  f^{(j)}(\mu_j) =  \sum_k (Ak)^j \mu_k \ch_k^{(j)}(Ak(\mu_j-\mu_k))  = (Aj)^j \mu_j,  
\end{align} 
since $|\mu_j-\mu_k|\ge 1$ unless $j=k$. 

Let us set 
\[
  g_0(x) := \exp(-(x^2+\frac{1}{x^2})), \quad x \in \R\setminus\{0\},   \quad  g_0(0) := 0. 
\]
Then $g_0$ belongs to the Gevrey class $\cG^{3/2}(\R)$. Consequently, 
\[
  g(t,x) := t + g_0(t)g_0(x) \in \cG^{3/2}(\R^2). 
\]
Moreover $g$ satisfies \eqref{cond:Bloc}. Indeed, if $p\ge 1$ or $p = 0$ and $k \ge 2$, then for all compact 
$I \subseteq \R$ and all $\rh>0$,
\begin{align*}
  \frac{|\p_t^k \p_x^p g(t,x)|}{\rh^k k! M_k} = \frac{|g_0^{(k)}(t)| |g_0^{(p)}(x)|}{\rh^k k! M_k} 
  \le C(\rh) |g_0^{(p)}(x)| \quad \forall t \in I, k \in \N,  
\end{align*}
and the right-hand side is globally bounded in $x \in \R$. 
(The cases $\p_t g(t,x) = 1+ g_0'(t)g_0(x)$ and $g(t,x) = t+g_0(t)g_0(x)$ are easy.)
Note that 
\[
  \p_t^k g(0,x) = 
  \begin{cases}
    1 & k=1 \\
    0 & k \ne 1
  \end{cases}
  .
\]

The function 
\[
h(t,x) := f(x+g(t,x)) \in C^{(M)}(\R^2).
\] 
satisfies
\begin{align*}
  \p_t^j h(0,x) 
  = j! \sum_{\ell \ge 1} \sum_{\substack{\al_i>0\\ \al_1+ \cdots +\al_\ell =j}} \frac{f^{\ell}(x)}{\ell!}
    \prod_{i=1}^\ell \frac{\p_t^{\al_i} g(0,x)}{\al_i!} = f^{(j)}(x)
\end{align*}
and thus, by \eqref{eq:jder} and \eqref{eq:infty}, 
\begin{equation} \label{eq:cex}
  \Big(\frac{|\p_t^j h(0,\mu_j)|}{j! M_j}\Big)^{1/j} = \Big(\frac{|f^{(j)}(\mu_j)|}{j! M_j}\Big)^{1/j}  
  = A j \Big(\frac{\mu_j}{j! M_j}\Big)^{1/j} \to \infty  
\end{equation}  
as $j \to \infty$.
\end{example}

\begin{theorem}
  Left translation is not $C^{(M)}$ on $\Diff{\cBl^{(M)}}(\R)$.
\end{theorem}

\begin{proof}
  This follows, in view of (\ref{prop:Bgroup}.\ref{eq:group1}) and Lemma \ref{lem:B}, from 
  \thetag{\ref{example}.\ref{eq:cex}} with the above choices for $f$ and $g$; 
  by multiplying $f$ and $g$ by a suitable constant we can achieve that $\inf_x \p_x f(t,x)>-1$ and $\inf_x \p_x g(t,x)>-1$ 
  for all $t$. 
\end{proof}

\begin{theorem}
  Right translation is $\CM$ on $\Diff{\cBl^{[M]}}(\R^n)$.
\end{theorem}

\begin{proof}
  Let $U$ be open in a Banach space $E$ 
  and let $f \in \CM(U \times \R^n,\R^n)$ satisfy condition \eqref{cond:Bloc} and $\inf_{x \in \R^n} \det (\I_n +df(u,x))>0$ 
  for all $u \in U$.
  Let $g \in \cBl^{[M]}(\R^n,\R^n)$ with $\inf_{x \in \R^n} \det (\I_n +dg(x))>0$. 
  In view of (\ref{prop:Bgroup}.\ref{eq:group1}) and Lemma \ref{lem:B} it suffices to show that 
  $h(u,x) := f(u,x+g(x))$ satisfies condition \eqref{cond:Bloc}. 
  By Fa\`a di Bruno's formula (\ref{ssec:Faa}.\ref{eq:FaaF}), for $k\in \N$, $\ell \in \N_{\ge1}$, $u \in U$, $x \in \R^n$, 
  and $\ph(x):= x+g(x)$,     
  \begin{align*}
    \frac{\|\p_1^k \p_2^\ell h(u,x)\|_{\Lin^{(k,\ell)}(E,\R^n;\R^n)}}{\ell!} 
    &\le  \sum_{j\ge 1} \sum_{\substack{\al\in \N_{>0}^j\\ \al_1+\dots+\al_j =\ell}}
  \frac{\|(\p_1^k \p_2^j f)(u,\ph(x))\|_{\Lin^{(k,j)}(E,\R^n;\R^n)}}{j!} 
  \\
  &\hspace{2.4cm} \x \prod_{i=1}^j \frac{\|\ph^{(\al_i)}(x)\|_{\Lin^{\al_i}(\R^n;\R^n)}}{\al_i!},
  \end{align*}
  and hence, since $f$ satisfies \eqref{cond:Bloc}, 
  \begin{align*}
    &\forall K \in \sK(U) ~ \Box \rh>0 ~\forall j \in \N_{\ge1} ~\exists C_{K,\rh,j}>0 ~\forall k\in \N ~\forall (u,x) \in K \times \R^n:\\
    &
    \frac{\|\p_1^k \p_2^\ell h(u,x)\|_{\Lin^{(k,\ell)}(E,\R^n;\R^n)}}{\rh^k k!\, M_k} 
    \le \ell! \sum_{j\ge 1}\!\! \sum_{\substack{\al\in \N_{>0}^j\\ \al_1+\dots+\al_j =\ell}}
    \!\!\frac{C_{K,\rh,j}}{j!}
    \prod_{i=1}^j \frac{\sup_x\|\ph^{(\al_i)}(x)\|_{\Lin^{\al_i}(\R^n;\R^n)}}{\al_i!}
    \\ 
    & \hspace{4.2cm} =: C(K,\rh,\ell,g),
  \end{align*}
  that means 
  \begin{align*}
    \forall K \in \sK(U) ~ \Box \rh>0 ~\forall \ell \in \N_{\ge1}: 
    \sup_{\substack{k\in \N \\ (u,x) \in K \times \R^n}}
    \frac{\|\p_1^k \p_2^\ell h(u,x)\|_{\Lin^{(k,\ell)}(E,\R^n;\R^n)}}{\rh^k k!\, M_k} 
    <\infty.
  \end{align*}
  One easily checks that this holds also for $\ell =0$. 
  Thus $h$ satisfies condition \eqref{cond:Bloc}. 
\end{proof}

\section{Extensions over \texorpdfstring{$\Diff\cA(\mathbb R)$}{DiffA} 
and Denjoy--Carleman solutions of the Hunter--Saxton equation on the real line}\label{sec:HS}

In this section we carry over the main results of the paper \cite{BBM14b} to the 
classes of 
diffeomorphism groups $\Diff\cA(\mathbb R)$ and show that the Hunter--Saxton equation is 
well-posed for all $\cA\in \{W^{[M],1},\SLM,\DM\}$.
We again assume that $M=(M_k)$ satisfies Hypothesis \ref{ssec:localDC} and that $L=(L_k)$ satisfies $L_k\ge 1$ for all $k$. 

\subsection{Extending the function spaces $\cA(\mathbb R)$}\label{HS1}
For $\cA\in \{W^{[M],1},\SLM,\DM\}$ we consider the space $\cA_2(\mathbb R):=\{f\in C^\oo(\mathbb R): f'\in\cA(\mathbb R)\}$ of 
antiderivatives of functions in $\cA(\mathbb R)$. 
Since $\cA(\R)  \subseteq L^1(\R)$ the limits
$f(\pm\oo):=\lim_{x\to\pm\oo}f(x)$
exist; thus $f\mapsto (f',f(-\infty))$ is a linear isomorphism 
$\cA_2(\mathbb R)\to \cA(\mathbb R)\x \mathbb R$ which we use to describe 
the convenient vector space structure on $\cA_2(\mathbb R)$. 
We consider 
the exact sequence 
\begin{equation*}
\xymatrix{
\cA_0(\mathbb R)\quad \ar@{^(->}[r] & 
 \cA_2(\mathbb R) 
\ar@{->>}[rrr]^{(\on{ev}_{-\infty},\on{ev}_\infty)} &&& \mathbb R^2
}
\end{equation*}
where $\cA_0(\mathbb R):= \ker(\ev_{-\infty},\ev_{\infty})$. 
We also consider the closed linear subspace 
$\cA_1(\mathbb R):=\ker(\ev_{-\oo})=\{f\in\cA_2(\mathbb R): f(-\infty)=0\}$ 
of antiderivatives of the form 
$x\mapsto \int_{-\infty}^x g(y)\,dy$ for $g\in \cA(\mathbb R)$.
For $\cA=\SLM$ or $=\DM$ we have $\cA_0=\cA$. For $\cA=W^{[M],1}$ we have 
$W^{[M],1}_0(\mathbb R)\ne W^{[M],1}(\mathbb R)$.

\subsection{The corresponding group extensions}\label{HS2}
For $\cA\in \{W^{[M],1},\SLM,\DM\}$ 
we consider the groups 
$$
\Diff\cA_i(\mathbb R) = \big\{\on{Id}+f: f\in \cA_i(\mathbb R), f'>-1\big\},\qquad i=\emptyset,0,1,2.
$$
 
\begin{theorem*} 
The groups $\Diff\DM(\mathbb R)$, 
$\Diff\DM_1(\mathbb R)$, 
$\Diff\DM_2(\mathbb R)$, 
$\Diff\SLM(\mathbb R)$, 
$\Diff\tensor{\cS}{}_{[L],1}^{[M]}(\mathbb R)$, 
$\Diff\tensor{\cS}{}_{[L],2}^{[M]}(\mathbb R)$, 
$\Diff W^{[M],1}(\mathbb R)$,
$\Diff W^{[M],1}_0(\mathbb R)$, 
$\Diff W^{[M],1}_1(\mathbb R)$, 
$\Diff W^{[M],1}_2(\mathbb R)$, 
are all $C^{[M]}$-regular Lie groups. 
We have the following injective $C^{[M]}$ group homomorphisms:
\[
  \xymatrix{
& & \Diff W^{[M],1}(\mathbb R) \ar@{^{ (}->}[d] & \\
\Diff\DM(\mathbb R)
 \ar@{{ >}->}[r] \ar@{^{ (}->}[d] &\Diff\SLM(\mathbb R) \ar@{^{ (}->}[d] \ar@{^{ >}->}[ru]
 \ar@{{ >}->}[r] &\Diff W^{[M],1}_0(\mathbb R) \ar@{^{ (}->}[d] &\\
\Diff\DM_1(\mathbb R) \ar@{{ >}->}[r] \ar@{^{ (}->}[d] &
\Diff\tensor{\cS}{}_{[L],1}^{[M]}(\mathbb R) \ar@{{ >}->}[r] \ar@{^{ (}->}[d]
 &\Diff W^{[M],1}_1(\mathbb R) \ar@{^{ (}->}[d] &\\
\Diff\DM_2(\mathbb R)
 \ar@{{ >}->}[r] &\Diff\tensor{\cS}{}_{[L],2}^{[M]}(\mathbb R)
 \ar@{{ >}->}[r] &\Diff W^{[M],1}_2(\mathbb R) \ar@{{ >}->}[r] &\Diff\BM(\mathbb R)\\
}  
\]
Each group is a normal subgroup in any other in which it is contained, in particular in 
$\Diff\BM(\mathbb R)$.
Moreover, the columns of the following diagram are $C^{[M]}$ extensions. They are splitting 
extensions (semidirect 
products) if and only if the weight sequence $M=(M_k)$ is non-quasianalytic. 
$$
\xymatrix{
 \Diff\DM(\mathbb R) \ar@{{ >}->}[r] \ar@{^{ (}->}[d] &\Diff\SLM(\mathbb R)
 \ar@{^{ (}->}[d] \ar@{{ >}->}[r] &\Diff W^{[M],1}_0(\mathbb R) \ar@{^{ (}->}[d]\\
\Diff\DM_2(\mathbb R) \ar@{ >->}[r] \ar@{->>}[d] &
\Diff\tensor{\cS}{}_{[L],2}^{[M]}(\mathbb R) \ar@{ >->}[r] \ar@{->>}[d] &
\Diff W^{[M],1}_2(\mathbb R) \ar@{->>}[d]^{(\on{Shift}_\ell,\on{Shift}_r)}\\
\mathbb R^2 \ar@{=}[r] &\mathbb R^2 \ar@{=}[r] &\mathbb R^2\\
}$$
The extensions $\on{Shift}_r:\Diff\cA_1(\mathbb R)\to \mathbb R$ are always splitting.
\end{theorem*}

\begin{proof}
That the groups  $\Diff\DM(\mathbb R)$, $\Diff\SLM(\mathbb R)$, 
$\Diff W^{[M],1}(\mathbb R)$, and $\Diff\BM(\mathbb R)$ are 
$\CM$-regular Lie groups is proved above.
The proof there can be adapted to the cases
$\Diff W^{[M],1}_i(\mathbb R)$ for $i=0,1,2$,  replacing \eqref{cond:W} by
\begin{align*}
                \forall K \in \sK(U) ~ \Box \rh>0 : &
    \sup_{\substack{k \in \N,\al \in \N_{>0}\\ u \in K\\  \|v_j\|_E\le 1}}  
    \frac{\int_{\R} |\p_u^k\p_x^{\al} f(u,x)(v_1,\dots,v_k)|\, dx}{\rh^{k+\al}\, 
    (k+\al)!\, M_{k+\al}}< \infty,
               \\&
    \sup_{\substack{k \in \N, u \in K\\  \|v_j\|_E\le 1}}  
    \frac{ |\p_u^k f(u,-\infty)(v_1,\dots,v_k)|}{\rh^{k}\, k!\, M_{k}}< \infty.
\end{align*}

The shift homomorphisms are given by 
$\on{Shift}_\ell(\on{Id}+f) = f(-\infty)$ and   
$\on{Shift}_r(\on{Id}+f) = f(\infty)$. 
If $M=(M_k)$ is non-quasianalytic we choose functions $f_\ell$ and $f_r$ in $\cA_2(\R)$ with 
$$
f_{\ell}(x) = \begin{cases} 1 &\text{  for } x\le -1 \\ 
                            0 &\text{  for } x\ge 0, \end{cases}
\qquad
f_{r}(x) = \begin{cases} 0 &\text{  for } x\le 0 \\ 
                         1 &\text{  for } x\ge 1 \end{cases}
$$
and consider the vector fields $X_\ell = f_\ell\p_x$ and $X_r=f_r\p_x$ which commute, 
$[X_\ell,X_r]=0$.
A splitting $C^{[M]}$ section $s:\mathbb R^2\to \Diff\cA_2(\mathbb R)$ of $(\on{Shift}_\ell,\on{Shift}_r)$ is given by
$s(a,b)= \on{Fl}^{X_\ell}_a\o \on{Fl}^{X_r}_b$. 
If $M=(M_k)$ is quasianalytic, then any homomorphic section $s$ gives rise to two commuting flows and 
corresponding vector fields as above, so that $f_\ell f_r'-f_\ell' f_r =0$. But then the rational 
function $f_r/f_\ell$ has vanishing derivative, so that the two vector fields are proportional. 
Thus there does not exist a homomorphic section in the quasianalytic case. 
We always find a homomorphic $\CM$ section for $\on{Shift}_r:\Diff\cA_1(\mathbb R)\to \mathbb R$ alone, 
using $0\ne f\ge 0$ in $\SLM(\mathbb R)$  and the vector field 
$\int_{-\infty}^x f(y)\,dy\,\p_x$ which is contained in all spaces $\cA_1(\mathbb R)$. 

That the extended groups $\Diff \cA_i(\R)$ for $\cA \in \{\SLM,\DM\}$ are Lie groups is easily seen using a smooth section, cf.\ 
\cite[15.12]{MichorH}.
And that they are regular is proved in 
\cite[38.6]{KM97}. 
That each group is normal in the largest one is also proved above.
\end{proof}

\subsection{The homogeneous $H^1$ Riemannian metric on $\Diff\cA_1(\mathbb R)$ and its geodesics}
\label{ssec:HS-results}
We choose $\cA\in \{ W^{[M],1}, \SLM, \DM\}$.
We consider the following weak right invariant Riemannian 
metric on $\Diff\cA(\mathbb R)$ and on $\Diff\cA_1(\mathbb R)$; it is called the homogeneous 
$H^1$-metric or the $\dot H^1$-metric. 
$$
G_\ph(X\o \ph, Y\o \ph) = \langle X\o\ph,Y\o\ph \rangle_{\dot H^1} = G_{\on{Id}}(X,Y) = \int_{\mathbb R}X'(x)\,Y'(x)\,dx.
$$
This is well defined since  
$\cA \subseteq W^{[M],2}$, by Proposition \ref{prop:incl}.
This Riemannian metric
has the following property:  

$\bullet$ \cite[Section~4.2]{BBM14b} 
For $\cA \in \{ W^{[M],1},\tensor{\cS}{}_{[L]}^{[M]},\DM\}$ the geodesic equation 
for this weak Riemannian metric on $\Diff\cA_1(\mathbb R)$ is the Hunter--Saxton equation 
\begin{equation*}
\boxed{\quad
\begin{aligned}
u &= (\ph_t)\circ\ph\i, 
\quad
u_{t} = -u u_x +\frac12 \int_{-\infty}^x   (u_x(z))^2 \,dz. 
\\ &\text{Equivalently, }\quad
u_{tx} = -u u_{xx}-\frac12 u^2_x \;.
\end{aligned}
\quad}
\end{equation*}
But 
the covariant derivative and, in particular, the geodesic equation does not exist 
on the closed subgroup $\Diff\cA(\mathbb R)$.

\subsection{The $R$-transform as an isometry onto a flat space} \label{ssec:Rtransform}
\cite[Section~4.3]{BBM14b}  
Let $\cA(\R,\R_{>-2}) = \left\{ f \in \cA(\R) \;:\; f(x) > -2 \right\}$ and consider 
the $R$-mapping given by 
$$ R:\left\{
\begin{aligned}
 \Diff\cA_1(\mathbb R)&\to 
\cA\big(\R,\mathbb R_{>-2}\big)\subset\cA(\R,\R)\\
\ph &\mapsto
2\;\big((\ph')^{1/2}-1\big)\; ,
\end{aligned}\right.
$$
    That $R$ has values in $\cA(\R)$ and is $C^{[M]}$ is seen as follows:
                As in \cite[Section~4.3]{BBM14b} 
    we write $\ph = \Id + f$ with $f \in \cA_1(\R)$ and conclude that $R(\ph) = f' + F(f') f'$, where 
    $F : \R_{>-1} \to \R$ is a real analytic function satisfying $F(0)=0$. 
    The assumption $f'>-1$ and the fact that $f' \in \cA(\R)$ vanishes at $\pm \infty$ imply $-1+\ep \le f'(x) \le C$ for 
    constants $\ep,C>0$ independent of $x$. Thus, 
    we may conclude that $F(f') \in \BM(\R)$, by (the proof of) Theorem \ref{thm:Bcomp}, since for $F$, being real analytic,  
    the required $[M]$-estimates hold on the interval $[-1+\ep,C]$ (thanks to our assumption $C^\om \subseteq \CM$). 
    The statement then follows, since $\cA(\R)$ is a $\BM(\R)$-module.
                That $R$ is $C^{[M]}$ is now a consequence of Theorem \ref{thm:compDiff}.  

The $R$-map is invertible with polynomial inverse
$$R\i :\left\{
\begin{aligned}\cA\big(\R,\mathbb R_{>-2}\big) &\to \Diff\cA_1(\mathbb R)
\\\ga&\mapsto (x\mapsto x+\frac14 \int_{-\infty}^x \big(\ga^2(y)+4\ga(y)\big)\;dy\;).
\end{aligned}\right.$$ 

$\bullet$ 
The pull-back of the flat $L^2$-metric via $R$ is the $\dot H^1$-metric on $\Diff\cA_1(\R)$, i.e.,
$$R^*\langle \cdot,\cdot\rangle_{L^2} = \langle\cdot,\cdot\rangle_{\dot H^1}\; .$$
Thus the space $\big(\Diff\cA_1(\R),\dot H^1\big)$ is a flat space in the sense of Riemannian 
geometry.

Here $\langle \cdot,\cdot\rangle_{L^2}$ denotes the $L^2$-inner product on $\cA(\R)$ interpreted as 
a weak Riemannian metric on $\cA(\R,\R_{>-2})$, that does not depend on the basepoint, i.e. 
\[
G^{L^2}_\ga(h,k) = \langle h,k\rangle_{L^2}=\int_\R h(x)k(x)\;dx\;,
\]
for $h,k \in \cA(\R) \cong T_\ga \cA(\R,\R_{>-2})$.

\subsection{Explicit solutions for the geodesic equation}
\cite[Section~4.4]{BBM14b} 
Given $\ph_0,\ph_1\in \Diff\cA_1(\R)$ the unique geodesic $\ph(t,x)$ connecting them is given by
\begin{align*}
\ph(t,x)=R\i\Big((1-t)R(\ph_0)+tR(\ph_1) \Big)(x)\; \,
\end{align*}
and their geodesic distance is
\begin{align}
d(\ph_0,\ph_1)^2=4\int_{\R} \big((\ph'_1(x))^{1/2}-(\ph'_0(x))^{1/2}\big)^2\;dx\; .
\end{align}
Furthermore the support of the geodesic is localized in the following sense: if $\ph(t,x)= x + 
f(t,x)$ with $f(t) \in \cA_1(\R)$ and similarly for $\ph_0,\ph_1$, then $\on{supp}(\p_x f(t))$ is 
contained in $\on{supp}(\p_x f_0)\cup \on{supp}(\p_x f_1)$.

$\bullet$ 
Thus the Hunter--Saxton equation is $C^{[M]}$-well posed in each space $\cA_1(\mathbb R)$ for each 
$\cA \in \{ W^{[M],1},\tensor{\cS}{}_{[L]}^{[M]},\DM\}$ and all weight sequences $M=(M_k)$ and $L=(L_k)$, $L_k \ge 1$, with 
the restrictions of Hypothesis \ref{ssec:localDC}.

$\bullet$ \cite[Section~4.5]{BBM14b} 
The metric space $\big(\Diff\cA_1(\R),\dot H^1\big)$ is path-connected and geodesically convex but 
not geodesically complete. Each non-trivial geodesic is incomplete.

\subsection{The strange behavior of geodesics on the Lie subgroup $\Diff\cA(\mathbb R)$}
\label{HS.10}
\cite[Sections~4.5, 4.6, and 4.7]{BBM14b} 
For $\cA\in \{ W^{[M],1},\SLM,\DM\}$, the $R$-mapping
is bijective 
$$ R:
\Diff\cA(\mathbb R)\to \Big\{\ga \in \cA(\R,\mathbb R_{>-2}):\int_{\R}\ga(x)\big(\ga(x)+4\big)\; dx = 0 \Big\}
 \subset \cA\big(\R,\mathbb R_{>-2}\big).
$$
The pull-back of the flat $L^2$-metric via $R$ is again the homogeneous Sobolev metric of order one.
The image of the $R$-map is a splitting submanifold 
in the sense of  \cite[Section 27.11]{KM97}. 
The geodesic equation (the Christoffel symbol) does not exist 
on $\Diff\cA(\mathbb R)$. 
The geodesic distance $d^\cA$ on $\Diff\cA(\R)$  coincides with the restriction of $d^{\cA_1}$ to 
$\Diff\cA(\R)$, i.e., for $\ph_0, \ph_1 \in \Diff\cA(\R)$ we have
$d^\cA(\ph_0, \ph_1) = d^{\cA_1}(\ph_0, \ph_1)$.
Every geodesic in $\Diff\cA_1(\R)$ intersects $\Diff\cA(\R)$ at most twice and 
every geodesic is tangent to a right-coset of $\Diff\cA(\R)$ at most once.
For $\ph_0,\ph_1\in\Diff\cA(\R)$ we can give the following formula for the size of the shift
$$
\on{Shift}_r(\ph(t))=\frac{t^2-t}{4}\big\|R(\ph_0)-R(\ph_1)\big\|^2_{L^2} 
= (t^2-t)\big\|\sqrt{\ph_0'}-\sqrt{\ph_1'}\,\big\|^2_{L^2}  \; .$$

\subsection{Continuing geodesics beyond the group and the geodesic completion to a monoid} 
\cite[Section~4.10]{BBM14b} 
Consider a straight line $\ga(t) = \ga_0 + t\ga_1$ in $\cA(\mathbb R,\mathbb R)$. 
Then $\ga(t)\in\cA(\mathbb R,\mathbb R_{>-2})$ precisely for $t$ in an open interval $(t_0,t_1)$ 
which is finite at least on one side; at $t_1<\infty$, say. 
Note that 
$$
\ph(t)(x):=R\i(\ga(t))(x) = x+ \frac14\int_{-\infty}^x \ga^2(t)(u)+4\ga(t)(u)\,du
$$
makes sense for all $t$, and 
that $\ph(t):\mathbb R\to \mathbb R$ is smooth and $\ph(t)'(x)\ge0$ for all $x$ and $t$ so that 
$\ph(t)$ is monotone non-decreasing. Moreover, $\ph(t)$ is proper and surjective, since $\ga(t)$
vanishes at $-\infty$ and $\infty$. Let 
$$
\on{Mon}_{\cA_1}(\mathbb R) := \big\{\on{Id}+f: f\in\cA_1(\mathbb R,\mathbb R), f'\geq-1\big\}
$$
be the monoid (under composition) of all such functions.

For $\ga\in\cA(\mathbb R,\mathbb R)$ let 
$x(\ga):=\min\{x\in \mathbb R\cup\{\infty\}: \ga(x)=-2\}.$
Then for the line $t\mapsto \ga(t)$ from above we see that $x(\ga(t))<\infty$ for all $t>t_1$.
Thus, if the so extended geodesic $\ph$ leaves the diffeomorphism group at $t_1$, it never comes back but 
stays inside $\on{Mon}_{\cA_1}(\mathbb R)$ for the rest of its life. 
In this sense $\on{Mon}_{\cA_1}(\mathbb R)$ is a `\emph{geodesic completion}' of 
$\Diff\cA_1(\mathbb R)$, and $\on{Mon}_{\cA_1}(\mathbb R)\setminus \Diff\cA_1(\mathbb R)$ is the 
`\emph{boundary}'.

\subsection{Remark} The results from \ref{ssec:HS-results} carry over to the periodic case; this was 
spelled out already in \cite[Section~6]{BBM14b}. They 
also carry over to the two-component Hunter--Saxton equation on the real line, namely to 
the semidirect product $\Diff\cA_1(\mathbb R)\ltimes \cA(\mathbb R)$; see 
\cite[Section~5]{BBM14b}.
Note that the Hunter--Saxton equation also admits soliton-like solutions; see 
\cite[Section~4.11]{BBM14b}. These are not $C^\infty$ 
as diffeomorphisms, so there is no Denjoy--Carleman improvement for them.

\section{The strange behavior of composition on \texorpdfstring{$\Diff\SpM_2(\mathbb R)$}{WM2} or 
\texorpdfstring{$\Diff(\SpM\cap L^1)_2(\mathbb R)$}{WM2} for \texorpdfstring{$1<p\le 2$}{1<p<=2}}
\label{sec:strangeSpM}

Assume that $M=(M_k)$ satisfies Hypothesis \ref{ssec:localDC} and that $L=(L_k)$ satisfies $L_k\ge 1$ for all $k$.
In this section we investigate extensions of the $\CM$-regular Lie groups $\Diff W^{[M],p}(\mathbb R)$ and $\Diff\BM(\mathbb R)$
(as well as of the regular Lie groups $\Diff \Sp(\R)$ and $\Diff \cB(\R)$)
which are similar to those from Section \ref{sec:HS}. Surprisingly, these extensions lead to 
{\it half-Lie groups} only, namely $\CM$ (or smooth) manifolds which are topological groups with 
$\CM$ (or smooth) right translations -- but left translations and inversions are only continuous.
Nevertheless, the $R$-transform is $\CM$ (or smooth) and similar results as in Section \ref{sec:HS} 
hold.

\subsection{Extending the function spaces $\SpM$ for $1<p<\infty$ and $\BM$}
\label{SpM.1}
For $\cA\in\{\SpM \text{ for } 1<p<\infty, \BM\}$ we consider the space 
$\cA_2:=\{f\in C^\infty(\mathbb R): f'\in\cA(\mathbb R)\}$ with the structure given by the 
linear isomorphism $f\mapsto (f',f(x_0))\in \cA(\mathbb R)\x \mathbb R$, for any fixed 
$x_0\in \mathbb R$. Since $\cA(\mathbb R)\not\subseteq L^1(\mathbb R)$, functions in 
$\cA_2(\mathbb R)$ are no longer bounded. Thus evaluations at $-\infty$ and at $\infty$ no longer 
make sense. 

Analogously, we consider the cases $\cA\in\{\Sp \text{ for } 1<p<\infty, \cB\}$.
Let
\begin{align*}
  \Diff \SpM_2(\mathbb R) &:= \big\{\on{Id}+ f: f\in \SpM_2(\mathbb R), f'>-1\big\}, \quad\text{ for } 1<p < \infty, \\
  \Diff \Sp_2(\mathbb R) &:= \big\{\on{Id}+ f: f\in \Sp_2(\mathbb R), f'>-1\big\}, \quad\quad\text{ for } 1<p < \infty, \\
  \Diff \BM_2(\mathbb R) &:= \big\{\on{Id}+ f: f\in \BM_2(\mathbb R), \inf_{x \in \R} f'(x)>-1\big\}, \\
  \Diff \cB_2(\mathbb R) &:= \big\{\on{Id}+ f: f\in \cB_2(\mathbb R), \inf_{x \in \R} f'(x)>-1\big\}.
\end{align*}

\begin{theorem}\label{SpM.2}
Then $\Diff \SpM_2(\mathbb R)$
for $1<p<\infty$ and 
$\Diff \BM_2(\mathbb R)$
are only $\CM$ half-Lie groups, and $\Diff \Sp_2(\mathbb R)$
for $1<p<\infty$ and 
$\Diff \cB_2(\mathbb R)$
are only half-Lie groups. 
\end{theorem}

\begin{proof}
This will follow from Proposition \ref{strangeSpM.1} and Theorem 
\ref{strangeSpM.3}. 
\end{proof}

\subsection{The homogeneous $\dot H^1$ metric on $\mathbb R\backslash\Diff\SpM_2(\mathbb R)$ for $1<p\le 2$}\label{SpM.4a}
For $h,k\in \SpM_2(\mathbb R)$, considered as elements of the Lie algebra, we consider the 
right invariant symmetric positive semi-definite bilinear form 
$$
G_\ph(h,k) = G_{\on{Id}}(h\o \ph\i,k\o\ph\i) = \int_{\mathbb R} (h\o\ph\i)'.(k\o\ph\i)'\,dx = 
\int\frac{h'.k'}{\ph'}\,dx\;.
$$
It is visibly $\CM$. 
Its kernel is exactly the tangent bundle of all left translation cosets
$\{\ps+\al:\al\in \mathbb R\}$, so it induces a weak Riemannian metric on the 
homogeneous space $\mathbb R\backslash\Diff\SpM_2(\mathbb R)$ which is a $\CM$ manifold 
diffeomorphic to $\{f\in \SpM_2(\mathbb R): f'>-1, f(0)=0\}$. Elements of $\Diff\SpM_2(\mathbb R)$ 
still act as $\CM$ diffeomorphisms from the right on this homogeneous space.
The corresponding geodesic equation is the Hunter--Saxton equation in the form
\begin{equation*}
\boxed{\quad
\ph_{tt} = \frac12\int_{x_0}^x\frac{\ph_{tx}^2}{\ph_x}\,dz \quad\text{  for any }x_0\in\mathbb R.
\quad}
\end{equation*}
This follows from the variation with fixed ends
\begin{align*}
\p_s E(\ph) &=\p_s\frac12 \int_0^1G_\ph(\ph_t,\ph_t)\,dt
=\p_s\frac12 \int_0^1\int_{\mathbb R}\frac{\ph_{tx}^2}{\ph_x}\,dx\,dt
\\
&= 
\int_0^1\int_{\mathbb R}\Big(-\ph_{tt}+\frac12\int_{x_0}^x\frac{\ph_{tx}^2}{\ph_x}\,dz\Big)_x.\frac{\ph_{sx}}{\ph_x}\,dx\,dt\;.
\end{align*}
If the  composition mapping were differentiable, passing to the right logarithmic derivative  
$\ph\mapsto u = (\ph_t)\circ\ph\i$ would equivalently render this geodesic equation to the Lie algebra valued 
equation 
$$
u_{t} = -u u_x +\frac12 \int_{x_0}^x   (u_x(z))^2 \,dz\quad\text{  for any }x_0\in\mathbb R,
$$ 
which in turn is equivalent to the usual form
$u_{tx} = -u u_{xx}-\frac12 u^2_x$.
But note that $u.u_x$ is not in $\SpM_2(\mathbb R)$, in general.

\subsection{The $R$-transform: $\mathbb R\backslash \Diff\SpM_2(\mathbb R) \to \SpM(\mathbb R)$}
\label{SpM.4}
For $1<p\le 2$ we consider the $R$-transform
$$ R:\left\{
\begin{aligned}
 \mathbb R\backslash\Diff\SpM_2(\mathbb R)&\to 
\SpM\big(\R,\mathbb R_{>-2}\big)\subset\SpM(\R,\R)\\
\ph &\mapsto
2\;\big((\ph')^{1/2}-1\big)\; ,
\end{aligned}\right.
$$
which is defined on the space of left cosets $\{\ps+\al:\al\in \mathbb R\}$. 
For any fixed $x_0\in \mathbb R$ 
its inverse is given by
$$R\i :\left\{
\begin{aligned}\SpM\big(\R,\mathbb R_{>-2}\big) &\to \mathbb R\backslash\Diff\SpM_2(\mathbb R)
\\\ga&\mapsto (x\mapsto x+\frac14 \int_{x_0}^x \big(\ga^2(y)+4\ga(y)\big)\;dy\;),
\end{aligned}\right.$$
modulo translations acting from the left. 

Again, the $R$-transform is an isometry between the weak Riemannian manifold 
$\mathbb R\backslash\Diff\SpM_2(\mathbb R)$ with the homogeneous $\dot H^1$-metric, and the flat 
open subset $\SpM(\mathbb R,\mathbb R_{>-2})$ in the pre-Hilbert space $\SpM(\mathbb R,\mathbb R)$ 
with the constant $L^2$ inner product. All results of Section \ref{sec:HS} continue to hold with 
the appropriate changes; but not the formula for the shift in \ref{HS.10}.

\subsection{The case $\cA=\SpM\cap L^1$ for $1<p<\infty$ also leads to a $\CM$ half-Lie group}
\label{ssec:SpMnegative}
The construction in \ref{SpM.1} leads to a half-Lie group because the antiderivatives of elements in the spaces 
$\cA_2(\mathbb R)$ are unbounded. But even if we force them to be bounded, we only get half-Lie 
groups as we show now.
In the case $p>1$ we have $\SpM(\mathbb R)\not\subseteq L^1(\mathbb R)$, in general. 
So we might consider the space 
$(\SpM\cap L^1)_2(\mathbb R)
:=\big\{f\in C^\infty(\mathbb R): f'\in \SpM(\mathbb R)\cap L^1(\mathbb R)\big\}$ 
of bounded antiderivatives of functions in $\SpM(\mathbb R)\cap L^1$; we use the convenient vector space 
structure induced by 
the embedding $f\mapsto (f',f',\ev_{-\oo})$
into $\SpM(\mathbb R)\x L^1(\mathbb R)\x \RR$. 
But the corresponding group $\Diff (W^{[M],p}\cap L^1)_2(\R)$ is only a $\CM$ half-Lie group if $p>1$.
Even the group $\Diff (W^{\infty,p}\cap L^1)_2(\R)$, 
where 
$(\Sp\cap L^1)_2(\mathbb R)
:=\big\{f\in C^\infty(\mathbb R): f'\in \Sp(\mathbb R)\cap L^1(\mathbb R)\big\}$, 
is only a half-Lie group if $p>1$.
See Proposition \ref{strangeSpM.1} and Theorem \ref{strangeSpM.3} below.

\begin{proposition}\label{strangeSpM.1}
Left translation is not even Gateaux differentiable on $\Diff (W^{\infty,p}\cap L^1)_2(\R)$,
$\Diff \Sp_2(\R)$, and $\Diff \cB_2(\R)$,
or on $\Diff (W^{[M],p}\cap L^1)_2(\R)$, $\Diff \SpM_2(\R)$, and $\Diff \BM_2(\R)$,  
for non-quasianalytic $M=(M_k)$ 
if $p>1$.
In fact, the derivative takes values in the larger space with weaker structure
$$
\SpM_2=\{f\in C^\infty(\mathbb R): f'\in \SpM(\mathbb R)\} \cong \SpM(\mathbb R)\x \mathbb R.
$$
\end{proposition}

\begin{proof}
  Let $t \mapsto \Id + f(t,~)$ and $t \mapsto \Id +g(t,~)$ be $C^\infty$-curves in 
  $\Diff (W^{\infty,p} \cap L^1)_2(\R)$. For $f$ this means
  precisely that $f\in C^\infty(\R^2)$ such that $\p_x f>-1$ and, 
  for all $k \in \N$, $\al\in \N_{>0}$, and $K \in \sK(\R)$,
  \begin{align}\label{eq:counterp1}
  \begin{split}
    &\sup_{t \in K} \|\p_t^k \p_x^\al f(t,~)\|_{L^p(\R)} <\infty,   \\
    &\sup_{t \in K} \|\p_t^k \p_x f(t,~)\|_{L^1(\R)} <\infty, \\
    &\sup_{t \in K} |\p_t^k  f(t,-\infty)| <\infty,
  \end{split}
  \end{align} 
  and likewise for $g$. In view of (\ref{prop:Bgroup}.\ref{eq:group1}) we must investigate $h(t,x) := f(t,x+g(t,x))$. 
  We shall find functions $f$ and $g$ satisfying \eqref{eq:counterp1} such that 
  \begin{align} \label{eq:counterp2}
  \begin{split}
    \p_t \p_x h(t,x) &= (\p_x^2f)(t,x+g(t,x)) \p_t g(t,x) 
    \\
    &\quad+ (\p_x^2f)(t,x+g(t,x)) \p_t g(t,x) \p_x g(t,x) \\
    &\quad+ 
    (\p_t \p_x f)(t,x+g(t,x)) 
    \\
    &\quad+ (\p_t \p_x f)(t,x+g(t,x))\p_x g(t,x) 
    \\
    &\quad+ (\p_x f)(t,x+g(t,x)) (\p_t \p_x g)(t,x).
  \end{split}
  \end{align}
  is not $1$-integrable in $x$ for any $t$; this will imply the result. 
  Let us define 
  \[
    f(t,x) := \int_0^x \ph(y)\, dy \quad \text{ and }\quad g(t,x) := t \int_0^x \ps(y)\, dy
  \]
  where $\ph$ is the function from Lemma \ref{strangeSpM.2} below and $\ps$ is any function in $W^{\infty,p}(\R)\cap L^1(\R)$ with $\ps>-1$
  and 
  $\int_0^x \ps(y)\, dy \to C_\pm \ne 0$ as $x \to \pm\infty$.
  Then $f$ and $g$ satisfy \eqref{eq:counterp1}. 
  However, the first term on the right-hand side of \eqref{eq:counterp2}, which equals
  \[
    \ph'(x+g(t,x)) \int_0^x \ps(y)\, dy,
  \]
  is not $1$-integrable in $x$ for any $t$, by the properties of $\ph$ and $\ps$. 
  All other terms on the right-hand side of \eqref{eq:counterp2} are $1$-integrable 
  and tend to $0$ as $x \to \pm \infty$, since $f$ and $g$ satisfy \eqref{eq:counterp1}. 
  It follows that $\p_t \p_x h(t,x)$ is not $1$-integrable in $x$ for any $t$.

  By allowing $p=\infty$ we may treat $\Diff \Sp_2(\mathbb R)$ and $\Diff \cB_2(\mathbb R)$ simultaneously.
  That $t \mapsto \Id +g(t,~)$ is a $C^\infty$-curve  in $\Diff \Sp_2(\mathbb R)$ means that 
  $g \in C^\infty(\R^2)$ such that $\inf_{x \in \R} \p_x g(t,x)>-1$ and, 
  for some $x_0 \in \R$ and all $k \in \N$, $\al\in \N_{>0}$, and $K \in \sK(\R)$,
  \begin{align}\label{eq:counterpu1}
  \begin{split}
    &\sup_{t \in K} \|\p_t^k \p_x^\al g(t,~)\|_{L^p(\R)} <\infty,   \\
    &\sup_{t \in K} |\p_t^k  g(t,x_0)| <\infty.
  \end{split}
  \end{align} 
  Let $\ch \ne 0$ be any nonnegative $C^\infty$-function on $\R$ with support in $[-1,1]$, and 
  set $\ph_a(x) := \sum_{n=1}^\infty a_n \ch(x-2n)$, where $a_n >0$ and $(a_n) \in \ell^p$. Then $\ph_a \in \Sp(\R)$ and 
  $\th_a(x) := \int_0^x \ph_a(y)\, dy \in \Sp_2(\R)$ is increasing and satisfies 
  \[
    \th_a(2n+1) = \int_\R \ch(y)\, dy \sum_{k=1}^n a_k.
  \]
  Define $g(t,x) := t \th_{(n^{-1})}(x)$, $f(x) := \th_b(x)$, and $h(t,x):=f(x+g(t,x))$, where 
  \[
    b_n := 
    \begin{cases}
      k^{-1} & \text{ if } n = \lceil e^k \rceil\\
      0 & \text{ otherwise }
    \end{cases}.
  \]
  Then 
  \[
    h_{tx}(0,x) = \ph'_b(x) \th_{(n^{-1})}(x)  +  \ph_b(x) \ph_{(n^{-1})}(x).  
  \]
  For $p<\infty$ the first term on the right-hand side is not in $L^p$, since
  \begin{align*}
    \int_{\R} |\ph'_b(x) \th_{(n^{-1})}(x)|^p\; dx  
    &= \sum_{n=1}^\infty \int_{2n-1}^{2n+1} b_n^p |\ch'(x-2n)|^p (\th_{(n^{-1})}(x))^p \; dx \\
    &\ge  \|\ch'\|_{L^p(\R)}^p \sum_{n=1}^\infty (b_n \th_{(n^{-1})}(2n-1))^p    
  \end{align*}
  and since $\sum_{k=1}^n k^{-1} \ge \log (n+1)$ and hence 
  \[
    b_n \th_{(n^{-1})}(2n-1) \ge \|\ch\|_{L^1(\R)} b_n \log (n) \ge 
    \|\ch\|_{L^1(\R)} \cdot 
    \begin{cases}
      1 & \text{ if } n = \lceil e^k \rceil\\
      0 & \text{ otherwise }
    \end{cases}. 
  \]
  As the second term clearly is in $L^p$, we conclude that $x \mapsto h_{tx}(0,x)$ is not in $L^p$. 
  For $p=\infty$ choose $b_n= (1+\sqrt{\log(n)})^{-1}$.

  This proof can be adapted to $\Diff (W^{[M],p}\cap L^1)_2(\R)$, $\Diff \SpM_2(\R)$, and $\Diff \BM_2(\R)$
  for non-quasianalytic $M=(M_k)$, by choosing a 
  nonnegative $\CM$-function $\ch$ with support in $[-1,1]$ (above and in the proof of Lemma \ref{strangeSpM.2}). 
\end{proof}

\begin{lemma}\label{strangeSpM.2}
There exists a function $\ph \in W^{\infty,p}(\mathbb R)\cap L^1(\mathbb R)$ 
with $\ph'\notin L^1(\mathbb R)$ and $\ph\ge 0$, for each $p>1$.
\end{lemma}

\begin{proof}
Let $0\ne \ch\in C^\oo(\RR)$ be nonnegative with $\on{supp}(\ch)\subs[-1,1]$
and define $\ph:t\mapsto \sum_{n=1}^\oo \frac1n \ch\big(\log(n)(t-2n)\big)$.
Then $\ph \in C^\oo(\R)$ as locally finite sum
with
\[
        \ph^{(k)}(t)=\sum_n \frac{\log(n)^k}n \ch^{(k)}\big(\log(n)(t-2n)\big),
\]
and hence 
\begin{align*}
        \|\ph^{(k)}\|_{L^p(\R)}^p
        &=\sum_n \Bigl(\frac{\log(n)^k}n\Bigr)^p \int_\RR |\ch^{(k)}\big(\log(n)(t-2n)\big)|^p\,dt \\
        &=\sum_n \Bigl(\frac{\log(n)^k}n\Bigr)^p \frac1{\log(n)}\,\|\ch^{(k)}\|_{L^p(\R)}^p \\
        &=\|\ch^{(k)}\|_{L^p(\R)}^p\sum_n \frac{\log(n)^{kp-1}}{n^p}
\end{align*}
which is finite for $k=0$ and $p\geq 1$, is infinite for $k=1$ and $p=1$,
and is finite for $k\geq 1$ and $p>1$.
\end{proof}

\begin{theorem}\label{strangeSpM.3}
  Right translation is $C^\infty$ on $\Diff (W^{\infty,p}\cap L^1)_2(\R)$, $\Diff \Sp_2(\R)$, and $\Diff \cB_2(\R)$ and $\CM$ on  
  $\Diff (W^{[M],p}\cap L^1)_2(\R)$, $\Diff \SpM_2(\R)$, and $\Diff \BM_2(\R)$ for all $p\ge 1$; 
  so these two groups are only half-Lie groups.
\end{theorem}

\begin{proof}
  Let us first consider $\Diff (W^{\infty,p} \cap L^1)_2(\R)$. 
  Consider a $C^\infty$-curve $t \mapsto \Id + f(t,~)$ in 
  $\Diff (W^{\infty,p} \cap L^1)_2(\R)$, i.e., $f \in C^\infty(\R^2)$ satisfies (\ref{strangeSpM.1}.\ref{eq:counterp1}) and $\p_x f>-1$. 
  Let $g \in (W^{\infty,p} \cap L^1)_2(\R)$ with $g' > -1$. 
  In view of (\ref{prop:Bgroup}.\ref{eq:group1}) we must show that $h(t,x) := f(t,x+g(x))$ satisfies 
  (\ref{strangeSpM.1}.\ref{eq:counterp1}). 
  The first condition in (\ref{strangeSpM.1}.\ref{eq:counterp1}) can be check as in 
  \cite{MichorMumford13} (there is only a shift by 1 in the  
  order of differentiation). 
  For the second condition consider 
  \begin{align} \label{eq:counter4}
        \p_t^k\p_x h(t,x) = (\p_t^k \p_x f)(t,x+g(x)) (1+g'(x)),  
  \end{align}
  and thus $\sup_{t \in K} \|\p_t^k \p_x h(t,~)\|_{L^1(\R)} <\infty$ for all $k \in \N$ and $K \in \sK(\R)$, 
  since $f$ satisfies (\ref{strangeSpM.1}.\ref{eq:counterp1}) and since $x \mapsto x +g(x)$ is a diffeomorphism.
  The third condition follows from 
  \begin{equation} \label{eq:counter5}
    \p_t^k h(t,-\infty) = \lim_{x \to -\infty} \p_t^k h(t,x) = \lim_{x \to -\infty} \p_t^k f(t,x+g(x)) = \p_t^k f(t,-\infty).
  \end{equation}

  Now let us turn to $\Diff (W^{[M],p}\cap L^1)_2(\R)$. 
  Let $U$ be open in a Banach space $E$ 
  and let $f \in \CM(U \times \R)$ satisfy $\p_x f>-1$ and 
  \begin{align} \label{eq:counterp3}
  \begin{split}
    \forall K \in \sK(U) ~ \Box \rh>0 : &
      \sup_{\substack{k \in \N,\al \in \N_{>0}\\ u \in K\\  \|v_j\|_E\le 1}}  
      \frac{\big(\int_{\R} |\p_u^k\p_x^{\al} f(u,x)(v_1,\dots,v_k)|^p\, dx\big)^{1/p}}{\rh^{k+\al}\, 
      (k+\al)!\, M_{k+\al}}< \infty,
      \\ &
      \sup_{\substack{k \in \N, u \in K\\  \|v_j\|_E\le 1}}  
      \frac{\int_{\R} |\p_u^k\p_x f(u,x)(v_1,\dots,v_k)|\, dx}{\rh^{k+1}\, (k+1)!\, M_{k+1}}< \infty,
      \\&
      \sup_{\substack{k \in \N, u \in K\\  \|v_j\|_E\le 1}}  
      \frac{ |\p_u^k f(u,-\infty)(v_1,\dots,v_k)|}{\rh^{k}\, k!\, M_{k}}< \infty.
  \end{split}
  \end{align}
  Let $g \in (W^{[M],p}\cap L^1)_2(\R)$ with $g' > -1$.
  We must show that $h(u,x) := f(u,x+g(x))$ satisfies \eqref{eq:counterp3}. 
  Again the first condition in \eqref{eq:counterp3} follows along the lines of the proof of Theorem \ref{thm:compDiff}. 
  As before we may infer from \eqref{eq:counter4} and \eqref{eq:counter5} that the second and the third condition in 
  \eqref{eq:counterp3} are transmitted from $f$ to $h$. 

  The other cases can be treated analogously; in the Fa\`a di Bruno formula for $\p_u^k \p_x^\al h$ there appear only 
  derivatives of $f$ and $g$ of order $\ge 1$ and all these are globally bounded in $x$. 
\end{proof}

\subsection{Half-Lie groups associated to $\SpM\cap L^1$ for 
$1<p\le 2$, and their $R$-transforms}
\label{ssec:SpMgroups}
We consider the convenient vector  space 
$(\SpM\cap L^1)_2(\mathbb R)
:=\big\{f\in C^\infty(\mathbb R): f'\in \SpM(\mathbb R)\cap L^1(\mathbb R)\big\}$ 
of bounded antiderivatives of functions in $\SpM(\mathbb R)\cap L^1$ as described in 
\ref{ssec:SpMnegative}. Note that the evaluations at $-\infty$ and at $\infty$ again make sense.
We consider the exact and splitting 
sequence of convenient vector spaces 
\begin{equation*}
\xymatrix{
(\SpM\cap L^1)_0(\mathbb R)\quad \ar@{^(->}[r] & 
 (\SpM\cap L^1)_2(\mathbb R) 
\ar@{->>}[rrr]^{(\on{ev}_{-\infty},\on{ev}_\infty)} &&& \mathbb R^2
}
\end{equation*}
where $(\SpM\cap L^1)_0(\mathbb R)$ is just the kernel of $(\on{ev}_{-\infty},\on{ev}_{\infty})$.

The space $(\SpM\cap L^1)_0(\mathbb R)$ differs from $\SpM(\mathbb R)\cap L^1$ if $p>1$; 
Lemma \ref{strangeSpM.2} can easily be adapted to show this.
We also consider the space 
$(\SpM\cap L^1)_1(\mathbb R)=\ker(\ev_{-\infty})$ of functions in 
$(\SpM\cap L^1)_2(\R)$ with $f(-\infty)=0$.
In the sequence of injections 
$$
(\SpM\cap L^1)_0(\mathbb R)\hookrightarrow(\SpM\cap L^1)_1(\mathbb R)
\hookrightarrow(\SpM\cap L^1)_2(\mathbb R)\rightarrowtail \BM(\mathbb R)
$$
each space is an ideal in $\BM(\mathbb R)$; this can be checked easily.

We consider corresponding groups:  
\begin{multline*}
\Diff(\SpM\cap L^1)_0(\mathbb R)\hookrightarrow\Diff(\SpM\cap L^1)_1(\mathbb R)
\hookrightarrow
\\
\hookrightarrow
\Diff(\SpM\cap L^1)_2(\mathbb R)\rightarrowtail \Diff\BM(\mathbb R)
\end{multline*}
Each of these is a normal subgroup in each other in which it is contained.
Only 
$\Diff\BM(\mathbb R)$ are $C^{[M]}$ Lie groups. The other groups 
$\Diff(\SpM\cap L^1)_0(\mathbb R)$, $\Diff(\SpM\cap L^1)_1(\mathbb R)$, and 
$\Diff(\SpM\cap L^1)_2(\mathbb R)$ are only $C^{[M]}$ half-Lie groups.

\subsection{Some easy observations on the half-Lie groups in this paper}
\label{strangeSpM.4} 
Not every tangent vector can be extended to a left
left invariant vector field on the whole group, but they can be extended
to right invariant vector fields, which are only continuous and not differentiable in general.
The same holds for right invariant Riemannian metrics. The tangent space at the identity is not a 
Lie algebra, since $[X,Y]=X'Y-XY'$ is not in the modelling space any more, in general.
The behavior of Sobolev completions of diffeomorphisms groups seems to be the same.

But the right invariant homogeneous $\dot H^1$ metric is $\CM$, even when applied to two right 
invariant vector fields.
Even geodesics exists and are $\CM$.
This is compatible with
Lemma \ref{strangeSpM.5} below.

\begin{lemma}\label{strangeSpM.5}
The $R$-transform, given by 
$$ R:\left\{
\begin{aligned}
 \Diff(\SpM\cap L^1)_1(\mathbb R)&\to 
\SpM\big(\R,\mathbb R_{>-2}\big)\cap L^1(\R,\R_{>-2})\\
\ph &\mapsto
2\;\big((\ph')^{1/2}-1\big)\; ,
\end{aligned}\right.
$$
is $C^{[M]}$.
\end{lemma}

This also holds (for  $C^\infty$ instead of $\CM$) for the half-Lie group 
$\Diff (W^{\infty,p}\cap L^1)_1(\mathbb R)$. 

\begin{proof}
The arguments given in \ref{ssec:Rtransform} imply that $R$ maps $\Diff(\SpM\cap L^1)_1(\mathbb R)$ to 
$\SpM\big(\R,\mathbb R_{>-2}\big)\cap L^1(\R,\R_{>-2})$. 
To see that $R$ is $\CM$ let $U$ be open in a Banach space $E$ and let $f \in \CM(U \times \R)$ satisfy $\p_x f>-1$ and 
(\ref{strangeSpM.3}.\ref{eq:counterp3}). 
We must check that $g:= R(\ph) = R(\Id +f) = \p_x f + F(\p_x f) \p_x f \in \CM(U \times \R)$ 
satisfies 
\begin{align} \label{eq:RtransformWM2} 
  \begin{split}
    \forall K \in \sK(U) ~ \Box \rh>0 : &
      \sup_{\substack{k \in \N,\al \in \N\\ u \in K\\  \|v_j\|_E\le 1}}  
      \frac{\big(\int_{\R} |\p_u^k\p_x^{\al} g(u,x)(v_1,\dots,v_k)|^p\, dx\big)^{1/p}}{\rh^{k+\al}\, 
      (k+\al)!\, M_{k+\al}}< \infty,
      \\ &
      \sup_{\substack{k \in \N, u \in K\\  \|v_j\|_E\le 1}}  
      \frac{\int_{\R} |\p_u^k g(u,x)(v_1,\dots,v_k)|\, dx}{\rh^{k}\, k!\, M_{k}}< \infty.
  \end{split}
  \end{align}  
We may conclude by Proposition \ref{prop:incl} and by Theorem \ref{thm:Bcomp} (as in \ref{ssec:Rtransform}) 
that $h:= F(\p_x f)$ satisfies 
\[
  \forall K \in \sK(U) ~ \Box \rh>0 : 
      \sup_{\substack{k \in \N,\al \in \N\\ u \in K, x\in \R\\  \|v_j\|_E\le 1}}  
      \frac{|\p_u^k\p_x^{\al} h(u,x)(v_1,\dots,v_k)|}{\rh^{k+\al}\, 
      (k+\al)!\, M_{k+\al}}< \infty.
\]
Then $g = \p_x f + h\, \p_x f$ satisfies \eqref{eq:RtransformWM2}, since so does $\p_x f$ and hence also $h\, \p_x f$.
\end{proof}

Thus all results about the $R$-transform from Section \ref{sec:HS} hold also for the half-Lie 
groups $\Diff(\SpM\cap L^1)_1(\mathbb R)$ for $1<p\le 2$.


\begin{thebibliography}{10}

\bibitem{BBMM14}
M.~Bauer, M.~Bruveris, S.~Marsland, and P.~W. Michor.
\newblock Constructing reparametrization invariant metrics on spaces of plane
  curves.
\newblock {\em Differential Geometry and its Applications}, 2014.
\newblock \texttt{doi:10.1016/j.difgeo.2014.04.008}.

\bibitem{BBM14b}
M.~Bauer, M.~Bruveris, and P.~W. Michor.
\newblock The homogeneous {S}obolev metric of order one on diffeomorphism
  groups on the real line.
\newblock {\em Journal of Nonlinear Science}, 2014.
\newblock \texttt{doi:10.1007/s00332-014-9204-y}.

\bibitem{BBM14a}
M.~Bauer, M.~Bruveris, and P.~W. Michor.
\newblock ${R}$-transforms for {S}obolev ${H^2}$-metrics on spaces of plane
  curves.
\newblock {\em Geometry, Imaging and Computing}, 1:1--56, 2014.

\bibitem{BM04}
E.~Bierstone and P.~D. Milman.
\newblock Resolution of singularities in {D}enjoy-{C}arleman classes.
\newblock {\em Selecta Math. (N.S.)}, 10(1):1--28, 2004.

\bibitem{BMT90}
R.~W. Braun, R.~Meise, and B.~A. Taylor.
\newblock Ultradifferentiable functions and {F}ourier analysis.
\newblock {\em Results Math.}, 17(3-4):206--237, 1990.

\bibitem{Glockner05}
H.~Gl\"ockner.
\newblock $\on{Diff}(\mathbb R^n)$ as a Milnor-Lie group. 
\newblock {\em Math. Nachr.} 278(9):1025Ð1032, 2005.
  
\bibitem{Komatsu73}
H.~Komatsu.
\newblock Ultradistributions. {I}. {S}tructure theorems and a characterization.
\newblock {\em J. Fac. Sci. Univ. Tokyo Sect. IA Math.}, 20:25--105, 1973.

\bibitem{Komatsu79b}
H.~Komatsu.
\newblock An analogue of the {C}auchy-{K}owalevsky theorem for
  ultradifferentiable functions and a division theorem for ultradistributions
  as its dual.
\newblock {\em J. Fac. Sci. Univ. Tokyo Sect. IA Math.}, 26(2):239--254, 1979.

\bibitem{Komatsu80}
H.~Komatsu.
\newblock Ultradifferentiability of solutions of ordinary differential
  equations.
\newblock {\em Proc. Japan Acad. Ser. A Math. Sci.}, 56(4):137--142, 1980.

\bibitem{KrieglMichor90}
A.~Kriegl and P.~W. Michor.
\newblock The convenient setting for real analytic mappings.
\newblock {\em Acta Math.}, 165(1-2):105--159, 1990.

\bibitem{KM97}
A.~Kriegl and P.~W. Michor.
\newblock {\em The convenient setting of global analysis}, volume~53 of {\em
  Mathematical Surveys and Monographs}.
\newblock American Mathematical Society, Providence, RI, 1997.

\bibitem{KM97r}
A.~Kriegl and P.~W. Michor.
\newblock Regular infinite-dimensional {L}ie groups.
\newblock {\em J. Lie Theory}, 7(1):61--99, 1997.

\bibitem{KMRc}
A.~Kriegl, P.~W. Michor, and A.~Rainer.
\newblock The convenient setting for non-quasianalytic {D}enjoy--{C}arleman
  differentiable mappings.
\newblock {\em J. Funct. Anal.}, 256:3510--3544, 2009.

\bibitem{KMRu}
A.~Kriegl, P.~W. Michor, and A.~Rainer.
\newblock The convenient setting for {D}enjoy--{C}arleman differentiable
  mappings of {B}eurling and {R}oumieu type.
\newblock \texttt{arXiv:1111.1819}, 2011.

\bibitem{KMRq}
A.~Kriegl, P.~W. Michor, and A.~Rainer.
\newblock The convenient setting for quasianalytic {D}enjoy--{C}arleman
  differentiable mappings.
\newblock {\em J. Funct. Anal.}, 261:1799--1834, 2011.

\bibitem{KMR14b}
A.~Kriegl, P.~W. Michor, and A.~Rainer.
\newblock The exponential law for exotic spaces of test functions and diffeomorphism groups.
\newblock 2014, in preparation.

\bibitem{MichorH}
P.~W. Michor.
\newblock {\em Topics in differential geometry}, volume~93 of {\em Graduate
  Studies in Mathematics}.
\newblock American Mathematical Society, Providence, RI, 2008.

\bibitem{MichorMumford13}
P.~W. Michor and D.~Mumford.
\newblock A zoo of diffeomorphism groups on {$\mathbb{R}^n$}.
\newblock {\em Ann. Global Anal. Geom.}, 44(4):529--540, 2013.

\bibitem{Neus78}
H.~Neus.
\newblock \"{U}ber die {R}egularit\"atsbegriffe induktiver lokalkonvexer
  {S}equenzen.
\newblock {\em Manuscripta Math.}, 25(2):135--145, 1978.

\bibitem{Petzsche88}
H.-J. Petzsche.
\newblock On {E}. {B}orel's theorem.
\newblock {\em Math. Ann.}, 282(2):299--313, 1988.

\bibitem{RainerSchindl12}
A.~Rainer and G.~Schindl.
\newblock Composition in ultradifferentiable classes.
\newblock to appear in {\em Studia Math.} (2015)
\newblock arXiv:1210.5102.

\bibitem{RainerSchindl14}
A.~Rainer and G.~Schindl.
\newblock Equivalence of stability properties for ultradifferentiable function classes.
\newblock 2014.
\newblock arXiv:1407.6673.

\bibitem{Retakh70}
V.~Retakh.
\newblock {Subspaces of a countable inductive limit.}
\newblock {\em Sov. Math., Dokl.}, 11:1384--1386, 1970.

\bibitem{Schindl14}
G.~Schindl.
\newblock {\em Exponential laws for classes of {D}enjoy--{C}arleman
  differentiable mappings}.
\newblock PhD thesis, 2014.

\bibitem{Schwartz66}
L.~Schwartz.
\newblock {\em Th{\'e}orie des distributions}.
\newblock Publications de l'Institut de Math{\'e}matique de l'Universit{\'e} de
  Strasbourg, No. IX-X. Nouvelle {\'e}dition, enti{\'e}rement corrig{\'e}e,
  refondue et augment{\'e}e. Hermann, Paris, 1966.

\bibitem{Walter12}
B.~Walter.
\newblock Weighted diffeomorphism groups of {B}anach spaces and weighted
  mapping groups.
\newblock {\em Dissertationes Math. (Rozprawy Mat.)}, 484:128, 2012.

\bibitem{Yamanaka89}
T.~Yamanaka.
\newblock Inverse map theorem in the ultra-{$F$}-differentiable class.
\newblock {\em Proc. Japan Acad. Ser. A Math. Sci.}, 65(7):199--202, 1989.

\bibitem{Yamanaka91}
T.~Yamanaka.
\newblock On {ODE}s in the ultradifferentiable class.
\newblock {\em Nonlinear Anal.}, 17(7):599--611, 1991.

\end{thebibliography}

\def\cprime{$'$}

\end{document}